\documentclass{amsart} 

\usepackage[utf8]{inputenc} 
\usepackage[T1]{fontenc} 
\usepackage{textcomp} 
\usepackage[english,frenchb]{babel}
\usepackage[autolanguage]{numprint}
\usepackage{hyperref} 
\usepackage{graphicx} 
\usepackage{verbatim} 
\usepackage[all,cmtip]{xy}
\usepackage{amsmath,amssymb,amsthm}
\usepackage{appendix} 
\usepackage[nobysame]{amsrefs}
\usepackage{enumitem}
\usepackage{stmaryrd}

\usepackage{amsmath,amssymb,amsthm} 
\usepackage{mathtools}

\newcommand\A{\mathbb{A}}
\newcommand\B{\mathbb{B}}

\newcommand\E{\mathbb{E}}

\newcommand\N{\mathbb{N}}

\newcommand\Q{\mathbb{Q}}
\newcommand\R{\mathbb{R}} 

\newcommand\Z{\mathbb{Z}}

\newcommand\FC{\mathcal{F}}

\newcommand\JC{\mathcal{J}}

\newcommand\OC{\mathcal{O}}

\newcommand\UF{\mathfrak{U}}

\newcommand\XF{\mathfrak{X}}

\newcommand\ab\allowbreak 

\newcommand{\AAinf}{\operatorname{\mathbb{A}_{inf}}}
\newcommand{\AAC}{\operatorname{\mathbb{A}_{cris}}}
\newcommand{\Bdrp}{\operatorname{\mathbb{B}^+_{dR}}}

\newcommand{\Bdr}{\operatorname{\mathbb{B}_{dR}}}
\newcommand{\Bst}{\operatorname{\mathbb{B}_{st}}}

\newcommand{\Spec}{\operatorname{Spec}}
\newcommand{\Spf}{\operatorname{Spf}}

\newcommand{\holim}{\operatorname{holim}}
\newcommand{\Kos}{\operatorname{Kos}}
\newcommand{\Syn}{\operatorname{Syn}}

\newcommand{\syn}{\operatorname{syn}}
\newcommand{\Gal}{\operatorname{Gal}}
\newcommand{\Lie}{\operatorname{Lie}}

\newcommand{\im}{\operatorname{Im}}
\newcommand{\Spa}{\operatorname{Spa}}

\newcommand{\dR}{\operatorname{dR}}
\newcommand{\HK}{\operatorname{HK}}

\newcommand{\cris}{\operatorname{cris}}

\newcommand{\et}{\operatorname{\text{\'et}}}
\newcommand{\proet}{\operatorname{\text{pro\'et}}}

\newcommand{\hocolim}{\operatorname{hocolim}}
\newcommand{\Sp}{\operatorname{Sp}}

\newcommand{\tr}{\operatorname{tr}}

\newcommand{\FM}{\operatorname{FM}}

\theoremstyle{definition} 
\newtheorem{Def}{D\'efinition}[section]
\theoremstyle{plain} 
\newtheorem{Pro}[Def]{Proposition} 
\newtheorem{Lem}[Def]{Lemme} 
\newtheorem{The}[Def]{Th\'eor\`eme}
\newtheorem{Cor}[Def]{Corollaire}
 
\theoremstyle{remark}
\newtheorem{Rem}[Def]{Remarque}

\title{Morphismes de p\'eriodes et cohomologie syntomique}
\author{Sally Gilles}
\date{}

\begin{document}

\maketitle

\begin{abstract}
On commence par donner la version g\'eom\'etrique d'un r\'esultat de Colmez et Nizio{\l} \'etablissant un th\'eor\`eme de comparaison entre
les cycles proches $p$-adiques arithm\'etiques et la cohomologie des faisceaux syntomiques.
La construction locale de cet isomorphisme utilise la th\'eorie des $(\varphi, \Gamma)$-modules et s'obtient en
r\'eduisant l'isomorphisme de p\'eriode \`a un th\'eor\`eme de comparaison entre des
cohomologies d'alg\`ebres de Lie. En appliquant ensuite la m\'ethode des "coordonn\'ees plus g\'en\'erales" utilis\'ee par Bhatt-Morrow-Scholze, on construit un isomorphisme global.  
On peut notamment d\'eduire de ce th\'eor\`eme la conjecture semi-stable de
Fontaine-Jannsen. Ce r\'esultat a
\'egalement \'et\'e prouv\'e par (entre autres) Tsuji, via l'application de Fontaine-Messing, et
par \v{C}esnavi\v{c}ius et Koshikawa, qui g\'en\'eralisent la preuve de la conjecture cristalline de
Bhatt, Morrow et Scholze. On utilise l'application
construite pr\'ec\'edemment pour montrer que les morphismes de p\'eriode de Tsuji et de
\v{C}esnavi\v{c}ius-Koshikawa sont \'egaux.
\end{abstract}

\tableofcontents

\section{Introduction}

Les morphismes de p\'eriode permettent de d\'ecrire la cohomologie \'etale $p$-adique de vari\'et\'es alg\'ebriques \`a l'aide de formes diff\'erentielles et rendent ainsi possible son calcul. Plusieurs constructions de tels morphismes existent. Dans \cite{Tsu99}, par exemple, Tsuji \'etablit que l'application de Fontaine-Messing (d\'ecrite dans \cite{FM87},  \cite{Kat94}) d\'efinit un quasi-isomorphisme, en degr\'es born\'es, entre les cycles proches $p$-adiques et la cohomologie \'etale $p$-adique de la vari\'et\'e rigide associ\'ee \`a un sch\'ema \`a r\'eduction semi-stable . Tsuji d\'eduit de ce r\'esultat une preuve de la conjecture semi-stable de Fontaine-Jannsen qui relie la cohomologie $p$-adique de la vari\'et\'e \`a sa cohomologie de Hyodo-Kato (telle qu'elle est d\'efinie dans \cite{Bei2013}). Une autre preuve de la conjecture semi-stable, g\'en\'eralisant la preuve de la conjecture cristalline de Bhatt-Morrow-Scholze (voir \cite{BMS16}), a \'et\'e donn\'ee par \v{C}esnavi\v{c}ius-Koshikawa dans \cite{CK17}. Ce travail comporte deux parties : dans un premier temps, on donne une preuve diff\'erente du r\'esultat de Tsuji en construisant un nouveau morphisme de p\'eriode (qui est la version g\'eom\'etrique de celui de Colmez-Nizio{\l} dans \cite{CN2017}). On utilise ensuite cette construction pour montrer que l'application de Fontaine-Messing et le morphisme de Bhatt-Morrow-Scholze sont \'egaux.

\subsection{\'Enonc\'e des principaux th\'eor\`emes} 

D\'ecrivons plus pr\'ecis\'ement les r\'esultats obtenus. On fixe $p$ un nombre premier. Soit $\OC_K$ un anneau de valuation discr\`ete complet, de corps de fraction $K$ de caract\'eristique $0$ et de corps r\'esiduel parfait $k$ de caract\'eristique $p$. On note $\varpi$ une uniformisante de $K$, $\OC_F:= W(k)$ l'anneau des vecteurs de Witt associ\'e \`a $k$ et $F= \text{Frac}(\OC_F)$ (avec $e:= [K : F]$). Soit $\overline{K}$ la cl\^oture alg\'ebrique de $K$ et $\OC_{\overline{K}}$ la cl\^oture int\'egrale de $\OC_K$ dans $\overline{K}$. Enfin, on note $C$ le corps complet alg\'ebriquement clos $\widehat{\overline{K}}$ et $\OC_C$ son anneau de valuation.

On note $\OC_K^{\times}$ le sch\'ema formel $\Spf(\OC_K)$ muni de la log-structure donn\'ee par son point ferm\'e et $\OC_F^0$ le sch\'ema formel $\Spf(W(k))$ muni de la log-structure $(\N \to \OC_F, 1 \mapsto 0)$. On consid\`ere $\XF$ un log-sch\'ema formel \`a r\'eduction semi-stable, log-lisse sur $\OC_K^{\times}$. On note $Y:=\XF_k$ sa fibre sp\'eciale, $\overline{\XF}:= \XF_{\OC_{C}}$ et $\overline{Y}:= \XF_{\overline{k}}$.   Soit $\mathcal{S}_n(r)_{\XF}$ dans $D(Y_{\text{\'et}}, \Z/p^n \Z)$ \footnote{O\`u $D(Y_{\text{\'et}},A)$ est la cat\'egorie d\'eriv\'ee des faisceaux de $A$-modules sur le site \'etale de $Y$, pour $A$ un anneau.}   et $\mathcal{S}_n(r)_{\overline{\XF}}$ dans $D((\overline{Y})_{\text{\'et}}, \Z/p^n \Z)$ les faisceaux syntomiques arithm\'etique et g\'eom\'etrique. On note $\XF_{K, \tr}$ (respectivement $\XF_{C, \tr}$) le lieu de $\XF_K$ (respectivement de $\XF_C$) o\`u la structure logarithmique est triviale, et $i$ et $j$ (respectivement $\overline{i}$ et $\overline{j}$) les morphismes \footnote{Les morphismes $j$ et $\overline{j}$ n'existent pas au niveau des espaces, mais on dispose bien de morphismes de topo\"i $Rj_*$ et $R\overline{j}_*$.}    $Y \hookrightarrow \XF$ et $\XF_{K,\rm{tr}}\dashrightarrow \XF$ (respectivement $\overline{Y} \hookrightarrow \overline{\XF}$ et $\XF_{C,\rm{tr}} \dashrightarrow \overline{\XF}$). Enfin, soit $\Z_p(r)':= \frac{1}{a(r)! p^{a(r)}}\Z_p(r)$, $r=a(r)(p-1)+b(r)$ avec $0 \le b(r) \le p-2$.

On a le th\'eor\`eme suivant :  

\begin{The}
\label{Main the1}
Pour tout $0 \le k \le r$, il existe un $p^{N}$-isomorphisme \footnote{On appelle $p^{N}$-isomorphisme, un morphisme dont le noyau et le conoyau sont tu\'es par $p^N$.} 
\[ \alpha^{0}_{r,n} : \mathcal{H}^k(\mathcal{S}_n(r)_{\overline{\XF}}) \to \overline{i}^{*}R^k \overline{j}_{*} \Z/p^n(r)'_{\XF_{C,\tr}} \]
o\`u $N$ est un entier qui d\'epend de $p$ et de $r$, mais pas de $\XF$ ni de $n$. 
\end{The}

Notons 
$ \alpha^{\rm FM}_{r,n} : \mathcal{S}_n(r)_{\XF} \to i^{*}Rj_{*} \Z/p^n(r)'_{\XF_{K,\rm{tr}}} $ (respectivement $\mathcal{S}_n(r)_{\overline{\XF}} \to \overline{i}^{*}R\overline{j}_{*} \Z/p^n(r)'_{\XF_{C,\rm{tr}}}$) l'application de Fontaine-Messing arithm\'etique (respectivement g\'eom\'etrique). Dans \cite[Th. 0.4]{Tsu99}, Tsuji montre le th\'eor\`eme \eqref{Main the1} dans le cas o\`u $\XF$ est un sch\'ema \`a r\'eduction semi-stable sur $\Spec(\OC_K)$ en prouvant que l'application de Fontaine-Messing fournit le $p^N$-isomorphisme voulu. Pour cela, il commence par se ramener au cas $n=1$ par d\'evissage. En multipliant le faisceau syntomique (respectivement le faisceau de cycle proche) par $t$ avec $t= \log([\varepsilon])$, o\`u $\varepsilon \in \OC_C^{\flat}$ est un syst\`eme de racine $p$-i\`eme de l'unit\'e (respectivement par $\zeta_p$), il se ram\`ene au cas $k=r$. Il utilise ensuite une description des faisceaux $\mathcal{H}^r(\mathcal{S}_n(r)_{\overline{\XF}})$ via des applications symboles qu'il compare avec celle de $\overline{i}^{*}R^k \overline{j}_{*} \Z/p^n(r)'_{\XF_{C,{\rm tr}}}$ donn\'ee par Bloch-Kato (\cite{BK86}) dans le cas de la bonne r\'eduction, puis \'etendue \`a la r\'eduction semi-stable par Hyodo (\cite{Hyo88}).

Dans \cite{CN2017}, Colmez-Nizio{\l} prouvent la version arithm\'etique du th\'eor\`eme ci-dessus c'est-\`a-dire : 

\begin{The}\cite[Th. 1.1]{CN2017}
 \label{qis arithmetique1}
Pour tout $0 \le k \le r$, l'application 
\[ \alpha^{\FM}_{r,n} : \mathcal{H}^k(\mathcal{S}_n(r)_{\XF}) \to i^{*}R^k j_{*} \Z/p^n(r)'_{\XF_{K,\rm{tr}}} \]
est un $p^{N}$-isomorphisme pour une constante $N$ qui d\'epend de $e$, $p$ et de $r$, mais pas de $\XF$ ni de $n$. 
\end{The}

Pour cela, ils construisent un $p^{N}$-quasi-isomorphisme local qu'ils comparent ensuite \`a la d\'efinition locale de l'application de Fontaine-Messing. 
Le th\'eor\`eme \ref{Main the1} peut ensuite se d\'eduire de \ref{qis arithmetique1} en passant \`a la limite sur les extensions finies $L$ de $K$. Une nouvelle preuve du r\'esultat de Colmez-Nizio{\l} a r\'ecemment \'et\'e donn\'ee dans \cite{AMMN20} dans le cas o\`u $\XF$ est propre et lisse sur $\OC_K$. 

On prouve ici le th\'eor\`eme \ref{Main the1} directement. On commence par montrer un r\'esultat local semblable \`a celui de Colmez-Nizio{\l}. On consid\`ere $R$ la compl\'etion $p$-adique d'une alg\`ebre \'etale sur $R_{\square} := \OC_{C} \{ X_1, \dots, X_d, \frac{1}{X_1  \cdots X_a}, \frac{ \varpi}{X_{a+1} \cdots X_{a+b}}\}$ (pour $a$, $b$, $c$ et $d$ des entiers tels que $a+b+c=d$). On note $D=\{ X_{a+b+1} \cdots X_d=0 \}$ le diviseur \`a l'infini et $\overline{R} [ \frac{1}{p} ]$ l'extension maximale de $R[\frac{1}{p}]$ non ramifi\'ee en dehors de $D$. Soit $G_R= \Gal(\overline{R}[\frac{1}{p}]/R[\frac{1}{p}])$
\footnote{On utilise ici la notation $\Gal$ pour d\'esigner le groupe des automorphismes de $\overline{R}[\frac{1}{p}]$ qui fixent $R[\frac{1}{p}]$.}.
On rel\`eve ensuite $R$ en une alg\`ebre $R_{{\rm inf}}^+$ sur $R_{\text{inf}, \square}^{+}:= \A_{\text{inf}} \{ X, \frac{1}{X_1  \cdots X_a}, \frac{ [\varpi^{\flat}]}{X_{a+1} \cdots X_{a+b}} \} $ (o\`u $\varpi^{\flat} \in \OC_{C}^{\flat}$ est un syst\`eme de racines $p$-i\`emes de $\varpi$), on d\'efinit\footnote{Ici le produit tensoriel est compl\'et\'e pour la topologie $p$-adique.} $R_{{\rm cris}}^+:=R_{{\rm inf}}^+ \widehat{\otimes}_{\A_{{\rm inf}}} \AAC$. Ici $\Spf(R)$ est muni de la log-structure donn\'ee par la fibre sp\'eciale et le diviseur \`a l'infini et $\Spf(\AAC)$, $\Spf(\AAinf)$ sont munis des log-structures usuelles d\'ecrites dans \cite[\textsection1.6]{CK17} (voir la section 2.2.2 ci-dessous). On s'int\'eresse au complexe syntomique donn\'e par la fibre homotopique : 
\[ \Syn(R_{{\rm cris}}^+,r):=[F^r \Omega_{R_{{\rm cris}}^+}^{\bullet} \xrightarrow{p^r- \varphi} \Omega_{R_{{\rm cris}}^+}^{\bullet}] \]
o\`u $\varphi$ est le Frobenius sur $R_{{\rm cris}}^+$ qui \'etend celui de $R_{{\rm inf}, \square}^{+}$ donn\'e par le Frobenius usuel sur $\A_{{\rm inf}}$ et par $\varphi(X_i)=X_i^p$. 
Le complexe $\Syn(R_{{\rm cris}}^+,r)$ calcule alors la cohomologie syntomique de $\Spf(R)$. On note $\Syn(R_{{\rm cris}}^+,r)_n$ sa r\'eduction modulo $p^n$.
On a le r\'esultat local suivant : 

 \begin{The}
 \label{syn/galois1}
 Il existe des $p^{N}$-quasi-isomorphismes : 
\[ \alpha^0_r : \tau_{\le r} \Syn(R_{{\rm cris}}^+,r) \xrightarrow{\sim} \tau_{\le r} R \Gamma( G_R, \Z_p(r)) \]
\[ \alpha^0_{r,n} : \tau_{\le r} \Syn(R_{{\rm cris}}^+,r)_n \xrightarrow{\sim} \tau_{\le r} R \Gamma( G_R, \Z/p^n(r)) \]
o\`u $N$ est une constante qui ne d\'epend que de $r$. 
\end{The}

Les applications ci-dessus sont notamment utilis\'ees par Colmez-Nizio{\l} dans \cite{CN4} pour donner une interpr\'etation "g\'eom\'etrique" des morphismes de p\'eriodes, \emph{i.e.} voir les th\'eor\`emes de comparaison entre les cohomologies \'etale et syntomique comme des th\'eor\`emes de comparaison entre les $C$-points d'espaces de Banach-Colmez. Cela leur permet ensuite d'\'etendre le th\'eor\`eme de comparaison semi-stable \`a des vari\'et\'es rigides analytiques plus g\'en\'erales dans \cite{CN5}.    

Si $\XF$ est propre sur $\OC_K$, on d\'eduit de \ref{Main the1} que $\alpha^{0}_{r}$ induit un quasi-isomorphisme Frobenius et $G_K$-\'equivariant, qui pr\'eserve la filtration apr\`es tensorisation par $\B_{\rm dR}$ : 
\begin{align*}
\widetilde{\alpha}^0 :  & H^i_{\text{\'et}}( \XF_{C}, \Q_p) \otimes_{\Q_p} \B_{\text{st}} \cong H^i_{\rm{HK}}(\XF) \otimes_{F} \B_{\text{st}} \\
 & H^i_{\text{\'et}}( \XF_{C}, \Q_p) \otimes_{\Q_p} \B_{\mathrm{dR}} \cong H^i_{\rm dR}( \XF_K) \otimes_{K} \B_{\mathrm{dR}}
\end{align*}
o\`u $R \Gamma_{\rm{HK}}(\XF) \cong R \Gamma_{\rm{cris}}(\XF_k/W(k)^0)_{\Q}$. 
Par les r\'esultats de Tsuji et de Colmez-Nizio{\l}, le morphisme $\alpha^{\rm FM}_{r}$ induit, lui aussi, un isomorphisme $\widetilde{\alpha}^{\rm{FM}}$ entre les m\^emes objets. Enfin, \v{C}esnavi\v{c}ius-Koshikawa dans \cite{CK17}, donnent une troisi\`eme construction d'un tel isomorphisme $\widetilde{\alpha}^{\rm CK}$ (qui \'etend la construction de Bhatt-Morrow-Scholze dans \cite{BMS16}). Dans la seconde partie de cet article, on montre : 

\begin{The}
\label{thmcomp1}
Les morphismes $\widetilde{\alpha}^{\rm FM}$ et $\widetilde{\alpha}^{0}$ d'une part et $\widetilde{\alpha}^{\rm CK}$ et $\widetilde{\alpha}^{0}$  d'autre part sont \'egaux. En particulier, $\widetilde{\alpha}^{\rm FM}=\widetilde{\alpha}^{\rm CK}$ .
\end{The}

Ce dernier th\'eor\`eme \'etend les r\'esultats de Nizio{\l} dans \cite{Niziol09} et \cite{Niziol18} dans lesquels elle prouve l'\'egalit\'e entre les morphismes de p\'eriode de \cite{Fal02}, \cite{Bei2013} et \cite{Niz08}. La preuve de Nizio{\l} fait intervenir la $h$-topologie et sa m\'ethode ne s'applique pas au morphisme $\widetilde{\alpha}^{\rm CK}$, puisque dans leur d\'efinition, \v{C}esnavi\v{c}ius-Koshikawa n'autorisent pas de diviseur horizontal. L'unicit\'e est ici obtenue en comparant "\`a la main" les constructions locales des morphismes de p\'eriode. Cela est possible car, localement, $\widetilde{\alpha}^{\rm CK}$ est construit en utilisant des complexes de Koszul similaires \`a ceux qui interviennent dans la construction de $\widetilde{\alpha}^{0}$. La principale diff\'erence est que la d\'efinition de \cite{CK17} utilise le foncteur de d\'ecalage $L\eta_{\mu}$ (\emph{i.e.} les termes des complexes sont multipli\'es par des puissances de $\mu=[\varepsilon]-1$) alors que dans le cas de $\widetilde{\alpha}^{0}$ les complexes sont multipli\'es par des puissances de $t=\log([\varepsilon])$. 

La d\'efinition de Fontaine-Messing et celle de Bhatt-Morrow-Scholze ont chacune leurs propres avantages. Le morphisme de p\'eriode construit par Bhatt-Morrow-Scholze et \v{C}esnavi\v{c}ius-Koshikawa est, par exemple, plus adapt\'e pour travailler au niveau int\'egral. Ce r\'esultat d'unicit\'e permet ainsi d'obtenir pour l'application de Bhatt-Morrow-Scholze, les propri\'et\'es d\'ej\`a connues pour l'application de Fontaine-Messing (notamment, la compatibilit\'e aux applications symboles).

\subsection{Preuve du th\'eor\`eme \ref{syn/galois1}}

Nous expliquons dans un premier temps comment construire le $p^{N}$-quasi-isomorphisme local. On part du complexe $\Syn(R_{{\rm cris}}^+,r)$. On commence par "changer la convergence" des \'el\'ements de $R_{\rm cris}^+$. Pour $0 < u <v$, on d\'efinit l'anneau $\A^{[u]}$ (respectivement $\A^{[u,v]}$) comme la compl\'etion $p$-adique de $\A_{{\rm inf}}[ \frac{[\beta]}{p}]$ (respectivement de $\A_{{\rm inf}}[\frac{p}{[\alpha]}, \frac{[\beta]}{p}]$) pour $\beta$ un \'el\'ement de $\OC_C^{\flat}$ de valuation $\frac{1}{u}$ (respectivement $\alpha$ et $\beta$ des \'el\'ements de $\OC_C^{\flat}$ de valuation $\frac{1}{v}$ et $\frac{1}{u}$). Pour des valeurs de $u$ et de $v$ convenables, on d\'eduit de la suite exacte 
\begin{equation}
\label{SEAcris}
 0 \to \Z_p t^{\{ r\}} \to F^r\AAC \xrightarrow{1 - \frac{\varphi}{p^r}} \AAC \to 0
 \end{equation}
o\`u $t^{\{r\}}:= \frac{t^r}{a(r)! p^{a(r)}}$, des suites $p^{6r}$-exactes : 
\begin{equation}
\label{SEAuAuv}
\begin{aligned}
& 0 \to \Z_p(r) \to F^r \A^{[u]} \xrightarrow{p^r - \varphi} \A^{[u]} \to 0 \\
& 0 \to \Z_p(r) \to F^r \A^{[u,v]} \xrightarrow{p^r - \varphi} \A^{[u,v]} \to 0.
\end{aligned}
 \end{equation}
On d\'efinit\footnote{Les produits tensoriels sont compl\'et\'es pour la topologie $p$-adique.}  $R^{[u]}:= R_{{\rm inf}}^+ \widehat{\otimes}_{\A_{{\rm inf}}} \A^{[u]}$ et $R^{[u,v]}:= R_{{\rm inf}}^+ \widehat{\otimes}_{\A_{{\rm inf}}} \A^{[u,v]}$ et pour $S \in \{ R^{[u]}, R^{[u,v]} \}$, on pose $C(S,r):=[F^r \Omega_{S}^{\bullet} \xrightarrow{p^r- \varphi} \Omega_{S}^{\bullet}]$. Les suites exactes \eqref{SEAuAuv} ci-dessus permettent alors de montrer les $p^{N_1r}$-quasi-isomorphismes \footnote{pour un $N_1 \in \N$.} :
\[ \Syn(R_{{\rm cris}}^+,r) \cong C(R^{[u]},r) \cong C(R^{[u,v]},r). \]
En plongeant $R^{[u]}$ et $R^{[u,v]}$ dans des anneaux de p\'eriodes $\A_{\overline{R}}^{[u]}$ et $\A_{\overline{R}}^{[u,v]}$ (construits de mani\`ere similaire \`a $\A^{[u]}$ et $\A^{[u,v]}$ en utilisant $\overline{R}$ au lieu de $\OC_C$), on les munit d'une action du groupe de Galois $G_R$. On note $\A_{R}^{[u]}$ et $\A_{R}^{[u,v]}$ leurs images respectives par ce plongement. On se place ensuite dans une base de $\Omega^1_R$ pour \'ecrire le complexe $C(R^{[u,v]},r)$ sous la forme d'un complexe de Koszul $\Kos(\varphi, \partial, F^r \A_R^{[u,v]})$. La multiplication par $t= \log([\varepsilon]) \in \AAC$ (o\`u $\varepsilon \in \OC_C^{\flat}$ est un syst\`eme de racine $p$-i\`eme de l'unit\'e) permet de se d\'ebarasser de la filtration en transformant l'action des diff\'erentielles en une action de l'alg\`ebre de Lie du groupe $\Gamma_R = \Gal(R_{\infty}[\frac{1}{p}]/R[\frac{1}{p}])\cong \Z_p^d$ avec $R_{\infty}:= (\varinjlim_m R_m)^{\wedge_p}$ et $$R_m:=\Big(\OC_C \{ X^{\frac{1}{p^m}}, \frac{1}{(X_1  \cdots X_a)^{\frac{1}{p^m}}}, \frac{ \varpi^{\frac{1}{p^m}}}{(X_{a+1} \cdots X_{a+b})^{\frac{1}{p^m}}} \}\Big) \widehat{\otimes}_{\OC_C} R.$$ C'est \`a cette \'etape qu'il est n\'ecessaire de tronquer les complexes en degr\'e $r$ pour garder des isomorphismes. Finalement, on a la suite de $p^{N_2r}$-quasi-isomorphismes \footnote{pour un $N_2 \in \N$.} suivante : 
\begin{equation}
\label{qisdiffLieGamma}
\tau_{\le r} \Kos(\varphi, \partial, F^r\A_R^{[u,v]}) \xrightarrow{\sim}  \tau_{\le r}\Kos( \varphi, \Lie \Gamma_R, \A_R^{[u,v]}(r) ) \xleftarrow{\sim} \tau_{\le r}\Kos(\varphi, \Gamma_R, \A_R^{[u,v]}(r)) 
\end{equation}
Soit $R_{{\rm inf}}$ la compl\'etion $p$-adique de $R_{{\rm inf}}^+ [ \frac{1}{[\varpi^{\flat}]}]$ et $\A_R$ son image dans $\A_{\rm inf}(\overline{R}^{\flat})$. On montre ensuite qu'on a un quasi-isomorphisme $\Kos(\varphi, \Gamma_R, \A_{R}^{[u,v]}) \cong \Kos(\varphi, \Gamma_R, \A_{R})$. Pour cela, on passe par l'interm\'ediaire d'un anneau $\A_R^{(0,v]+}$ (construit \`a partir de l'anneau $\A^{(0,v]+}:= \{ x=\sum_{n \in \N} [x_n] p^n \in \A_{\rm inf} \; | \; x_n \in \OC_C^{\flat}, \; v_{\OC_C^{\flat}}(x_n) + \frac{n}{v} \ge 0 \text{ et } v_{\OC_C^{\flat}}(x_n) + \frac{n}{v} \to + \infty \text{ quand } n \to \infty \}$) et on utilise $\psi$, un inverse \`a gauche du Frobenius $\varphi$. Enfin, des arguments de descente presque \'etale et de d\'ecompl\'etion permettent d'obtenir un quasi-isomorphisme : 
\[ \Kos(\varphi, \Gamma_R, \A_{R}) \xrightarrow{\sim} \Kos(\varphi, \Gamma_R, \A_{R_{\infty}}) \xrightarrow{\sim} \Kos(\varphi, G_R, \A_{\overline{R}})\]
et la suite exacte $0 \to \Z_p \to \A_{\overline{R}} \xrightarrow{1- \varphi} \A_{\overline{R}} \to 0$ montre que le complexe $\Kos(\varphi, G_R, \A_{\overline{R}})$ calcule la cohomologie galoisienne $R\Gamma(G_R, \Z_p)$.
En r\'esum\'e le morphisme construit est le suivant : 
\begin{equation}
\label{qisconstruction}
\begin{split} R\Gamma_{\text{syn}}(R,r) \xrightarrow{\sim} \Syn(R_{\text{cris}}^{+}, r) \xrightarrow{\sim} C(R^{[u,v]}, r) \xrightarrow{\sim} \Kos( \varphi, \partial, F^r \A_R^{[u,v]}) \xrightarrow{\sim} \Kos( \varphi, \Gamma_R,\A_R^{[u,v]}(r))  \\ \xrightarrow{\sim} \Kos( \varphi, \Gamma_R,\A_R(r)) \xleftarrow{\sim} R \Gamma(G_R, \Z_p(r)).
\end{split}
\end{equation}

\begin{Rem}
La construction du morphisme local dans le cas arithm\'etique trait\'e dans \cite{CN2017} est similaire \`a celle d\'ecrite ci-dessus. Il existe cependant quelques diff\'erences que nous soulignons ici. Dans le cas o\`u $\Spf(R)$ est d\'efini sur $\OC_K$, on ne dispose plus du quasi-isomorphisme $R\Gamma_{{\rm cris}}(R/ \OC_F) \cong R\Gamma_{{\rm cris}}(R/ \AAC)$ qui nous permet de travailler avec l'anneau $R_{{\rm cris}}^+$. Au lieu de cela, Colmez et Nizio{\l} \'ecrivent $R$ comme le quotient d'une $\OC_F$-alg\`ebre log-lisse $R_{\varpi}^+$. Cela n\'ecessite l'introduction d'une variable suppl\'ementaire $X_0$ : on d\'efinit $R_{\varpi}^+$ comme le relev\'e \'etale de $R$ sur $\OC_{F} \{ X_0, X_1, \dots, X_d, \frac{1}{X_1  \cdots X_a}, \frac{ X_0}{X_{a+1} \cdots X_{a+b}}\}$. Le r\^ole de l'anneau $\A_{{\rm inf}}$ (respectivement $\A^{[u]}$, $\A^{[u,v]}$) est jou\'e par l'anneau $r_{\varpi}^+:= \OC_F \llbracket X_0 \rrbracket$ (respectivement par l'anneau $r_{\varpi}^{[u]}$ des fonctions analytiques sur $F$ qui convergent sur le disque $v_p(X_0) \ge \frac{u}{e}$, par l'anneau $r_{\varpi}^{[u,v]}$ des fonctions analytiques sur $F$ qui convergent sur la couronne $ \frac{v}{p} \ge v_p(X_0) \ge \frac{u}{e}$). On perd, dans le cas g\'eom\'etrique, cette interpr\'etation des anneaux en termes de s\'eries de Laurent. 

Travailler directement sur $C$ permet aussi de simplifier la preuve lors du passage au complexe de Koszul. Dans le cas arithm\'etique, il existe deux mani\`eres diff\'erentes de plonger $R_{\varpi}^{[u]}$, $R_{\varpi}^{[u,v]}$ dans $\A_{\overline{R}}^{[u]}$, $\A_{\overline{R}}^{[u,v]}$ : le plongement "de Kummer", qui envoie $X_0$ sur $[\varpi^{\flat}]$ et le plongement "cyclotomique" par lequel $X_0$ s'envoie sur un \'el\'ement $\pi_K \in \A_{{\rm inf}}(\overline{K}^{\flat})$ tel que $\bar{\pi}_K=(\varpi_{K_m})_m$, o\`u $(\varpi_{K_m})_m$ est une suite compatible d'uniformisantes des extensions cyclotomiques $K_m$ de $K$. Les quasi-isomorphismes \eqref{qisdiffLieGamma} sont obtenus en utilisant le plongement cyclotomique, mais le Frobenius qui intervient alors diff\`ere de celui apparaissant dans la d\'efinition du complexe syntomique (qui correspond, lui, au plongement de Kummer).

Une autre diff\'erence entre la preuve arithm\'etique et la preuve g\'eom\'etrique se trouve dans les propri\'et\'es de l'inverse \`a gauche $\psi$ du Frobenius. Dans \cite{CN2017}, par construction, il est topologiquement nilpotent sur le quotient $R_{\varpi}^{[u,v]}/R_{\varpi}^{[u]}$. Colmez et Nizio{\l} utilisent notamment cette propri\'et\'e pour montrer le quasi-isomorphisme $C(R_{\varpi}^{[u]},r) \cong C(R_{\varpi}^{[u,v]},r)$. Ce n'est plus le cas lorsqu'on travaille avec l'anneau $R^{[u,v]}$ mais en montrant que $\psi-1$ est surjective sur $\A^{[u,v]}/\A^{[u]}$, on obtient que quasi-isomorphisme voulu. 

Enfin, dans \cite{CN2017}, du fait de cette variable suppl\'ementaire $X_0$, l'alg\`ebre de Lie de $\Gamma_R$ n'est pas commutative. Cette difficult\'e n'apparait pas dans le cas g\'eom\'etrique.     
\end{Rem}

\subsection{Quasi-isomorphisme global}

Si $\XF$ est un sch\'ema formel \`a r\'eduction semi-stable et propre sur $\OC_K$, on montre que le morphisme local du th\'eor\`eme \ref{syn/galois1} se globalise en un $p^{N}$-quasi-isomorphisme : 
\[ \alpha^0_{r} : \tau_{\le r} R \Gamma_{ {\rm syn}}( \XF_{\OC_{C}},r) \to \tau_{\le r} R \Gamma_{\text{\'et}}(\XF_C, \Z_p(r)). \]
Pour faire cela, on applique la m\'ethode utilis\'ee par Bhatt-Morrow-Scholze dans \cite{BMS16} (et \v{C}esnavi\v{c}ius-Koshikawa dans \cite[\textsection 5]{CK17}). L'id\'ee est de construire pour $\XF= \Spf(R)$ assez petit, une version fonctorielle des complexes intervenant en \eqref{qisconstruction}. 

Pour $\XF$ un sch\'ema formel \`a r\'eduction semi-stable, il existe une base $(\Spf(R))$ de $\XF_{\OC_C,\text{\'et}}$ telle que pour chaque $R$ on ait des ensembles finis non vides $\Sigma$ et $\Lambda$ tels qu'il existe une immersion ferm\'ee   
$$
\Spf(R) \to \Spf(R_{\Sigma}^{\square}):= \Spf(\OC_C \{ X_{\sigma}^{\pm 1} \; | \;  \sigma \in \Sigma \})
$$
et pour chaque $\lambda$, une application \'etale 
$$
\Spf(R) \to \Spf(R_{\lambda}^{\square}):= \Spf(\OC_C \{ X_{\lambda, 1}, \dots , X_{\lambda, d}, \frac{1}{X_{\lambda, 1} \cdots X_{\lambda, a_{\lambda}}}, \frac{\varpi}{X_{\lambda,a_{\lambda}+1} \cdots X_{\lambda,d} }\}).
$$
On note $\Spf(R_{\Sigma, \Lambda}^{\square}):= \Spf(R_{\Sigma}^{\square}) \times \prod_{\lambda \in \Lambda} \Spf(R_{\lambda}^{\square})$. On construit alors une application 
\[ \alpha_{r, \Sigma, \Lambda} : \tau_{\le r} \Kos( \varphi, \partial, F^r R_{\Sigma, \Lambda}^{PD} ) \to \tau_{\le r} R \Gamma( G_{R}, \Z_p(r)) \]
de la m\^eme fa\c{c}on qu'en \eqref{qisconstruction} en rempla\c{c}ant $R_{\rm cris}^+$ par $R_{\Sigma, \Lambda}^{PD}$ d\'efini sur $\A_{\text{cris}}(R_{\Sigma, \Lambda}^{\square})$. On d\'eduit ensuite du th\'eor\`eme \eqref{syn/galois1} que $\alpha_{r, \Sigma, \Lambda}$ est un $p^{N}$-quasi-isomorphisme. On obtient des complexes fonctoriels en prenant la limite sur l'ensemble des donn\'ees $(\Sigma, \Lambda)$ comme ci-dessus. Un argument de descente cohomologique permet finalement d'obtenir le $p^{N}$-quasi-isomorphisme $\alpha^0_{r}$. 

\begin{Rem}
Dans \cite{CN2017}, Colmez et Nizio{\l} construisent le $p^{N}$-quasi-isomorphisme localement, puis comparent ce morphisme \`a l'application de Fontaine-Messing, ce qui permet de montrer que l'application de Fontaine-Messing globale est un $p^{N}$-quasi-isomorphisme. Le m\^eme argument que celui que nous utilisons permettrait de montrer que le morphisme arithm\'etique admet, lui aussi, une d\'efinition globale. Il suffit de d\'efinir des anneaux $R_{\varpi, \Sigma, \Lambda}^+$, $R_{\varpi, \Sigma, \Lambda}^{[u]}$, ... de la m\^eme fa\c{c}on que nous le faisons ici.  
\end{Rem}

\subsection{Preuve du th\'eor\`eme \ref{thmcomp1}} 

On commence par comparer les morphismes : \[ \widetilde{\alpha}^0, \widetilde{\alpha}^{\rm FM} : H^i_{\text{\'et}}( \XF_{C}, \Q_p) \otimes_{\Q_p} \B_{\text{st}} \cong H^i_{\text{HK}}( \XF) \otimes_{F} \B_{\text{st}}.\] Pour cela, il suffit de montrer l'\'egalit\'e des morphismes \[ \alpha^0_r, \alpha^{\rm FM}_r : \tau_{\le r} R \Gamma_{ {\rm syn}}( \XF_{\OC_{C}},r) \to \tau_{\le r} R \Gamma_{\text{\'et}}(\XF_C, \Z_p(r)).\] 

D\'efinissons l'anneau $\E_{\overline{R}}^{\rm PD}$ (respectivement $\E_{\overline{R}}^{[u,v]}$) comme la log-PD-enveloppe compl\'et\'ee de $R_{\rm cris}^+ \otimes_{\AAC} \AAC(\overline{R}) \to \AAC(\overline{R})$ (respectivement de $R^{[u,v]} \otimes_{\A^{[u,v]}} \A_{\overline{R}}^{[u,v]} \to \A_{\overline{R}}^{[u,v]}$). On note $C(G_R,M)$ le complexe des cochaines continues de $G_R$ \`a valeurs dans $M$. Localement, $\alpha^{\rm FM}_r$ est la compos\'ee des applications :
\begin{equation*}
\Kos(\partial, \varphi, F^r R_{\rm cris}^+) \to C(G_R, \Kos(\partial, \varphi, F^r \E_{\overline{R}}^{\rm PD}))   \xleftarrow{\sim} C(G_R, \Kos(\varphi, F^r \A_{\rm cris}(\overline{R}))) \xleftarrow{\sim} C(G_R,\Z_p(r))
\end{equation*}
o\`u la deuxi\`eme fl\`eche est un quasi-isomorphisme via le lemme de Poincar\'e : 
\[ F^r \A_{\rm cris}(\overline{R}) \xrightarrow{\sim} F^r \Omega_{ F^r \E_{\overline{R}}^{\rm PD}}^{\bullet} \]
et la troisi\`eme fl\`eche se d\'eduit de la suite exacte \eqref{SEAcris}.
L'application $\alpha^0_r$ est donn\'ee par la compos\'ee \eqref{qisconstruction}. On montre qu'il existe un diagramme commutatif : 
{\footnotesize \begin{equation}
\hspace{-1cm} \label{DiagComp1}
\xymatrix{
K_{\varphi, \partial}( F^r R_{\rm cris}^+) \ar[rd] \ar[rdd] \ar[r] & \ar[d] C_G(K_{\varphi, \partial} (F^r \E_{\overline{R}}^{\rm PD})) & \ar[d]  \ar[l] C_G(K_{\varphi}(F^r \A_{\rm cris}(\overline{R})) & \ar[d] \ar[dr] \ar[l] C_G(\Z_p(r)) & \\
& C_G(K_{\varphi, \partial}( F^r \E_{\overline{R}}^{[u,v]})) &  \ar[l] C_G(K_{\varphi} (F^r \A_{\overline{R}}^{[u,v]}) & C_G(K_{\varphi}( \A_{\overline{R}}^{(0,v]+}(r)) \ar[l] \ar[r] & C_G(K_{\varphi}( F^r \A_{\overline{R}})) \\
&K_{\varphi, \partial}(F^r \A_R^{[u,v]}) \ar[r]_{\tau_{\le r}}\ar[u] &\ar[u] K_{\varphi, \Gamma}(  \A^{[u,v]}_R(r)) & \ar[l] K_{\Gamma, \varphi}( \A^{(0,v]+}_R(r))) \ar[r] \ar[u] &  K_{\varphi, \Gamma}( \A_R(r))  \ar[u] 
}
\end{equation}}
o\`u, pour all\'eger, $C_G(\cdot)$ d\'esigne les complexes de cocha\^ines et $K_{\varphi, \partial}(\cdot)$ la fibre homotopique de l'application $1-\varphi$ sur les complexes de Koszul. On en d\'eduit l'\'egalit\'e locale. On prouve ensuite que le diagramme \eqref{DiagComp1} est toujours valable pour les anneaux construits \`a partir des coordonn\'ees $(\Sigma, \Lambda)$ et en passant \`a la limite, on obtient l'\'egalit\'e globale. 

Il reste \`a comparer $\widetilde{\alpha}^0$ et $\widetilde{\alpha}^{\rm CK}$. Pour les m\^emes raisons, il suffit de montrer l'\'egalit\'e des morphismes locaux. Celle-ci d\'ecoule de leurs constructions. Le morphisme $\widetilde{\alpha}^0$ s'obtient par la compos\'ee \footnote{Si $M$ est un module muni d'une action d'un Frobenius $\varphi$, on note $M\{r\}$ le module $M$ muni de l'action du Frobenius $p^r \varphi$. Voir la section \textsection Notations pour la d\'efinition de $- \widehat{\otimes}^L -$ utilis\'ee ici.} : 
\begin{equation}
\label{constructiontilde0}
\begin{split}
\tau_{\le r} R \Gamma_{\text{\'et}} (R[\frac{1}{p}], \Z_p)_{\Q} \widehat{\otimes}^L_{\Q_p} \Bst \xleftarrow[t^r]{\sim} \tau_{\le r} R \Gamma_{\text{\'et}} (R[\frac{1}{p}], \Z_p(r))_{\Q} \widehat{\otimes}^L_{\Q_p} \Bst \{ r \} \xleftarrow[\alpha^0_r]{\sim} \tau_{\le r} R\Gamma_{\rm syn}(R,r)_{\Q} \widehat{\otimes}_{\Q_p}^L \Bst \{ r \}  \\ \xrightarrow[{\rm can}]{\sim} \tau_{\le r} R\Gamma_{\rm cris}(R/ \AAC)_{\Q} \widehat{\otimes}_{\AAC[\frac{1}{p}]}^L \Bst \xleftarrow[\iota^{\rm B}_{{\rm cris}, \varpi}]{\sim} \tau_{\le r} R\Gamma_{\rm HK}(R) \widehat{\otimes}^L_{F} \Bst
\end{split}
\end{equation}
o\`u \[\iota_{{\rm cris}, \varpi}^{\rm B} : R\Gamma_{\rm cris}((R/p)/W(k)^0) \widehat{\otimes}^L_{W(k)} \B_{\rm st}^+ \xrightarrow{\sim} R\Gamma_{\rm cris}((R/p)/\AAC) \widehat{\otimes}^L_{\AAC} \B_{\rm st}^+\] est le quasi-isomorphisme (qui d\'epend du $\varpi$ choisi) de \cite[1.18.5]{Bei2013}. Rappelons maintenant comment $\widetilde{\alpha}^{\rm CK}$ est d\'efini. La preuve de \cite{CK17} passe par la construction d'un complexe $A\Omega_R \in D^+((\Spf(R))_{\text{\'et}}, \A_{\rm inf})$ qui v\'erifie (voir \cite[2.3]{CK17}) : 
\begin{equation}
\label{ainfet}
R \Gamma_{\text{\'et}}( R[\frac{1}{p}], \Z_p) \widehat{\otimes}^L_{\Z_p} (\A_{\mathrm{inf}} [ \frac{1}{\mu} ]) \xrightarrow{\sim} R\Gamma_{\text{\'et}}(R, A \Omega_{R}) \widehat{\otimes}^L_{\A_{\mathrm{inf}}} (\A_{\mathrm{inf}} [ \frac{1}{\mu}]).
\end{equation} 
o\`u $\mu = [\varepsilon]-1 \in \A_{\rm inf}$. On obtient ensuite $\widetilde{\alpha}^{\rm CK}$ en construisant un quasi-isomorphisme \[ \gamma_{\rm CK} : Ru_*(\OC_{(R/p)/\AAC}) \xrightarrow{\sim} A\Omega_R \widehat{\otimes}^L_{\A_{\rm inf}} \AAC\] o\`u $\OC_{(R/p)/\AAC} \in ((R/p)/\AAC)_{\rm log-cris}$ d\'esigne le faisceau structural sur le site cristallin et $u$ est la projection $((R/p)/\AAC)_{\rm log-cris} \to (\Spf(R))_{\text{\'et}}$. Pour cela, \v{C}esnavi\v{c}ius-Koshikawa commencent par montrer qu'on peut calculer $A\Omega_R$ comme $L\eta_{\mu} R\Gamma(\Gamma_R, \A_{\rm inf}(R_{\infty}))$ o\`u $\eta_{\mu}$ est le foncteur de d\'ecalage de Berthelot-Ogus (\cite{BMS16}). Pour d\'efinir $\gamma_{\rm CK}$, il reste \`a construire un quasi-isomorphisme entre $\Kos(\partial, R_{\rm inf}^+)$ et $\eta_{\mu} \Kos(\Gamma_R, \A_{\rm inf}(R_{\infty})) \widehat{\otimes}_{\A_{\rm inf}}^L \AAC$. Ce passage entre la cohomologie d'une alg\`ebre de Lie et la cohomologie de groupe est similaire \`a celui apparaissant dans la construction de $\alpha_r^0$, avec $\mu$ jouant le r\^ole de $t$ et c'est ce qui permet de d'obtenir un diagramme commutatif : 
\begin{equation}
\label{DiagComp2}
\xymatrix{
\tau_{\le r}R \Gamma_{\text{\'et}}(R, A \Omega_{R})_{\Q} \widehat{\otimes}^L_{\A_{\mathrm{inf}}[\frac{1}{p}]} \Bst & \ar[l]^-{\gamma^{\rm CK}}_-{\sim} \tau_{\le r}R \Gamma_{\mathrm{cris}}( R/ \AAC)_{\Q} \widehat{\otimes}^L_{\AAC[\frac{1}{p}]} \B_{\rm st} \\
\tau_{\le r}R\Gamma_{\text{\'et}}(R[\frac{1}{p}], \Z_p(r))_{\Q}  \widehat{\otimes}_{\Q_p}^L \Bst \{ r \}  \ar[u] & \tau_{\le r}R\Gamma_{\rm syn}(R,r)_{\Q} \widehat{\otimes}_{\Q_p}^L \Bst \{ r \} \ar[u] \ar[l]^{\alpha_r^0}_{\sim} 
 }
\end{equation} 
Finalement, $\widetilde{\alpha}^{\rm CK}$ est donn\'ee par la compos\'ee : 
\begin{equation}
\label{constructiontildeCK}
\begin{split}
 R\Gamma_{\text{\'et}}(R[\frac{1}{p}], \Z_p)_{\Q} \widehat{\otimes}_{\Q_p}^L \Bst  \xrightarrow[\sim]{\eqref{ainfet}} R \Gamma_{\text{\'et}}(R, A \Omega_{R})_{\Q} 
\widehat{\otimes}^L_{\A_{\mathrm{inf}}[\frac{1}{p}]} \Bst  \xleftarrow[\sim]{\gamma^{\rm CK}} R \Gamma_{\mathrm{cris}}(R/ \AAC)_{\Q} \widehat{\otimes}^L_{\AAC[\frac{1}{p}]} \B_{\rm st} 
\\ \xleftarrow[\sim]{\iota^{\rm B}_{{\rm cris}, \varpi}} R\Gamma_{\rm HK}(R) \widehat{\otimes}^L_{F} \B_{\rm st}
\end{split}
\end{equation}
et le diagramme \eqref{DiagComp2} permet d'obtenir l'\'egalit\'e entre \eqref{constructiontildeCK} et \eqref{constructiontilde0}.

\vspace{0.5cm}

\textbf{Remerciements.}
Cet article pr\'esente les r\'esultats obtenus durant ma th\`ese \`a l'\'ENS de Lyon. Je remercie ma directrice de th\`ese, Wies{\l}awa Nizio{\l}, pour m'avoir sugg\'er\'e ce sujet et pour les conseils et l'aide qu'elle m'a apport\'es tout au long de ce projet. Je remercie \'egalement le/la referee anonyme pour sa lecture attentive et ses nombreuses remarques et corrections. 

Ce projet a re\c{c}u des financements du Labex Milyon, dans le cadre du programme "Investissements d'Avenir" (ANR-11-IDEX-0007) et de l'European Research Council (ERC) dans le cadre du programme "European Union's Horizon 2020 research and innovation programme" (grant agreement No. 804176)..   

\vspace{0.5cm} 

\textbf{Notations et conventions.}
Pour $N$ un entier, on dit qu'une application $f : A \to B$ est $p^N$-\emph{injective} (respectivement $p^N$-\emph{surjective}) si son noyau (respectivement son conoyau) est tu\'e par $p^N$. Le morphisme $f$ est un $p^N$-\emph{isomorphisme} si il est \`a la fois $p^N$-injectif et $p^N$-surjectif. La suite \[ 0 \to A \xrightarrow{i} B \xrightarrow{f} C \xrightarrow{g} 0 \]
est $p^N$-\emph{exacte} si $i$ est $p^N$-injective, $\im(g)= p^N \ker(f)$ et $g$ est $p^N$-surjective. On d\'efinit de m\^eme un $p^N$-\emph{quasi-isomorphisme} comme \'etant une application $f: A^{\bullet} \to B^{\bullet}$ dans la cat\'egorie d\'eriv\'ee qui induit un $p^N$-isomorphisme sur la cohomologie.

Dans tout l'article, on utilise les r\'esultats de \cite{DLLZ19} sur les log-espaces adiques. On utilise en particulier que le foncteur usuel des sch\'emas formels localement noeth\'eriens dans les espaces adiques, s'\'etend en un foncteur pleinement fid\`ele des log-sch\'emas formels localement noeth\'eriens dans la cat\'egorie des log-espaces adiques localement noeth\'eriens (\cite[2.2.22]{DLLZ19}).  

Si $T$ est un topos et $A$ un anneau, on note $\mathcal{D}^+(T,A)$ (respectivement $\mathcal{D}^+(A)$) l'$\infty$-cat\'egorie d\'eriv\'ee des faisceaux de $A$-modules sur $T$ inf\'erieurement born\'es (respectivement l'$\infty$-cat\'egorie d\'eriv\'ee des $A$-modules). Les cat\'egories homotopiques associ\'ees sont les cat\'egories d\'eriv\'ees $D^+(T,A)$ et $D^+(A)$. Le foncteur d\'eriv\'e $R\Gamma(-): D^+(T,A) \to D^+(A)$ se rel\`eve en un foncteur $R\Gamma(-) : \mathcal{D}^+(T,A) \to \mathcal{D}^+(A)$. On note aussi $\mathbf{Fsc}(T, \mathcal{D}^+(A))$ la cat\'egorie des faisceaux de $\mathcal{D}^+(A)$ sur $T$ et on identifie les $\infty$-cat\'egories $\mathcal{D}^+(T,A)$ et $\mathbf{Fsc}(T, \mathcal{D}^+(A))$.

Si $f: K^{\bullet} \to L^{\bullet}$ est un morphisme de $\mathcal{D}^+(T,A)$, on note 
$[ K^{\bullet} \xrightarrow{f} L^{\bullet} ] := \holim ( K^{\bullet} \to L^{\bullet} \leftarrow 0 )$ la fibre d'homotopie de l'application $f$.

Pour $A \to B$ un morphisme d'anneaux et $K^{\bullet}$ et $L^{\bullet}$ des objets de $\mathcal{D}(A)$, on note 
\[ K^{\bullet} \widehat{\otimes}^L_A L^{\bullet}:= \holim_n ((K^{\bullet} \otimes^L_A L^{\bullet})\otimes^L_{\Z} \Z/p^n\Z) \]
\[ \text{ et } (K^{\bullet})_{\Q} \widehat{\otimes}^L_A B:= (\holim_n ((K^{\bullet} \otimes^L_A B)\otimes^L_{\Z} \Z/p^n\Z)) \otimes \Q_p. \] 
Plus g\'en\'eralement, si $J=(f_1, \dots, f_r)$ est un id\'eal finiment engendr\'e de $A$, on d\'efinit le produit tensoriel d\'eriv\'e compl\'et\'e pour la topologie $J$-adique: 
\[ K^{\bullet} \widehat{\otimes}^L_A L^{\bullet}:= \holim_n ((K^{\bullet} \otimes^L_A L^{\bullet})\otimes^L_{\Z} \Z[f_1, \dots, f_r]/(f_1^n, \dots, f_r^n)) \]
\[ \text{ et } (K^{\bullet})_{\Q} \widehat{\otimes}^L_A B:= (\holim_n ((K^{\bullet} \otimes^L_A B)\otimes^L_{\Z} \Z[f_1, \dots, f_r]/(f_1^n, \dots, f_r^n))) \otimes \Q_p. \] 

Si $G$ est un groupe qui agit sur un module $M$, on notera $R\Gamma(G,M)$ la cohomologie de groupe \emph{continue} associ\'ee. 

Si $\OC_K$ est un anneau de valuation discr\`ete en caract\'eristique mixte $(0,p)$ et $\varpi$ une uniformisante de $\OC_K$, on dira qu'un sch\'ema (respectivement sch\'ema formel) est \`a \emph{r\'eduction semi-stable} sur $\OC_K$ s'il s'\'ecrit localement pour la topologie \'etale $\Spec(R)$ (respectivement $\Spf(R)$) avec $R$ une alg\`ebre \'etale sur \[\OC_{K}[X_1, \dots, X_d, \frac{1}{X_1 \cdots X_a}, \frac{\varpi}{X_{a+1} \cdots X_{a+b}}] \] 
\[(\text{respectivement } \OC_{K} \{ X_1, \dots, X_d, \frac{1}{X_1 \cdots X_a}, \frac{\varpi}{X_{a+1} \cdots X_{a+b}}\}) \]avec $a,b$ et $d$ des entiers tels que $a+b \le d$. 
Enfin, si $A$ est un anneau de valuation discr\`ete, on notera $A^{\times}$ le log-sch\'ema formel $\Spf(A)$ muni de la log-structure induite par son point ferm\'e (appel\'ee log-structure standard de $A$) : c'est la log-structure donn\'ee par l'inclusion $A \setminus \{ 0 \} \hookrightarrow A$. On notera $A^0$ le log-sch\'ema formel $\Spf(A)$ muni de la log-structure induite par $(\N \to A, 1 \mapsto 0)$.  

\section{Cohomologie syntomique et r\'esultat local}

Dans cette section, on commence par donner les d\'efinitions des faisceaux syntomiques arithm\'etique $\mathcal{S}_n(r)_{\XF}$ et g\'eom\'etrique $\mathcal{S}_n(r)_{\overline{\XF}}$. On d\'ecrit ensuite la forme locale de ces complexes.  

\subsection{D\'efinition de la cohomologie syntomique}

On consid\`ere $\XF$ un sch\'ema formel \`a r\'eduction semi-stable sur $\OC_K$ i.e. localement $\XF$ s'\'ecrit $\Spf(R)$ avec $R$ la compl\'etion d'une alg\`ebre \'etale sur \[\OC_{K}\{X_1, \dots, X_d, \frac{1}{X_1 \cdots X_a}, \frac{\varpi}{X_{a+1} \cdots X_{a+b}}\} \] avec $a,b,c$ et $d$ des entiers tels que $a+b+c=d$. On munit $\XF$ de la log-structure donn\'ee par 
$ \mathcal{M}:= \{ g \in \OC_X \; | \; g \text{ inversible en dehors de } D \cup \XF_k \}$ o\`u $D$ est le diviseur horizontal et on a alors que $\XF$ est log-lisse sur $\OC_K^{\times}$. On note $\XF_{K, {\rm tr}}$ le lieu de $\XF_K$ o\`u la log-structure est triviale : si $\XF= \Spf(R)$ alors $\XF_{K, {\rm tr}}=\Sp(R_K)\setminus D_K$. On appelle $i$ (respectivement $j$) l'immersion $\XF_k \hookrightarrow \XF$ (respectivement le morphisme de topo\"i \'etales $\XF_{K,\tr, \et} \rightarrow \XF_{\et}$) et $\bar{i} : \XF_{\bar{k}} \hookrightarrow \XF_{\OC_C}$ (respectivement $\bar{j} : \XF_{C,\tr} \rightarrow \XF_{\OC_C}$) le changement de base associ\'e. Enfin, pour $n$ dans $\N$, on note $\XF_n$ la r\'eduction de $\XF$ modulo $p^n$. 

On note $R\Gamma_{\cris}(\XF_n):=R\Gamma_{\cris}(\XF_n/W_n(k))$ la cohomologie cristalline absolue de $\XF_n$ et $R\Gamma_{\cris}(\XF):= \holim_n R\Gamma_{\cris}(\XF_n)$. On note $\OC_{\XF_n / W_n(k)}$ le faisceau structurel du site cristallin $(\XF_n / W_n(k))_{\cris}$, 
$ \JC_n = \ker( \OC_{\XF_n / W_n(k)} \to \OC_{\XF_n})$ et $\JC_{n}^{[r]}$ sa $r$-i\`eme puissance divis\'ee.
On d\'efinit ensuite : 
\[ R \Gamma_{\cris}( \XF, \JC^{[r]})_n := R \Gamma(\XF_{\text{\'et}} , R{u_n}_{*} \JC^{[r]}_n ) \text{ et } 
R \Gamma_{\cris}( \XF, \JC^{[r]}):= \holim_n R \Gamma_{\cris}( \XF, \JC^{[r]})_n \]
o\`u $u_n$ est la projection $(\XF_n/W_n(k))_{\cris} \to \XF_{\text{\'et}}$. On a alors les complexes syntomiques : 
\[ R \Gamma_{\syn}(\XF,r)_n := [R\Gamma_{\cris}(\XF, \JC^{[r]})_n \xrightarrow{p^r- \varphi} R\Gamma_{\cris}(\XF_n)] \] \[ \text{ et } R \Gamma_{\syn}(\XF,r):=\holim_n R \Gamma_{\syn}(\XF,r)_n, \]
o\`u $\varphi$ d\'esigne le morphisme de Frobenius. Enfin, on note $\mathcal{S}_n(r)_{\XF}$ le faisceau de $\mathcal{D}^+((\XF_k)_{\text{\'et}}, \Z/p^n \Z)$ associ\'e au pr\'efaisceau $ U \mapsto R \Gamma_{\syn}(U,r)_n$. 

On d\'efinit de m\^eme la version g\'eom\'etrique $\mathcal{S}_n(r)_{\overline{\XF}}$ de ce faisceaux (\textit{i.e.} sur $\OC_C$).

\begin{Rem}
\label{syntomique diff}
Dans la suite, on trouvera utile la description suivante, plus explicite, de la cohomologie syntomique. Consid\'erons un log-sch\'ema formel $\XF$ \`a r\'eduction semi-stable sur $\Spf(\OC_C)$ et notons $Y$ sa fibre sp\'eciale. On choisit $\mathfrak{U}^{\bullet}$ un hyper-recouvrement de $\XF$ pour la toplogie \'etale. Soit $\{Z_n^{\bullet} \}$ un syst\`eme inductif de log-sch\'emas log-lisses simpliciaux sur $\AAC$ tels qu'on ait des immersions ferm\'ees $\mathfrak{U}^{s}_n \hookrightarrow Z^s_n$ et des rel\`evements $\varphi_{Z^{\bullet}_n}$ du Frobenius. On note $D_n^s$ la log-PD-enveloppe de $\mathfrak{U}^{s}_n$ dans $Z^s_n$, $J_{D_n^s}$ l'id\'eal associ\'e et $J_{D_n^s}^{[r]}$ sa $r$-i\`eme puissance divis\'ee. Pour chaque $s$, on d\'efinit le faisceau $\mathcal{S}_n(r)_{\mathfrak{U}^{s}, Z^{s}, \varphi^{s}}$ de $D(Y_{\text{\'et}}, \Z/p^n \Z)$ par :
\[ \mathcal{S}_n(r)_{\mathfrak{U}^{s}, Z^{s}, \varphi^{s}} := [J_{D_n^{s}}^{[r- \bullet]} \otimes_{\OC_{Z_n^{s}}} \omega^{\bullet}_{Z_n^{s}} \xrightarrow{p^r- \varphi} \OC_{D_n^{s}} \otimes_{\OC_{Z_n^{s}}} \omega^{\bullet}_{Z_n^{s}}  ] \]
 o\`u $\omega^{\bullet}_{Z_n^{s}}$ est le complexe des log-diff\'erentielles de $Z_n^{s}$. On note $R \Gamma_{\mathrm{syn}}((\mathfrak{U}^{\bullet},Z^{\bullet}, \varphi^{\bullet}),r)_n$ l'hypercohomologie associ\'ee (on utilise ici la d\'efinition (6.10) de \cite{Con2003}). En particulier, on a une suite spectrale (voir \cite[Th. 6.11]{Con2003}) : 
\begin{equation}
\label{hypercoho}
E_1^{p,q}:= H_{\rm syn}^q((\mathfrak{U}^p, Z^p, \varphi^p),r)_n \Rightarrow H_{\rm syn}^{p+q}((\mathfrak{U}^{\bullet}, Z^{\bullet}, \varphi^{\bullet}),r)_n 
\end{equation} 

On note $\operatorname{HRF}(\XF)$ la cat\'egorie dont les objets sont donn\'es par les triplets $(\mathfrak{U}^{\bullet}, \{ Z_n^{\bullet} \}, \{\varphi_{n}^{\bullet} \})$ o\`u $\mathfrak{U}^{\bullet} \to \XF$ est un hyper-recouvrement et les $\{ Z_n^{\bullet} \}$ et $\{\varphi_{n}^{\bullet} \}$ sont d\'efinis comme ci-dessus. On abr\'egera par  $\mathfrak{U}^{\bullet}$ un tel triplet. Un morphisme $(\mathfrak{U}^{\bullet}, \{ Z_{\mathfrak{U},n}^{\bullet} \}, \{\varphi_{\mathfrak{U},n}^{\bullet} \})\to (\mathfrak{V}^{\bullet}, \{ Z_{\mathfrak{V},n}^{\bullet} \}, \{\varphi_{\mathfrak{V},n}^{\bullet} \})$ de $\operatorname{HRF}(\XF)$ est une paire $(u : \mathfrak{U}^{\bullet} \to \mathfrak{V}^{\bullet}, \tilde{u}_n: Z_{\mathfrak{U},n}^{\bullet} \to Z_{\mathfrak{V},n}^{\bullet}  )$ telle que, pour tout $s$, le diagramme
\[ \xymatrix{
  Z_{\mathfrak{U},n}^{s} \ar[r]^{\tilde{u}_n} & Z_{\mathfrak{V},n}^{s} \\ 
  \mathfrak{U}^{s} \ar[r]^{u} \ar[u] & \mathfrak{V}^{s} \ar[u] } \]
  commute et telle que les $\tilde{u}_n$ commutent aux $\varphi_n^s$. 
On peut alors consid\'erer la limite $\hocolim_{\operatorname{HRF}(\XF)} R \Gamma_{\mathrm{syn}}((\mathfrak{U}^{\bullet},Z^{\bullet}, \varphi^{\bullet}),r)_n$ ; les morphismes de transition sont des quasi-isomorphismes. Comme le morphisme $\mathfrak{U}^{\bullet} \to \XF$ est de descente cohomologique (voir \cite[Ex. 6.9]{Con2003}), on a un quasi-isomorphisme (\cite[Th. 7.11]{Con2003}) : 
\begin{equation}
\label{cohodescente}
R \Gamma_{\mathrm{syn}}(\XF, r)_n \xrightarrow{\sim} \hocolim_{\operatorname{HRF}(\XF)} R \Gamma_{\mathrm{syn}}((\mathfrak{U}^{\bullet},Z^{\bullet}, \varphi^{\bullet}),r)_n
\end{equation}
et cette colimite est cribl\'ee.
\end{Rem}

\subsection{R\'esultat local}

On va supposer $\XF_{\OC_C}= \Spf(R)$. 
 
\subsubsection{R\'esultat pr\'eliminaire}

La proposition suivante est donn\'ee dans \cite[\textsection 2.1]{CN2017}, on l'utilise pour d\'efinir un morphisme de Frobenius sur les anneaux utilis\'es.
On consid\`ere $\lambda : \Lambda_1 \to \Lambda_2$ un morphisme d'anneaux topologiques et $\Lambda_1'$ la compl\'etion d'une $\Lambda_1$-alg\`ebre \'etale : 
\[ \Lambda_1':= \Lambda_1 \{ Z_1, \dots , Z_s \}/ (Q_1, \dots, Q_s) \] 
avec $J:=(\frac{\partial Q_j}{\partial Z_i})_{1 \le i,j \le s}$ inversible. 

\textbf{Notation.} Si $F= \sum_{\mathbf{k} \in \N^s} a_{\mathbf{k}} Z^{\mathbf{k}} $ est une s\'erie dans un anneau $\Lambda_1\{Z\}$, on note $F^{\lambda}$ la s\'erie $\sum_{\mathbf{k} \in \N^s} \lambda(a_{\mathbf{k}}) Z^{\mathbf{k}} $. 

\begin{Pro}\cite[Prop. 2.1 et Rem. 2.2]{CN2017}
\label{fonctions implicites} 
On suppose qu'il existe $Z_{\lambda}$ dans $\Lambda_2^s$ et $I$ un id\'eal de $\Lambda_2$ tels que $\Lambda_2$ est s\'epar\'e et complet pour la topologie $I$-adique et tels que les coordonn\'ees de $Q^{\lambda}(Z_{\lambda})$ sont dans $I$. 
Alors l'\'equation $Q^{\lambda}(Y)=0$ admet une unique solution dans $Z_{\lambda} + I^s$ et on peut \'etendre $\lambda$ de mani\`ere unique \`a $\Lambda_1'$. 
\end{Pro}

\subsubsection{Mod\`eles locaux}

On munit $\Spf(\OC_C)$ de la log-structure standard $\OC_C \setminus \{0 \} \hookrightarrow \OC_C$ et on note $\OC_C^{\times}$ ce log-sch\'ema formel. On  note 
\[ R_{\square} := \OC_{C} \{ X, \frac{1}{X_1  \cdots X_a}, \frac{ \varpi}{X_{a+1} \cdots X_{a+b}} \}\]
o\`u $X=(X_1, \dots , X_d)$ et on consid\`ere $R$ la compl\'etion $p$-adique d'une alg\`ebre \'etale sur $R_{\square}$. On munit $\Spf(R_{\square})$ et $\Spf(R)$ de la structure logarithmique induite par la fibre sp\'eciale et le diviseur $D:=\{X_{a+b+1} \cdots X_d =0\}$.  

\begin{Rem}
\label{HomeGo}
Les log-structures d\'ecrites ci-dessus ne sont pas fines, mais de la m\^eme fa\c{c}on que dans \cite{CK17}, on peut, via un changement de base, se ramener \`a des log-structures fines, et de cette fa\c{c}on, travailler avec les log-structures pr\'ec\'edentes comme si elles \'etaient fines. Plus pr\'ecis\'ement, munissons $\Spf(\OC_C)$ et $\Spf(R)$ des log-structures fines suivantes (voir \cite[\textsection 1.6]{CK17}) : 
\begin{itemize}
\item[$\bullet$] sur $\Spf(R)$, on consid\`ere la log-structure donn\'ee par la carte : 
\[ \N^{d-a}  \to \Gamma(\Spf(R), \OC_{\Spf(R)}) \]
donn\'ee par $(n_i) \mapsto \prod_{a+1 \le i \le d} X_i^{n_i}$ sur $\N^{d-a}$.
\item[$\bullet$] sur $\Spf(\OC_C)$, on consid\`ere la log-structure donn\'ee par la carte :
\[ \N \to \OC_C, n \mapsto \varpi^n. \]
\end{itemize} 
Les changements de base le long des applications 
\[ \N^{d-a} \to \OC_{\Spf(R)} \cap \overline{j}_*(\OC^*_{\Sp(R[\frac{1}{p}]) \setminus D_C}), (n_i) \mapsto X_i^{n_i} \text{ et } \N \to \OC_C \setminus \{0 \}, n \mapsto \varpi^n \]
permettent de retrouver les log-structures pr\'ec\'edentes. La plupart des propri\'et\'es des morphismes de log-sch\'emas (formels) \'etant stables par changement de base, on se ram\`ene de cette fa\c{c}on \`a des log-structures fines.  
\end{Rem}

Si, pour tout $n \ge 0$, on munit $\OC_C/p^n$ et $R/p^n$ des log-structures induites, alors $R/p^n$ est log-lisse sur $\OC_C/p^n$, de sorte que les complexes des log-diff\'erentielles $\Omega^i_{(R/p^n)/ (\OC_C/p^n)}$ est le $\OC/p^n$-module engendr\'e par les $\frac{dX_i}{X_i}$ pour $1 \le i \le d$. 
 
On consid\`ere les anneaux de Fontaine $\A_{\inf}$ et $\AAC$ et on rappelle qu'on a les \'el\'ements suivants : 
\[ \varepsilon = (1, \zeta_p, \zeta_{p^2}, \dots) \in \OC_C^{\flat} \text{ (o\`u $\zeta_{p^n}$ d\'esigne une racine primitive $p^n$-i\`eme de l'unit\'e)} \]
\[ \mu= [\varepsilon]-1 \in \A_{\inf}, \qquad \xi = \frac{\mu}{\varphi^{-1}(\mu)} \in \A_{\inf} \qquad \text{et} \qquad t=\log(1+ \mu) \in \AAC. \]
On a une application surjective $\theta : \AAC \to \OC_C$ de noyau engendr\'e par $\xi$. On note $F^r\AAC$ la filtration sur $\AAC$ donn\'ee par les puissances divis\'ees de $\xi$.   

On munit $\AAC$ de la log-structure associ\'ee \`a la pr\'e-log-structure 
\[ \OC_C^{\flat} \setminus \{ 0 \} \to \AAC, x \mapsto [x] \]
et les $\AAC/p^n$ de la log-structure tir\'ee en arri\`ere. En fait, cette log-structure est l'unique log-structure sur $\AAC/p^n$ qui \'etend celle de $\OC_C/p^n$ (voir \cite[\textsection 1.17]{Bei2013}). Enfin, on munit $\A_{\inf}$ de la log-structure venant de celle de $\AAC$.   

 On d\'efinit 
\[R_{\inf, \square}^{+}:= \A_{\inf} \{ X, \frac{1}{X_1  \cdots X_a}, \frac{ [\varpi^{\flat}]}{X_{a+1} \cdots X_{a+b}} \}. \]
$R_{\inf, \square}^{+}$ est complet pour la topologie $(p, \xi)$-adique. On va relever $R_{\square} \to R$ en un morphisme \'etale $R_{\inf,\square}^{+} \to R_{\inf}^+$. Pour cela, on \'ecrit $R:= R_{\square} \{ Z_1, \dots, Z_s \}/ (Q_1, \dots , Q_s)$ avec $\mathrm{det}(\frac{\partial Q_j}{\partial Z_i})$ inversible dans $R$. Soient $\tilde{Q}_j$ des relev\'es des $Q_j$ dans $R_{\inf, \square}^{+}$. On note $R_{\inf}^{+}$ la compl\'etion $(p, \xi)$-adique de \[R_{\inf,\square}^{+} [Z_1, \dots , Z_s]/ (\tilde{Q_1}, \dots, \tilde{Q_s})\] et $\mathrm{det}(\frac{\partial \tilde{Q}_j}{\partial Z_i})$ est inversible dans $R_{\inf, \square}^{+}$ (car il l'est modulo $\xi$ et $R_{\inf, \square}^{+}$ est $\xi$-adiquement complet).
      
 On note $R^+_{\cris, \square}$ (respectivement $R^+_{\cris}$) l'anneau $R_{\inf,\square}^{+} \widehat{\otimes}_{\A_{\inf}} \A_{\cris}$ (resp. $R_{\inf}^{+} \widehat{\otimes}_{\A_{\inf}} \A_{\cris}$), o\`u les produits tensoriels sont compl\'et\'es pour la topologie $p$-adique. L'anneau $R_{\cris}^+$ admet une filtration donn\'ee par $F^rR_{\cris}^+:= R_{\inf}^+ \widehat{\otimes}_{\A_{\inf}} F^r \AAC$.
 
On a le diagramme commutatif :  
\[
   \xymatrix{
    \Spf(R) \ar@{^{(}->}[r] \ar[d] & \Spf(R_{\cris}^+) \ar[d] \\
    \Spf(R_{\square}) \ar@{^{(}->}[r] \ar[d] & \Spf(R_{\cris, \square}^+) \ar[d] \\
    \Spf(\OC_C) \ar@{^{(}->}[r] & \Spf(\AAC) 
  }.
\]

On d\'efinit une log-structure sur $\Spf(R_{\inf, \square}^+)$ (respectivement $\Spf(R_{\cris, \square}^+)$) en ajoutant le diviseur \`a l'infini $\{ X_{a+b+1} \cdots X_d=0 \}$ \`a la log-structure venant de $\Spf(\A_{\inf})$ (respectivement $\Spf(\AAC)$). On munit $R_{\inf}^+$ et $R_{\cris}^+$ de la log-structure venant de $R_{\inf, \square}^+$ et $R_{\cris, \square}^+$, respectivement.    

Comme ci-dessus, on a alors que le complexe des log-diff\'erentielles de $R_{\inf}^+$ (respectivement $R_{\cris}^+$) sur $\A_{\inf}$ (respectivement $\AAC$) est le $\A_{\inf}$-module (respectivement $\AAC$-module) engendr\'e par les $\frac{dX_i}{X_i}$ pour $1 \le i \le d$.     
  
On note $R_{\inf}$ (respectivement $R_{\cris}$) la compl\'etion $p$-adique de $R_{\inf}^+ [ \frac{1}{ [ \varpi^{\flat} ]} ]$ (resp. de $R_{\cris}^+ [ \frac{1}{ [ \varpi^{\flat} ]} ]$).

\subsubsection{Complexe syntomique local}

On va commencer par prouver une version locale du th\'eor\`eme \ref{Main the1}. On suppose donc que $\XF$ s'\'ecrit $\Spf(R_0)$ avec $R_0$ la compl\'etion $p$-adique d'une alg\`ebre \'etale sur \[\OC_{K}\{X_1, \dots, X_d, \frac{1}{X_1 \cdots X_a}, \frac{\varpi}{X_{a+1} \cdots X_{a+b}}\}\] et on \'ecrit $\XF_{\OC_C}:= \Spf(R)$. Dans la suite, on note 
$ R_{\square} := \OC_{C} \{ X_1, \dots, X_d, \frac{1}{X_1  \cdots X_a}, \frac{ \varpi}{X_{a+1} \cdots X_{a+b}} \}.$
On utilise les log-structures d\'efinies dans la partie pr\'ec\'edente. 
 
Via les quasi-isomorphismes suivants montr\'es par Beilinson dans \cite[(1.18.1)]{Bei2013}, 
\begin{equation}
\label{Cohocris-Acris}
 \begin{split}
 R\Gamma((\overline{\XF}_{n}/W_n(\bar{k}))_{\cris}, \FC) \xrightarrow{\sim} R\Gamma((\overline{\XF}_{n}/\A_{\cris,n})_{\cris}, \FC) \\
 R\Gamma((\overline{\XF}/W(\bar{k}))_{\cris}, \FC) \xrightarrow{\sim} R\Gamma((\overline{\XF}/\A_{\cris})_{\cris}, \FC)
 \end{split}
\end{equation}
on obtient que la cohomologie cristalline absolue de $\overline{\XF}_{n}$ et sa filtration en degr\'e $r$ sont calcul\'ees par les complexes  
$R\Gamma((\overline{\XF}_{n}/\A_{\cris,n})_{\cris}, \OC)  \text{ et } R\Gamma((\overline{\XF}_{n}/\A_{\cris,n})_{\cris}, \JC^{[r]}) .$
Comme $\widetilde{\XF}:=\Spf(R^+_{\cris})$ est un rel\`evement log-lisse de $\XF$ sur $\AAC$ et que $\XF \hookrightarrow \widetilde{\XF}$ est exacte on obtient des quasi-isomorphismes naturels 
\[ R\Gamma((\overline{\XF}_{n}/\A_{\cris,n})_{\cris}, \OC) \cong \Omega^{\bullet}_{R^+_{\cris,n}} \text{ et }  R\Gamma((\overline{\XF}_{n}/\A_{\cris,n})_{\cris}, \JC^{[r]}) \cong F^r\Omega^{\bullet}_{R^+_{\cris,n}} \] 
o\`u $R^+_{\cris,n}$ est la r\'eduction modulo $p^n$ de $R^+_{\cris}$ et 
\[F^r\Omega^{\bullet}_{R^+_{\cris,n}} := F^r R_{\cris,n}^+ \to F^{r-1} \Omega^1_{R_{\cris,n}^+ } \to F^{r-2} \Omega^2_{R_{\cris,n}^+ } \to \dots. \]
En particulier, on a $R\Gamma_{\cris}(\XF, \JC^{[r]}) \cong F^r \Omega^{\bullet}_{R_{\cris}^+}$ et $R \Gamma_{\cris}( \XF ) \cong  \Omega^{\bullet}_{R_{\cris}^+}$ et on en d\'eduit que $R \Gamma_{\syn}( \XF_{\OC_C}, r)$ et $R \Gamma_{\syn}( \XF_{\OC_C}, r)_n$ sont calcul\'es par les complexes :
\[ \Syn( R_{\cris}^+, r):= [ F^r \Omega^{\bullet}_{R_{\cris}^+} \xrightarrow{p^r - \varphi}  \Omega^{\bullet}_{R_{\cris}^+} ] \text{ et } \Syn( R_{\cris}^+, r)_n:= [ F^r \Omega^{\bullet}_{R_{\cris,n}^+} \xrightarrow{p^r - \varphi}  \Omega^{\bullet}_{R_{\cris,n}^+} ]. \]
  
 Soit $\overline{R}$ l'extension maximale de $R$ telle que $\overline{R}[\frac{1}{p}]/ R[\frac{1}{p}]$ est non ramifi\'ee en dehors de $D$. On note $G_R:= \text{Gal}(\overline{R}[\frac{1}{p}]/R[\frac{1}{p}])$ le groupe des automorphismes de $\overline{R}[\frac{1}{p}]$ qui fixent $R[\frac{1}{p}]$. Le but des sections suivantes est de montrer la version g\'eom\'etrique du th\'eor\`eme (4.14) de \cite{CN2017} :  
 
 \begin{The}
 \label{syn/galois}
 Il existe des $p^{N}$-quasi-isomorphismes : 
\[ \alpha^0_r : \tau_{\le r} \Syn(R_{\cris}^+,r) \xrightarrow{\sim} \tau_{\le r} R \Gamma( G_R, \Z_p(r)) \]
\[ \alpha^0_{r,n} : \tau_{\le r} \Syn(R_{\cris}^+,r)_n \xrightarrow{\sim} \tau_{\le r} R \Gamma( G_R, \Z/p^n(r)) \]
o\`u $N$ est une constante qui ne d\'epend que de $r$. 
\end{The} 

\begin{Rem}
\label{K(pi,1)}
On verra plus loin que ce r\'esultat implique le th\'eor\`eme global \eqref{Main the1}. Dans le cas o\`u $\XF_{\OC_C}:= \Spf(R)$ n'a pas de diviseur horizontal, la cohomologie de Galois $R\Gamma(G_R, \Z/p^n\Z(r))$ calcule la cohomologie \'etale $R\Gamma_{\text{\'et}}(\Sp(R[\frac{1}{p}]), \Z/p^n\Z)$ car l'affino\"ide $\Spa(R[\frac{1}{p}],R)$ est un espace $K(\pi,1)$ pour les coefficients de torsion (voir \cite[Th. (4.9)]{Sch13}). Quand le diviseur horizontal n'est pas trivial, le r\'esultat a \'et\'e prouv\'e par Colmez et Nizio{\l} (voir \cite[ \textsection 5.1.4]{CN2017}) dans le cas arithm\'etique. La preuve est identique pour le cas g\'eom\'etrique et on obtient : 
\[ R\Gamma(G_R, \Z/p^n(r))  \cong R\Gamma_{\text{\'et}}(\Sp(R\left[ \frac{1}{p} \right] )\setminus D_C, \Z/p^n\Z(r)) .\]
\end{Rem}

\section{Anneaux de p\'eriodes}

Dans cette partie, $R$ est un anneau comme en \textsection 2.2.2 et on suppose de plus que $\Spf(R)$ est connexe. Pour la preuve de \ref{syn/galois}, on a besoin de d\'efinir une notion d'anneaux de p\'eriodes ($\A_{\inf}$, $\AAC$,...) sur $\overline{R}$ et $R$. Les suites exactes prouv\'ees dans les sections suivantes joueront notamment un r\^ole important dans la suite. 

\subsection{D\'efinitions des anneaux}

On reprend ici la construction des anneaux de p\'eriodes donn\'ee dans \cite[\textsection 2.4]{CN2017}. 

\subsubsection{Les anneaux $\E_{\overline{R}}$, $\A_{\overline{R}}$ et $\AAC(\overline{R})$}

La compl\'etion $p$-adique $\widehat{\overline{R}}$ de $\overline{R}$ est une alg\`ebre perfecto\"ide et on peut d\'efinir les anneaux : 
\[ \E^+_{\overline{R}}:= \varprojlim_{x \mapsto \varphi(x)}( \widehat{\overline{R}}/p) \text { et } \A^+_{\overline{R}}:=W(\E^+_{\overline{R}}) \] 
\[ \E_{\overline{R}}= \E^+_{\overline{R}}[ \frac{1}{[p^{\flat}]}] \text{ et } \A_{\overline{R}}=W(\E_{\overline{R}}). \]

On rappelle qu'on a les \'el\'ements : 
\[ \varepsilon = (1, \zeta_p, \zeta_{p^2}, \dots) \in \E^+_{\overline{R}}, \qquad \mu= [\varepsilon]-1 \in \A^+_{\overline{R}} \qquad \text{et} \qquad \xi = \frac{\mu}{\varphi^{-1}(\mu)} \in \A^+_{\overline{R}}. \]

On d\'efinit $v_{\E} : \E_{\overline{R}} \to \Q \cup \{ \infty \} $ par 
$ v_{\E}(x_0, x_1, x_2, \dots)= v_p(x_0) $. On a (voir \cite[\textsection 2.10]{AI08}) : 
\begin{enumerate}
\item $v_{\E}(x) = + \infty \text{ si et seulement si } x=0$, 
\item $v_{\E}(xy) \ge v_{\E}(x) + v_{\E}(y)$, 
\item $v_{\E}(x+y) \ge \min(v_{\E}(x), v_{\E}(y)) \text{ avec \'egalit\'e si } v_{\E}(x) \neq v_{\E}(y)$,
\item $v_{\E}(\varphi(x)) = pv_{\E}(x)$. 
\end{enumerate}
On d\'efinit un morphisme surjectif : 
\[ \theta : \begin{cases} \A_{\overline{R}}^+ & \to   \widehat{\overline{R}} \\
                                        \sum_{k \in \N} p^k [x_k] &\mapsto  \sum_{k \in \N} p^k x_k^{(0)}  
                 \end{cases} \]
o\`u $[.] : \E_{\overline{R}} \to \A_{\overline{R}}$ est le morphisme de rel\`evement et $x_k = (x_k^{(0)}, \dots, x_k^{(n)}, \dots)$. On a de plus que le noyau de $\theta$ est principal engendr\'e par $\xi$ (ou par $\xi_0:=p-[p^{\flat}]$).

On peut aussi munir $\A_{\overline{R}}$ de la topologie faible dont un syst\`eme fondamental de voisinages de 0 est donn\'e par les $U_{n,h}:= p^n \A_{\overline{R}} + \mu^h \A_{\overline{R}}$.  

On d\'efinit $\AAC(\overline{R})$ comme la compl\'etion $p$-adique de $ \A_{\overline{R}}^+[ \frac{\xi^k}{k!}, k \in \N]$  et on le munit de la filtration $F^{\bullet} \AAC( \overline{R})$ o\`u $F^{r} \AAC( \overline{R})$ est l'id\'eal de $\AAC(\overline{R})$ engendr\'e par les $\frac{\xi^k}{k!}$ pour $k \ge r$.  
On a l'\'el\'ement 
\[t=\log(1+ \mu) \in \AAC(\overline{R}). \]
On d\'efinit $\B_{\cris}^{+}(\overline{R})= \AAC(\overline{R}) [\frac{1}{p}]$ et $\B_{\cris}(\overline{R})= \B_{\cris}^{+}(\overline{R})[\frac{1}{t}]$ et on les munit de la filtration induite de celle de $\AAC$.
Soit $\B^{+}_{\dR}(\overline{R})= \varprojlim_{r} \B^{+}_{\cris}(\overline{R})/F^r  \B^{+}_{\cris}(\overline{R})$ et $\B_{\dR}(\overline{R})= \B_{\dR}^{+}(\overline{R})[\frac{1}{t}]$. On munit ces anneaux de la filtration donn\'ee par les puissances de $\xi$. On note ensuite $\B_{\text{st}}^+(\overline{R}):= \B_{\cris}^+(\overline{R})[u]$ o\`u $u$ est une variable. On \'etend l'action du Frobenius $\varphi$ en posant $\varphi(u)=pu$. On d\'efinit une application de monodromie $N : \B_{\text{st}}^+(\overline{R}) \to \B_{\text{st}}^+(\overline{R})$ par $N=- \frac{d}{du}$. On a un plongement $\B_{\text{st}}^+(\overline{R}) \hookrightarrow \B_{\dR}^+(\overline{R})$ donn\'e par $u \mapsto u_{\varpi}:=\log( \frac{[\varpi^{\flat}]}{\varpi})$ et qui permet d'induire l'action de $G_K$ \`a $\B_{\text{st}}^+(\overline{R})$. Enfin, on note $\B_{\text{st}}(\overline{R})= \B_{\cris}(\overline{R})[u_{\varpi}]$, l'anneau de p\'eriode semi-stable.  

\subsubsection{Les anneaux $\A_{\overline{R}}^{(0,v]}$, $\A_{\overline{R}}^{[u]}$ et $\A_{\overline{R}}^{[u,v]}$}

 Pour $0 < u \le v$, on d\'efinit les anneaux : 
  
\begin{itemize}
\item[$\bullet$] $\A_{\overline{R}}^{[u]}$, la compl\'etion $p$-adique de $\A_{\overline{R}}^+ \left[ \frac{[ \beta ]}{p} \right] $ pour $\beta$ un \'el\'ement de $\E_{\overline{R}}^+$ avec $v_{\E}(\beta)=\frac{1}{u}$. 
  
\item[$\bullet$] $\A_{\overline{R}}^{[u,v]}$, la compl\'etion $p$-adique de $\A_{\overline{R}}^+ \left[ \frac{p}{[\alpha]}, \frac{[ \beta ]}{p} \right] $ pour $\alpha$ et $\beta$ des \'el\'ements de $\E_{\overline{R}}^+$ avec $v_{\E}(\alpha)=\frac{1}{v}$ et $v_{\E}(\beta)=\frac{1}{u}$. 

\item[$\bullet$] $\A_{\overline{R}}^{(0, v]}:= \{ x=\sum_{n \in \N} [x_n] p^n \in \A_{\overline{R}} \; | \; x_n \in \E_{\overline{R}}, \; v_{\E}(x_n) + \frac{n}{v} \to + \infty \text{ quand } n \to \infty \}$. 

\end{itemize}
On d\'efinit une application $w_v : \A_{\overline{R}}^{(0, v]} \to \R \cup \{\infty \}$ par : 
 $$ w_v(x) = \begin{cases} \infty & \text{ si } x=0 \\ \inf_{n \in \N} (v v_{\E}(x_n)+n) & \text{ sinon}. \end{cases} $$
On note ensuite $\A_{\overline{R}}^{(0, v]+}$ l'ensemble des \'el\'ements $x$ de $\A_{\overline{R}}^{(0, v]}$ tels que $w_v(x) \ge 0$. L'anneau $\A_{\overline{R}}^{(0, v]}$ est s\'epar\'e et complet pour la topologie induite par $w_v$ (la preuve est identique \`a celle de \cite[Prop. 4.3]{AB08}).

Si $v \ge 1 \ge u$, on a des injections 
$$\A_{\overline{R}}^{[u]} \hookrightarrow \Bdrp( \overline{R}), \; \A_{\overline{R}}^{[u,v]} \hookrightarrow \Bdrp( \overline{R}) \text{ et } \A_{\overline{R}}^{(0, v]} \hookrightarrow \Bdrp( \overline{R}),$$
ce qui nous permet de d\'efinir des filtrations sur ces trois anneaux. 

On note $\A^{[u]}$, $\A^{[u,v]}$ et $\A^{(0,v]}$ les anneaux pr\'ec\'edents obtenus pour $R= \OC_C$. 

\begin{Pro}
\label{FrAuv}
Supposons $u \le 1 \le v$. L'id\'eal $p^r F^r \A^{[u,v]}$ est inclus dans $\xi^r \A^{[u,v]}= \xi_0^r \A^{[u,v]}$ o\`u $\xi_0:= p-[p^{\flat}]$. 
\end{Pro}

\begin{proof}
On rappelle d'abord qu'on a (voir \cite[A2.13]{Tsu99}) :  
\begin{equation}
\label{FrAinf}
\A_{\inf} \cap (\xi^r \Bdrp) = \xi^r \A_{\inf}= \xi_0^r \A_{\inf}.
\end{equation}
Soit maintenant $x$ dans $F^r \A^{[u,v]} = (\xi^r \Bdrp) \cap \A^{[u,v]}$. Par d\'efinition de $\A^{[u,v]}$, on peut \'ecrire 
\[ x= \sum_{n \ge 0} a_n \frac{p^n}{[ \alpha^n]} + \sum_{n \ge 0} b_n \frac{[\beta^n]}{p^n} \] 
avec $v_{\E}(\alpha)=\frac{1}{v}$ et $v_{\E}(\beta)=\frac{1}{u}$ et $(a_n)$ et $(b_n)$ des suites de $\A_{\inf}$ qui tendent vers 0. Pour simplifier, on note 
\[ A:= \sum_{n \ge 0} a_n \frac{p^n}{[\alpha^n]} \text{ et } B:= \sum_{n \ge 0} b_n \frac{[\beta^n]}{p^n}. \]

On montre d'abord que $p^r A$ est dans $\A_{\inf} + \xi^r \A^{[u,v]}$. On a $\frac{p}{[ \alpha]}= \frac{1}{1+ \frac{\xi_0}{p}} [ \alpha' ]$ avec $v_{\E}(\alpha')= 1 - \frac{1}{v} \ge 0$. D'o\`u, dans $\Bdrp$ : 
\[
p^r \frac{p}{[\alpha]}  =   p^r \cdot (\sum_{n \ge 0} (-1)^n \frac{\xi_0^n}{p^n}) \cdot [ \alpha']  
 = (\sum_{0 \le n < r} (-1)^n \xi_0^n p^{r-n}) \cdot [\alpha'] + \xi_0^r \cdot (\sum_{n \ge 0} (-1)^{n+r} \frac{\xi_0^n}{p^n}) \cdot [\alpha']  \]  
c'est-\`a-dire : 
\[ p^r \frac{p}{[\alpha]}= a + \xi_0^r \frac{(-1)^r}{1+ \frac{\xi_0}{p}} [ \alpha' ] =  a + \xi_0^r \cdot (-1)^r \cdot \frac{p}{[\alpha]} \text{ avec } a \in \A_{\inf}. \] 
On en d\'eduit le r\'esultat. 

Montrons ensuite que $p^rB$ est dans $\xi^r \A^{[u,v]}$. Comme pour le point pr\'ec\'edent, on a $\frac{[\beta]}{p}= (1+ \frac{\xi_0}{p}) [\beta']$ avec $v_{\E}(\beta')= \frac{1}{u}-1 \ge 0$. On peut alors \'ecrire 
\[ B = \sum_{n \ge 0} b'_n [(\beta')^n] \frac{\xi_0^n}{p^n} \text{ avec } b'_n \in \A_{\inf} \]
et on peut supposer que $\xi_0$ ne divise pas $b'_n$ dans $\A_{\inf}$. 

Mais $\Bdrp$ est un anneau de valuation discr\`ete dont l'id\'eal maximal est donn\'e par $\xi$ : on note $v_{\xi}$ sa valuation. Par hypoth\`ese, on a $v_{\xi}(B) \ge r$. De plus, en utilisant \eqref{FrAinf}, on voit que 
$v_{\xi} (b'_n [(\beta')^n] \frac{\xi_0^n}{p^n})= n$ et donc $b'_n=0$ pour tout $n < r$. On a alors 
\[ p^r B = p^r \cdot (\sum_{n \ge r} b'_n [(\beta')^n] \frac{\xi_0^n}{p^n} )= \xi_0^r \cdot( \sum_{n \ge 0} b'_{n+r} [(\beta')^{n+r}] \frac{\xi_0^{n}}{p^n})  \in \xi^r \A^{[u,v]}. \]        

On obtient que $p^rx=p^rA+p^rB$ est dans $\A_{\inf} + \xi^r \A^{[u,v]}$, \'ecrivons-le $x_0 + \xi^r \tilde{x}$ avec $x_0 \in \A_{\inf}$ et $\tilde{x} \in \A^{[u,v]}$. On a $x_0=p^rx- \xi^r\tilde{x}$ est dans $\xi^r \B_{\dR}^+$ et comme $x_0$ est aussi dans $\A_{\inf}$, via \eqref{FrAinf}, on obtient que $x_0$ est dans $\xi^r \A_{\inf}$. 

Finalement, on obtient que $p^rx$ est dans $\xi^r \A^{[u,v]}$ et donc $p^rF^r \A^{[u,v]} \subseteq \xi^r \A^{[u,v]}$.

\end{proof}

On utilisera plus loin le r\'esultat suivant : 

\begin{Pro}\cite[\textsection 2.4.2]{CN2017}
\label{Au et Acris}
On a les inclusions : 
\begin{enumerate}
\item $\AAC( \overline{R} ) \subseteq \A_{\overline{R}}^{[u]} $ si $u \ge \frac{1}{p-1}$ ;
\item $\AAC( \overline{R} ) \supseteq \A_{\overline{R}}^{[u]} $ si $u \le \frac{1}{p}$.
\end{enumerate}
\end{Pro}

\subsection{Suites exactes fondamentales}

On rappelle ici les diff\'erentes suites exactes fondamentales v\'erifi\'ees par les anneaux de p\'eriode d\'efinis pr\'ec\'edemment. 

\subsubsection{Suites exactes pour les anneaux $\A_{\overline{R}}$ et $\A_{\overline{R}}^{(0,v]+}$}

\begin{The}\cite[8.1]{AI08}
On a les suites exactes : 
\begin{equation}
\label{suite exacte Ar}
0 \to \Z_p \to \A_{\overline{R}} \xrightarrow{1- \varphi} \A_{\overline{R}} \to 0
\end{equation}
\begin{equation}
\label{suite exacte Av}
0 \to \Z_p \to \A_{\overline{R}}^{(0,v]+} \xrightarrow{1- \varphi} \A_{\overline{R}}^{(0, \frac{v}{p}]+} \to 0
\end{equation}
\end{The}

\begin{Rem}
\label{suite exacte Ar+}
De la m\^eme fa\c{c}on on a une suite exacte : $0 \to \Z_p \to \A_{\overline{R}}^{+} \xrightarrow{1- \varphi} \A_{\overline{R}}^{+} \to 0$.
\end{Rem}

\subsubsection{Suite exacte pour l'anneau $\AAC$}
 
Pour $r$ dans $\N$, on \'ecrit $r= a(r) (p-1) +b(r)$ avec $0 \le b(r) \le p-2$ et on note $t^{\{r\}}:= \frac{t^r}{a(r)! p^{a(r)}}$. Alors, comme on a 
$$  \frac{t^r}{a(r)! p^{a(r)}}= t^{b(r)} \left(\frac{t^{p-1}}{p}\right)^{[a(r)]} $$
et que $\frac{t^{p-1}}{p}$ est dans $\AAC(\overline{R})$, on obtient que $t^{\{r\}}$ est dans $\AAC(\overline{R})$. 

\begin{The}\cite[A3.26]{Tsu99}
\label{suite exacte Acris1}
Pour tout $r \ge 0$, on a la suite exacte :
\begin{equation}
\label{suite exacte Acris}
 0 \to \Z_p t^{\{ r\}} \to F^r\AAC(\overline{R}) \xrightarrow{1 - \frac{\varphi}{p^r}} \AAC(\overline{R}) \to 0. \end{equation}
\end{The}

\subsubsection{Suites exactes pour les anneaux $\A_{\overline{R}}^{[u,v]}$ et  $\A_{\overline{R}}^{[u]}$ }

Colmez-Nizio{\l} dans \cite{CN2017} ont montr\'e que des r\'esultats similaires sont v\'erifi\'es par les anneaux $\A_{\overline{R}}^{[u,v]}$ et  $\A_{\overline{R}}^{[u]}$.

\begin{The}\cite[Lem. 2.23]{CN2017}
Soit $u$ tel que $ \frac{p-1}{p} \le u \le 1$. 
\begin{enumerate}
\item On a une suite $p^{2r}$-exacte si $p>2$ (ou $p^{3r}$-exacte si $p=2$) : 
\begin{equation}
\label{suite exacte Au 1}
0 \to \Z_p(r) \to (\A_{\overline{R}}^{[u]})^{\varphi=p^r} \to \A_{\overline{R}}^{[u]}/F^r \to 0
\end{equation}
\item On a une suite $p^{2r}$-exacte : 
 \begin{equation}
\label{suite exacte Au 2}
0 \to (\A_{\overline{R}}^{[u]})^{\varphi=p^r} \to \A_{\overline{R}}^{[u]} \xrightarrow{p^r - \varphi} \A_{\overline{R}}^{[u]} \to 0
\end{equation}
\item On a une suite $p^{4r}$-exacte si $p>2$ (ou $p^{5r}$-exacte si $p=2$) :
\begin{equation}
\label{suite exacte Au 3}
0 \to \Z_p(r) \to F^r \A_{\overline{R}}^{[u]} \xrightarrow{p^r - \varphi} \A_{\overline{R}}^{[u]} \to 0.
\end{equation}
\end{enumerate}
\end{The}

\begin{The}\cite[Lem. 2.23]{CN2017}
Soient $u$ et $v$ tels que $ \frac{p-1}{p} \le u \le 1 \le v$. 
\begin{enumerate}
\item On a une suite $p^{3r}$-exacte si $p>2$ (ou $p^{4r}$-exacte si $p=2$) : 
\begin{equation}
\label{suite exacte Auv 1}
0 \to \Z_p(r) \to (\A_{\overline{R}}^{[u,v]})^{\varphi=p^r} \to \A_{\overline{R}}^{[u,v]}/F^r \to 0
\end{equation}
\item On a une suite $p^{3r}$-exacte : 
 \begin{equation}
\label{suite exacte Auv 2}
0 \to (\A_{\overline{R}}^{[u,v]})^{\varphi=p^r} \to \A_{\overline{R}}^{[u,v]} \xrightarrow{p^r - \varphi} \A_{\overline{R}}^{[u,\frac{v}{p}]} \to 0
\end{equation}
\item On a une suite $p^{6r}$-exacte si $p>2$ (ou $p^{7r}$-exacte si $p=2$) :
\begin{equation}
\label{suite exacte Auv 3}
0 \to \Z_p(r) \to F^r \A_{\overline{R}}^{[u,v]} \xrightarrow{p^r - \varphi} \A_{\overline{R}}^{[u,\frac{v}{p}]} \to 0.
\end{equation}
\end{enumerate}
\end{The}

\subsection{Anneaux de convergence et morphisme de Frobenius}

On consid\`ere les anneaux : 
$$ R^{[u]} := \A^{[u]} \widehat{\otimes}_{\A_{\inf}} R_{\inf}^+, \qquad R^{[u,v]} := \A^{[u,v]} \widehat{\otimes}_{\A_{\inf}} R_{\inf}^+, \qquad R^{(0,v]+} := \A^{(0,v]+} \widehat{\otimes}_{\A_{\inf}} R_{\inf}^+ $$
o\`u les produits tensoriels sont compl\'et\'es pour la topologie $(p, \xi)$-adique. On d\'efinit les filtrations : 
\[ R^{[u]} := F^r\A^{[u]} \widehat{\otimes}_{\A_{\inf}} R_{\inf}^+, \qquad R^{[u,v]} := F^r\A^{[u,v]} \widehat{\otimes}_{\A_{\inf}} R_{\inf}^+, \qquad R^{(0,v]+} := F^r\A^{(0,v]+} \widehat{\otimes}_{\A_{\inf}} R_{\inf}^+. \]
On rappelle qu'on avait :
$$ R^{+}_{\cris, \square}:= \AAC\{ X, \frac{1}{X_1  \cdots X_a}, \frac{ [\varpi^{\flat}]}{X_{a+1} \cdots X_{a+b}} \} \text{ et } R^+_{\cris} := \AAC \widehat{\otimes}_{\A_{\inf}} R_{\inf}^+ $$
 et que $R_{\inf}$ et $R_{\cris}$ sont les compl\'etions $p$-adiques de $R^+_{\inf}[\frac{1}{[\varpi^{\flat}]}]$ et $R^+_{\cris}[\frac{1}{[\varpi^{\flat}]}]$.  
 
Si $\varphi$ est le Frobenius sur $\A_{\mathrm{inf}}$, on \'etend $\varphi$ \`a $R_{\inf, \square}^{+}$ en posant $\varphi(X_i)=X_i^p$ pour $i$ dans $\{1, \dots, d \}$. On utilise ensuite la proposition \ref{fonctions implicites} pour l'\'etendre \`a $R_{\inf}^{+}$ (prendre $\Lambda_1=R_{\inf, \square}^{+}$, $\Lambda_2= \Lambda_1'= R_{\inf}^{+}$, $Z_{\varphi}=Z^p$ et $I=(p)$). 

On d\'efinit $\varphi$ sur $\A^{[u] }$, $\A^{[u,v]}$ et $\A^{(0,v]+}$ par : 
\[ \varphi\left(\frac{p}{[\alpha]}\right)= \frac{p}{[\alpha^p]}, \qquad \varphi\left(\frac{[ \beta ]}{p} \right)= \frac{[ \beta^p ]}{p}  \qquad  \text{et} \qquad  \varphi(\sum_{n \in \N} [x_n] p^n) = \sum_{n \in \N} [x_n^p] p^n \]

et on \'etend $\varphi$ \`a $R^{[u]}$, $R^{[u,v]}$, $R^{(0,v]+}$ et $R_{\cris}^+$. On a alors : 
$$ \varphi(R^{[u]})= R^{[\frac{u}{p}]}, \qquad \varphi(R^{[u,v]})= R^{[\frac{u}{p},\frac{v}{p}]} \qquad \text{et} \qquad \varphi(R^{(0,v]+})= R^{(0,\frac{v}{p}]+}. $$

On note $\psi$ l'inverse de $\varphi$ sur $\A_{\inf}$. On va \'etendre $\psi$ aux anneaux $R_{\inf}$, $R_{\inf}^+$, $R_{\cris}$, $R_{\cris}^+$, $R^{[u]}$, $R^{[u,v]}$ et $R^{(0,v]+}$. 


Pour $\alpha=(\alpha_1, \dots , \alpha_d)$ avec $\alpha_i$ dans $\{0, \dots, p-1 \}$, on note 
\[ u_{\alpha} = X_1^{\alpha_1} \cdots X_d^{\alpha_d} \]
et pour $j$ dans $\{ 1, \dots, d \}$, 
\[ \partial_j = X_j \frac{\partial}{\partial X_j}. \]
La m\^eme preuve que \cite[\textsection 2.2.7]{CN2017} donne le lemme et le corollaire suivants : 

\begin{Lem}\cite[Lem. 2.7]{CN2017}
Tout \'el\'ement $x$ de $R_{\rm{inf}}/p$ s'\'ecrit de mani\`ere unique sous la forme $\sum_{\alpha}c_{\alpha}(x)$ avec $c_{\alpha}(x)= x_{\alpha}^p u_{\alpha}$ pour un $x_{\alpha}$ dans $R_{\rm{inf}}/p$.    

De plus, si $x$ est dans $R^+_{\rm{inf}}/p$ alors $c_{\alpha}(x)$ et $x_0$ sont dans $R^+_{\rm{inf}}/p$.
\end{Lem}

\begin{Cor}\cite[Cor. 2.8]{CN2017}
Tout \'el\'ement $x$ de $R_{\rm{inf}}$ s'\'ecrit de mani\`ere unique sous la forme $\sum_{\alpha}c_{\alpha}(x)$ avec $c_{\alpha}(x)= \varphi(x_{\alpha}) u_{\alpha}$ pour un $x_{\alpha}$ dans $R_{\rm{inf}}$.    

De plus, si $x$ est dans $R^+_{\rm{inf}}$ alors $x_0$ est dans $R^+_{\rm{inf}}$ et 
$ \partial_j c_{\alpha}(x)- \alpha_j c_{\alpha}(x) \in p R^+_{\rm{inf}}. $
\end{Cor}

On d\'efinit alors $\psi$ sur $R_{\rm{inf}}^{+}$ par : 
$\psi(x) = \varphi^{-1}(c_0(x))$. L'application $\psi$ n'est pas un morphisme d'anneau, mais on a $\psi \circ \varphi(x)=x$ et plus g\'en\'eralement, $\psi(\varphi(x)y)=x \psi(y)$ pour $x$ et $y$ dans $R_{\rm{inf}}^{+}$. 

On \'etend ensuite $\psi$ \`a $R_{\inf}$, $R_{\rm{cris}}$, $R_{\rm{cris}}^+$, $R^{[u]}$, $R^{[u,v]}$ et $R^{(0,v]+}$ et on obtient des applications surjectives : 
\[ \psi : R^{[u]} \to R^{[pu]}, \;  R^{[u,v]} \to R^{[pu, pv]} \text{ et } R^{(0,v]+} \to R^{(0, pv]+}. \]

\begin{Rem}
\label{psi-decompo}
Comme dans \cite{CN2017}, on peut voir que les applications $x \mapsto c_{\alpha}(x)$ donnent des d\'ecompositions :
$$ S = \oplus_{\alpha} S_{\alpha} \text{ et } S^{\psi=0} = \oplus_{ \alpha \neq 0} S_{\alpha} $$
pour $S \in \{ R_{\rm{inf}}, R_{\rm{inf}}^+, R_{\rm{cris}}, R_{\rm{cris}}^+, R^{[u]}, R^{[u,v]}, R^{(0,v]+} \}$. On a de plus $\partial_j = \alpha_j$ sur $S_{\alpha}/ p S_{\alpha}$. 

\end{Rem}

\section{Passage de $R_{\cris}^+$ \`a $R^{[u,v]}$} 

Soit $R$ comme dans la section pr\'ec\'edente. La premi\`ere \'etape dans la preuve du th\'eor\`eme \ref{syn/galois} est de construire un quasi-isomorphisme entre le complexe $\Syn(R_{\mathrm{cris}}^+,r)$ et un complexe $C(R^{[u,v]},r)$ d\'efini \`a partir de l'anneau $R^{[u,v]}$. 

Dans cette partie, on suppose $u \ge \frac{1}{p-1}$ de telle sorte que $\AAC \subseteq \A^{[u]} \subseteq \A^{[u,v]}$. 
Si $S= R^{[u]}$ (respectivement $R^{[u,v]}$), on \'ecrit $S'= R^{[u]}$ (respectivement $R^{[u, \frac{v}{p}]}$) et $S''=R^{[pu]}$ (respectivement $R^{[pu,v]}$) et on note $\Omega^{\bullet}_S$ le complexe des log-diff\'erentielles de $S$ sur $\A^{[u]}$ (respectivement $\A^{[u,v]}$). 
Pour $i$ dans $\{1, \dots , d \}$, on note $J_i = \{ (j_1, \dots, j_i)  \; | \; 1 \le j_1 \le \dots \le j_i \le d \}$ et $\omega_i= \frac{ d X_i}{X_i}$. Si $\textbf{j}$ est dans $J_i$, on \'ecrit $\omega_{\textbf{j}}= \omega_{j_1} \wedge \dots \wedge \omega_{j_i}$. La filtration $F^r\Omega_S^i$ est le sous-$S$-module de $\Omega_S^i$ engendr\'e par $F^rS.\Omega_S^i$ : 
\[ F^r\Omega^i_S= \bigoplus_{\textbf{j} \in J_i} F^rS. \omega_{\textbf{j}}. \] 
On \'etend $\varphi$ \`a $\Omega^i_S$ par 
\[ \varphi( \sum_{ \textbf{j} \in J_i} f_{\textbf{j} }\omega_{\textbf{j}}) = \sum_{ \textbf{j} \in J_i} \varphi(f_{\textbf{j}}) \omega_{\textbf{j}}. \]
On d\'efinit ensuite le complexe 
$ C(S,r) := [F^r \Omega_S^{\bullet} \xrightarrow{p^r -  p^{\bullet} \varphi} \Omega_{S'}^{\bullet}]. $  
Enfin, on \'etend $\psi$ \`a $\Omega^i_S$ en posant  
$\psi( \sum_{ \textbf{j} \in J_i} f_{\textbf{j} }\omega_{\textbf{j}}) = \sum_{ \textbf{j} \in J_i} \psi(f_{\textbf{j}}) \omega_{\textbf{j}}$ et on note $C^{\psi}(S,r):=[ F^r \Omega_{S}^{\bullet} \xrightarrow{p^r \psi - p^{\bullet}} \Omega_{S''}^{\bullet}]$

Le $p^{10r}$-quasi-isomorphisme $\tau_{\le r} C(R_{\cris}^+,r) \to \tau_{\le r} C(R^{[u,v]}, r)$ (voir \eqref{Rcris-Ruv} ci-dessous) s'obtient de mani\`ere similaire \`a son analogue arthm\'etique (section \textsection 3.2 de \cite{CN2017}). Les principales diff\'erences viennent du fait qu'on ne dispose pas de la variable arithm\'etique $X_0$ qui permet de donner une interpr\'etation des anneaux consid\'er\'es en termes de s\'eries de Laurent (sur l'anneau $W(k)$) ou de d\'efinir $\psi$ de telle sorte \`a ce qu'il soit "suffisamment" topologiquement nilpotent. La d\'emonstration se fait en trois parties. On montre dans un premier temps qu'on a un $p^{8r}$-quasi-isomorphisme $\Syn(R_{\cris}^{+},r) \xrightarrow{\sim} C(R^{[u]},r)$. Contrairement au cas arithm\'etique, si $f$ est un \'el\'ement de $R^{[u]}$, on n'a pas, a priori, une d\'ecomposition $f=f_1+f_2$ avec $f_1$ dans $F^rR^{[u]}$ et $f_2$ tel que $p^rf_2$ est dans $R_{\inf}^+$ (voir \cite[Rem. 2.6]{CN2017}) : au lieu de cela, on utilise ici la $p^{2r}$-exactitude de la suite \eqref{suite exacte Au 1} (voir lemme \ref{OmegaRcrisRu}). 
L'\'etape suivante consiste \`a passer du complexe $C(R^{[u]},r)$ au complexe $C^{\psi}(R^{[u]},r)$. 
Enfin, on montre que l'inclusion $R^{[u]} \hookrightarrow R^{[u,v]}$ induit un $p^{2r}$-quasi-isomorphisme sur les complexes tronqu\'es $\tau_{\le r} C^{\psi}(S,r)$. La preuve diff\`ere ici de celle de \cite{CN2017} dans la mesure o\`u la s\'erie $\sum_{n \ge 1} \psi^n(x)$ pour $x$ dans $R^{[u,v]}/R^{[u]}$ ne converge pas n\'ecessairement pour notre d\'efinition de $\psi$. Il reste vrai, cependant, $\psi-1$ est surjective sur $\A^{(0,v]+}$, ce qui permet de conclure (voir lemme \ref{Ru to Ruv}).

\subsection{Disque de convergence}

On commence par montrer qu'on a un $p^{10r}$-quasi-isomorphisme $\Syn(R_{\cris}^{+},r) \xrightarrow{\sim} C(R^{[u]},r)$. La preuve du lemme suivant est identique \`a celle du cas arithm\'etique.  

\begin{Lem}\cite[Lem. 3.1]{CN2017}
Soit $f$ dans $R^{[u]}$ tel que $f=\sum_{n \ge N} x_n \frac{[\beta]^n}{p^n}$ pour un $N$ dans $\N$ et soit $s$ dans $\Z$. Si $N \ge \frac{s}{p-1}$ alors il existe $g$ dans $R^{[u]}$ tel que $f=(1-p^{-s} \varphi)(g)$. 
\end{Lem}

Comme dans \cite[Lem. 3.2]{CN2017}, on utilise le lemme ci-dessus pour montrer les isomorphismes suivants : 

\begin{Lem}
\label{OmegaRcrisRu}
Soient $r$ dans $\N$ et $u$ et $u'$ des r\'eels tels que $\frac{1}{p-1} \le u \le 1$ et $u' \le u \le pu'$. Alors, 
\begin{enumerate}  
\item L'application $p^r- p^i\varphi$ induit un $p^{5r}$-isomorphisme $F^r \Omega^i_{R^{[u]}} / F^r \Omega^i_{R_{\cris}^{+}} \xrightarrow{\sim} \Omega^i_{R^{[u]}} / \Omega^i_{R_{\cris}^{+}}$. 
\item L'application $p^r- p^i\varphi$ induit un $p^{5r}$-isomorphisme $F^r \Omega^i_{R^{[u]}} / F^r \Omega^i_{R^{[u']}} \xrightarrow{\sim} \Omega^i_{R^{[u]}} / \Omega^i_{R^{[u']}}$.
\end{enumerate}
\end{Lem}

\begin{proof}
Comme on a $\varphi(\omega_{\textbf{j}})=  \omega_{\textbf{j}}$ pour $\textbf{j}$ dans $J_i$, il suffit de montrer que $p^r-p^i \varphi$ induit un $p^{4r}$-isomorphisme : $F^r R^{[u]} / F^r R_{\cris}^{+} \xrightarrow{\sim} R^{[u]} / R_{\cris}^{+}$ (respectivement $F^r R^{[u]} / F^r R^{[u']} \xrightarrow{\sim} R^{[u]} / R^{[u']}$). On note $A=R^{[u']}$ ou $R_{\cris}^+$ et $B= R^{[u]}$.

On montre d'abord la $p^r$-injectivit\'e. Soit $f$ dans $F^rB$ tel que $(p^r-p^i \varphi)(f)$ est dans $A$. Il suffit de voir que $p^r f$ est dans $A$ : c'est bien le cas car $p^r f =  (p^r-p^i \varphi)(f) + p^i \varphi(f)$ et $\varphi(B) \subseteq A$. 

Il reste \`a voir que l'application est $p^{5r}$-surjective. Soit $f$ dans $B$ : on peut \'ecrire $f=f_1+ f_2$ avec 
\[ f_1 = \sum_{n < N} x_n \frac{[\beta]^n}{p^n} \text{ et } f_2 =\sum_{n \ge N} x_n \frac{[\beta]^n}{p^n} \]
o\`u $N = \lfloor \frac{r-i}{p-1} \rfloor$ et $x_n \in R_{\inf}^+$ tend vers $0$ en l'infini. On a alors que $p^r f_1$ est dans $A$. Par le lemme pr\'ec\'edent, il existe $g$ dans $B$ tel que $f_2 = (1- p^{i-r}\varphi)(g)$. En utilisant la suite $p^{3r}$-exacte \eqref{suite exacte Au 1}: 
\[ 0 \to \Z_p(r) \to (\A^{[u]})^{\varphi=p^r} \to \A^{[u]}/F^r \to 0 \] 
on remarque que pour tout \'el\'ement $y$ de $\A^{[u]}$ on a $p^{3r}y= y_1+y_2$ avec $y_1$ dans $(\A^{[u]})^{\varphi=p^r}$ et $y_2$ dans $F^r \A^{[u]}$. Mais on a un $p^r$-isomorphisme : 
\[ (\AAC)^{\varphi=p^r} \xrightarrow{\sim}  (\A^{[u]})^{\varphi=p^r} \] 
et donc $p^r y_1$ est dans $\AAC$ (respectivement dans $\A^{[u']}$). 

Comme $R^{[u]} = \A^{[u]} \widehat{\otimes}_{\A_{\mathrm{inf}}} R_{\mathrm{inf}}^+$, on peut donc \'ecrire $p^{3r} g= g_1 + g_2$ avec $p^rg_1$ dans $A$ et $g_2$ dans $F^r B$. On obtient : 
\[ p^{5r} f = p^{5r} f_1 + (p^r-p^{i} \varphi)(p^r g_1 + p^r g_2) \]  
et donc, modulo $A$, on a 
$ p^{4r} f = (p^{r}-p^i \varphi)(p^rg_2) $, ce qui termine la d\'emonstration. 
 \end{proof}
 
 \begin{Cor}
Pour $u$ et $u'$ comme pr\'ec\'edemment, on a : 
 \begin{enumerate}
 \item L'injection $R_{\cris}^+ \subseteq R^{[u]}$ induit un $p^{10r}$-quasi-isomorphisme $C(R_{\cris}^+, r) \xrightarrow{\sim} C(R^{[u]}, r)$. 
 \item L'injection $R^{[u']} \subseteq R^{[u]}$ induit un $p^{10r}$-quasi-isomorphisme $C(R^{[u']},r) \xrightarrow{\sim} C(R^{[u]} , r)$.
 \end{enumerate}
 \end{Cor}
 
 \subsection{Passage de $\varphi$ \`a $\psi$}
 
 On rappelle qu'on a $C^{\psi}(R^{[u]},r):=[ F^r \Omega_{R^{[u]}}^{\bullet} \xrightarrow{p^r \psi - p^{\bullet}} \Omega_{R^{[pu]}}^{\bullet}]$.
 
 \begin{Lem}\cite[Lem. 3.4]{CN2017}
 \label{phi-psi Ru}
 On a un quasi-isomorphisme $C(R^{[u]},r) \xrightarrow{\sim} C^{\psi}(R^{[u]},r)$ donn\'e par le diagramme commutatif : 
\[  \xymatrix{ 
F^r \Omega_{R^{[u]}}^{\bullet} \ar[r]^{p^r-p^{\bullet} \varphi} \ar[d]^{\operatorname{Id}} & \Omega_{R^{[u]}}^{\bullet} \ar[d]^{\psi} \\
F^r \Omega_{R^{[u]}}^{\bullet} \ar[r]^{p^r \psi-p^{\bullet}} & \Omega_{R^{[pu]}}^{\bullet} 
.}\]
\end{Lem}

\begin{Rem}
\label{phi-psi Ruv}
La preuve du lemme est identique \`a celle donn\'ee dans \cite{CN2017}. On montre de la m\^eme fa\c{c}on que le diagramme commutatif : 
\[  \xymatrix{ 
F^r \Omega_{R^{[u,v]}}^{\bullet} \ar[r]^{p^r-p^{\bullet} \varphi} \ar[d]^{\operatorname{Id}} & \Omega_{R^{[u,\frac{v}{p}]}}^{\bullet} \ar[d]^{\psi} \\
F^r \Omega_{R^{[u,v]}}^{\bullet} \ar[r]^{p^r \psi-p^{\bullet}} & \Omega_{R^{[pu,v]}}^{\bullet} 
}\]
induit un quasi-isomorphisme $C(R^{[u,v]},r) \xrightarrow{\sim} C^{\psi}(R^{[u,v]},r)$.
\end{Rem}

\subsection{Anneaux de convergence}

On passe maintenant de l'anneau $R^{[u]}$ \`a $R^{[u,v]}$. 

\begin{Lem}
\label{SVSSS}
Soient $u$ et $v$ tels que $\frac{1}{p-1} \le u \le 1 \le v$. Alors on a un $p^r$-isomorphisme : 
$$ F^r R^{[u,v]} / F^r R^{[u]} \xrightarrow{\sim} R^{[u,v]} / R^{[u]}. $$
\end{Lem}

\begin{proof}
Comme on a pour tout $n \ge 1$, 
\[ \{ x \in \A^{[u,v]}/ \A^{[u]} \; | \; p^nx=0 \} = \{ x \in  \A^{(0,v]+}/ \A_{\inf} \; | \; p^nx=0 \} =0 \]
on obtient une suite exacte : 
\[ 0 \to \A^{[u]}/p^n \to \A^{[u,v]}/p^n \to (\A^{[u,v]}/ \A^{[u]})/p^n \to 0. \]
Pour tout $m \ge 1$, on a de plus que $R_{\inf}^+/p^m$ est plat sur $\A_{\inf}/p^m$ et donc la suite (on peut supposer $m \ge n$) 
\[ 0 \to \A^{[u]}/p^n \otimes_{\A_{\inf}/p^{m}} R^+_{\inf}/p^m \to \A^{[u,v]}/p^n \otimes_{\A_{\inf}/p^{m}} R^+_{\inf}/p^m \to (\A^{[u,v]}/ \A^{[u]})/p^n \otimes_{\A_{\inf}/p^{m}} R^+_{\inf}/p^m \to 0 \]
est exacte. On obtient finalement que, pour tout $n$ et $m$, la suite  
\[ 0 \to \A^{[u]}/p^n \otimes_{\A_{\inf}} R^+_{\inf}/p^m \to \A^{[u,v]}/p^n \otimes_{\A_{\inf}} R^+_{\inf}/p^m \to (\A^{[u,v]}/ \A^{[u]})/p^n \otimes_{\A_{\inf}} R^+_{\inf}/p^m \to 0 \]
est exacte. En prenant la limite sur $(n,m)$ et en utilisant que les syst\`emes projectifs v\'erifient la condition de Mittag-Leffler, on obtient une suite exacte : 
\[ 0 \to R^{[u]} \to R^{[u,v]} \to \A^{[u,v]}/ \A^{[u]} \widehat{\otimes}_{\A_{\inf}} R_{\inf}^+ \to 0. \]
Comme $F^rR^{?}= F^r\A^{?} \widehat{\otimes}_{\A_{\inf}} R_{\inf}^+$ pour $? \in \{[u], [u,v]\}$, le m\^eme raisonnement donne une suite exacte :  
\[ 0 \to F^rR^{[u]} \to F^rR^{[u,v]} \to F^r\A^{[u,v]}/ F^r\A^{[u]} \widehat{\otimes}_{\A_{\inf}} R_{\inf}^+ \to 0. \]
On obtient :
\[ F^r R^{[u,v]} / F^r R^{[u]}\cong F^r\A^{[u,v]}/ F^r \A^{[u]} \widehat{\otimes}_{\A_{\inf}} R_{\inf}^+ \text{ et } R^{[u,v]} / R^{[u]} \cong \A^{[u,v]}/ \A^{[u]} \widehat{\otimes}_{\A_{\inf}} R_{\inf}^+ \]
et il suffit de montrer le r\'esultat pour $R=\OC_C$. 

Comme $\A^{[u,v]} $ est la compl\'etion de  $\A_{\inf} \left[ \frac{p}{[\alpha]}, \frac{[\beta]}{p} \right]$ et $\A^{[u]}$ est celle de $\A_{\inf} \left[ \frac{[\beta]}{p} \right]$, il suffit de montrer que $p^r \frac{p}{[ \alpha ]}$ est dans l'image : c'est le cas puisqu'on a (voir la preuve de la proposition \ref{FrAuv})
$$ \frac{p}{[\alpha]} = \left(1+ \frac{[p^{\flat}]-p}{p} \right)^{-1} [ \alpha' ] \text{ avec } v_p( \alpha')= 1 - \frac{1}{v}. $$  
\end{proof}

On peut maintenant montrer la version g\'eom\'etrique du lemme 3.6 de \cite{CN2017}.

\begin{Lem}
\label{Ru to Ruv}
L'inclusion $R^{[u]} \hookrightarrow R^{[u,v]}$ induit un $p^{2r}$-quasi-isomorphisme : 
\[ \tau_{\le r} C^{\psi}(R^{[u]},r) \xrightarrow{\sim} \tau_{\le r} C^{\psi}(R^{[u,v]},r). \]
\end{Lem}

\begin{proof}
L'application est induite par 

\[ \xymatrix{ 
     F^r \Omega_{R^{[u]}}^{\bullet} \ar[d] \ar[r]^{p^r \psi - p^{\bullet}} & \Omega_{R^{[pu]}}^{\bullet} \ar[d] \\
     F^r \Omega_{R^{[u,v]}}^{\bullet} \ar[r]^{p^r \psi - p^{\bullet} } & \Omega_{{R}^{[pu,v]}}^{\bullet}. }\]
     
Pour montrer qu'on a un $p^{2r}$-quasi-isomorphisme $ \tau_{\le r} C^{\psi}(R^{[u]}, r) \xrightarrow{\sim} \tau_{\le r} C^{\psi}(R^{[u,v]} , r)$, il suffit de voir qu'on a un $p^{2r}$-quasi-isomorphisme : 
\[ \tau_{\le r} (F^r \Omega_{R^{[u,v]}}^{\bullet} /   F^r \Omega_{R^{[u]}}^{\bullet}) \xrightarrow{p^r \psi - p^{\bullet}} \tau_{\le r} (   \Omega_{R^{[pu,v]}}^{\bullet}/\Omega_{R^{[pu]}}^{\bullet}). \]

Pour simplifier, on note 
\[ A^{\bullet}= (F^r \Omega_{R^{[u,v]}}^{\bullet} /   F^r \Omega_{R^{[u]}}^{\bullet}) \text{ et } B^{\bullet}= (   \Omega_{R^{[pu,v]}}^{\bullet}/\Omega_{R^{[pu]}}^{\bullet}). \]

On va prouver :
\begin{enumerate}

\item pour tout $i \le r$, $F^r \Omega_{R^{[u,v]}}^{i} /   F^r \Omega_{R^{[u]}}^{i} \xrightarrow{p^r \psi - p^{i}}   \Omega_{R^{[pu,v]}}^{i}/\Omega_{R^{[pu]}}^{i}$ est un $p^{r}$-isomorphisme ; 

\item pour $i=r+1$, le morphisme $p^r \psi -p^{r+1} : H^{r+1}(A^{\bullet}) \to H^{r+1}(B^{\bullet})$ est $p^{2r}$-injectif.
\end{enumerate} 
 
Prouvons le point 1). Soit $i \le r$. Comme on a  
$ \psi( \sum_{\textbf{j} \in J_i} f_{\textbf{j}} \omega_{\textbf{j}}) = \sum_{\textbf{j} \in J_i}  \psi(f_{\textbf{j}}) \omega_{\textbf{j}} $, 
on se ram\`ene \`a prouver que $p^r \psi - p^i : F^rR^{[u,v]}/ F^rR^{ [u] } \to R^{ [pu, v]}/ R^{[ pu ]}$ est un $p^{r}$-isomorphisme. Par le lemme pr\'ec\'edent, il suffit de montrer que $p^r \psi - p^i : R^{[u,v]}/ R^{ [u] } \to R^{ [pu, v]}/ R^{[ pu ]}$ est un $p^{r}$-isomorphisme.  

Pour $i < r$, $p^s \psi-1$ avec $s=r-i$ est inversible d'inverse $-(1+ p^s \psi +  p^{2s} \psi^{2} + \cdots) $ et donc $p^r \psi-p^i$ est un $p^r$-isomorphisme. 
Il reste \`a voir le cas $i=r$ : on va montrer que $(\psi -1): R^{[u,v]}/ R^{ [u] } \to R^{ [pu, v]}/ R^{[ pu ]} $ est un isomorphisme.

Comme dans la preuve du lemme \ref{SVSSS}, on peut \'ecrire 
\[ R^{[u,v]}/ R^{[u]} = \A^{[u,v]}/ \A^{[u]} \widehat{\otimes}_{\A_{\inf}} R_{\inf}^+ \cong  \A^{(0,v]+}/ \A_{\inf} \widehat{\otimes}_{\A_{\inf}} R_{\inf}^+. \] 
Si $x$ est un mon\^ome $aX_1^{\alpha_1}\cdots X_d^{\alpha_d}$ de $R_{\inf}^+$ avec $a \in \A_{\inf}$ et $\alpha=(\alpha_1, \dots, \alpha_d) \neq 0$ alors, par construction de $\psi$, $(\psi^k(x))_{k}$ tend vers $0$ et la s\'erie $x+ \psi(x)+ \psi^2(x) +\cdots$ converge. Il suffit donc de v\'erifier que $\psi-1$ est un isomorphisme de $\A^{(0,v]+}/ \A_{\inf}$ dans lui-m\^eme. 

L'injectivit\'e se d\'eduit des suites exactes : 
\[ 0 \to \Z_p \to \A_{\inf} \xrightarrow{\psi-1} \A_{\inf} \to 0 \text{ et } 0 \to \Z_p \to \A^{(0,\frac{v}{p}]+} \xrightarrow{\psi-1} \A^{(0,v]+} \to 0 \]
obtenues \`a partir des suites exactes \eqref{suite exacte Ar} et \eqref{suite exacte Av}. 

Montrons la surjectivit\'e. Soit $x$ dans $\A^{(0,v]+}$. Comme $(1-\psi)$ est surjective de $\A$ dans $\A$, il existe $y$ dans $\A$ tel que $x=(1-\psi)(y)$. On va montrer que $\psi(y)$ est dans $\A^{(0,v]+}$ (et on aura en particulier que $y=x+ \psi(y)$ est dans $\A^{(0,v]+}$). 

\'Ecrivons $ x = \sum_{n \in \N} [x_n]p^n \text{ et } y= \sum_{n \in \N} [y_n]p^n$, on a l'\'egalit\'e 
\[ \sum_{n \in \N} [x_n]p^n = \sum_{n \in \N} ([y_n]-[y_n^{\frac{1}{p}}])p^n. \]

 On veut monter que $v \frac{v_{\E}(y_n)}{p}+n$ tend vers l'infini et que $v_{\E}(y_n) \ge -\frac{pn}{v}$. Il suffit de v\'erifier que $v_{\E}(y_n) \ge -\frac{n}{v}$ pour tout $n \in \N$ : on aura alors 
\[ v_{\E}(y_n) \ge \frac{-n}{v} > \frac{-pn}{v} \text{ et } v \frac{v_{\E}(y_n)}{p}+n \ge \frac{p-1}{p} n \xrightarrow[n \to \infty]{} \infty. \]

Pour $n=0$, on a $x_0=y_0-y_0^{\frac{1}{p}}$. Si $v_{\E}(y_0)<0$ alors $v_{\E}(y_0) < v_{\E}(y_0^{\frac{1}{p}})$ et $v_{\E}(y_0)=v_{\E}(x_0) \ge 0$, contradiction. Donc $v_{\E}(y_0) \ge 0$.  

Soit $n$ dans $\N$. Supposons que pour tout $i \le n$, $v_{\E}(y_n) \ge \frac{-n}{v}$. Dans $\A$, on a 
\[ \hspace{-1cm}-([x_{n+1}]-[y_{n+1}]+[y_{n+1}^\frac{1}{p}])p^{n+1}=\left(\sum_{i \le n} ([x_i]-[y_i]+[y_i]^{\frac{1}{p}})p^i\right)+ p^{n+2} \left(\sum_{i \ge n+2} ([x_i]-[y_i]+[y_i^{\frac{1}{p}}])p^{i-(n+2)}\right). \]  
En divisant par $p^{n+1}$ et en projetant dans $C^{\flat}$, on obtient que $x_{n+1}-y_{n+1}+y_{n+1}^{\frac{1}{p}}$ est l'image de 
\[ \frac{1}{p^{n+1}} \sum_{i \le n} ([x_i]-[y_i]+[y_i]^{\frac{1}{p}})p^i \]
dans $C^{\flat}$. Mais par hypoth\`ese de r\'ecurrence $[\alpha^n]( \sum_{i \le n} ([x_i]-[y_i]+[y_i^{\frac{1}{p}}]p^i) \in \A^+ \cap p^{n+1}\A=p^{n+1} \A^+$ donc $\alpha^{n}(x_{n+1}-y_{n+1}+y_{n+1}^{\frac{1}{p}})$ est dans $\OC^{\flat}$. On obtient 
\[ v_{\E}( x_{n+1}-y_{n+1}+y_{n+1}^{\frac{1}{p}}) \ge \frac{-n}{v}. \] 
Si $v_{\E}(y_{n+1}) \ge v_{\E}(x_{n+1})$ ou $v_{\E}(y_{n+1}) \ge 0$, on a fini. Si $v_{\E}(y_{n+1}) < v_{\E}(x_{n+1})$ et $v_{\E}(y_{n+1}) < 0$ alors \[v_{\E}(y_{n+1})= v_{\E}( x_{n+1}-y_{n+1}+y_{n+1}^{\frac{1}{p}}) \ge \frac{-n}{v} \ge \frac{-(n+1)}{v}. \] 

On obtient finalement que $\psi-1$ est un isomorphisme.

Montrons maintenant le point 2). On veut montrer que $\psi-p : H^{r+1}(A^{\bullet}) \to H^{r+1}(B^{\bullet})$ est $p^r$-injectif. Par le quasi-isomorphisme prouv\'e en \ref{phi-psi Ru}, il suffit de voir que le morphisme $1-p \varphi : H^{r+1}(A^{\bullet}) \to H^{r+1}(\widetilde{B}^{\bullet})$ avec 
\[ \widetilde{B}^{\bullet} = (\Omega_{R^{[u,\frac{v}{p}]}}^{\bullet}/\Omega_{R^{[u]}}^{\bullet}) \]
 est injectif. Ce morphisme est induit par le diagramme commutatif : 
\[ \xymatrix{ 
\cdots \ar[r] & A^r \ar[r] \ar[d]^{1-\varphi} & A^{r+1} \ar[r] \ar[d]^{1-p\varphi} & A^{r+2} \ar[r] \ar[d]^{1-p^2 \varphi}  \ar[r] & \cdots \\ 
\cdots \ar[r] & \widetilde{B}^r \ar[r] & \widetilde{B}^{r+1} \ar[r] & \widetilde{B}^{r+2} \ar[r] & \cdots
} \] 
Les fl\`eches $1 - p \varphi: A^{r+1} \to \widetilde{B}^{r+1}$ et $1 - p^2 \varphi : A^{r+2} \to \widetilde{B}^{r+2} $ sont bijectives (d'inverse $1+p \varphi + p^2 \varphi^2 + \cdots$ et $1+p^2 \varphi + p^4 \varphi^2 + \cdots$) donc il suffit de montrer que $1- \varphi : A^r \to \widetilde{B}^r$ est surjective. 

De la m\^eme fa\c{c}on que pr\'ec\'edemment, cela revient \`a montrer que 
\[ 1- \varphi : \A^{(0, v]^+}/ \A_{\inf} \to \A^{(0, \frac{v}{p}]^+}/ \A_{\inf}  \]
est surjective et c'est le r\'esultat de \eqref{suite exacte Av}. Ceci termine la d\'emonstration.   
\end{proof}

Le r\'esultat suivant se d\'eduit des quasi-isomorphismes des lemmes \ref{phi-psi Ru}, \ref{Ru to Ruv} et de la remarque \ref{phi-psi Ruv}.

\begin{Cor}
Si $pu \le v$, on a un $p^{2r}$-quasi-isomorphisme : 
$$ \tau_{\le r} C(R^{[u]}, r) \xrightarrow{\sim} \tau_{\le r} C(R^{[u,v]} , r).$$
\end{Cor} 

En combinant les r\'esultats de ces deux sections, on obtient un $p^{12r}$-quasi-isomorphisme : 
\begin{equation}
\label{Rcris-Ruv}
 \tau_{\le r} C(R_{\cris}^+,r) \to \tau_{\le r} C(R^{[u,v]}, r).
\end{equation}

\section{Utilisation des $(\varphi, \Gamma)$-modules}

On suppose toujours $R$ comme en \textsection 2.2.2, avec $\Spf(R)$ connexe. Dans cette section, on d\'efinit un isomorphisme $R^{[u,v]} \cong \A_R^{[u,v]}$ o\`u $\A_R^{[u,v]}$ est un anneau muni d'une action du groupe de Galois $G_R$. Cela va permettre de construire un quasi-isomorphisme entre le complexe $C(R^{[u,v]},r)$ de la section pr\'ec\'edente et un complexe de $(\varphi, \Gamma)$-modules $\Kos(\varphi, \Gamma_R, \A_R^{[u,v]})$ qui intervient dans le calcul de la cohomologie de Galois.

Comme dans \cite[\textsection 4]{CN2017}, la preuve se fait en deux \'etapes. Premi\`erement, en divisant par des puissances de $t$, on transforme l'action des diff\'erentielles $\partial_i$ en une action d'une alg\`ebre de Lie, $\Lie \Gamma_R$ : cela permet de se d\'ebarrasser de la filtration (c'est possible par le lemme \ref{filtration [u,v]} ci-dessous). On utilise ensuite que les op\'erateurs $\tau_i$ qui sont d\'efinis dans le paragraphe suivant et qui traduisent l'action du groupe $\Gamma_R$ sur l'anneau $\A_R^{[u,v]}$ sont topologiquement nilpotents (pour la topologie $\mu$-adique) pour passer de $\Lie \Gamma_R$ \`a $\Gamma_R$. 

\begin{Rem}
\begin{enumerate}
\item Dans \cite{CN2017}, Colmez-Nizio{\l} ne travaillent pas ici avec leur complexe original $\operatorname{Kum}(R_{\varpi}^{[u,v]},r)$ mais avec un complexe quasi-isomorphe $\operatorname{Cycl}(R_{\varpi}^{[u,v]},r)$. Ce changement de complexe correspond \`a un changement de la variable \emph{arithm\'etique} $X_0$ par une variable \emph{cyclotomique} $T$ sur laquelle on peut d\'efinir une action de $\Gamma_R$. Dans notre cas, les deux complexes sont confondus. 
\item Dans \cite{CN2017}, du fait de la variable suppl\'ementaire $T$, l'alg\`ebre de Lie obtenue n'est pas commutative. Ce n'est pas le cas ici.    
\end{enumerate}
\end{Rem}  
 
\subsection{Plongement dans les anneaux de p\'eriodes}

Pour chaque $i$ de $\{ 1, \dots , d \}$, on choisit un \'el\'ement $X_i^{\flat}=(X_i, X_i^{\frac{1}{p}}, \dots)$ dans $\E_{\overline{R}}$ et on d\'efinit un plongement de $R^{+}_{\inf, \square}$ dans $\A^+_{\overline{R}}$ en envoyant $X_i$ sur $[X_i^{\flat}]$. On \'etend le plongement \`a 
\[R_{\inf}^+ \to \A^+_{\overline{R}}, \; R_{\cris}^{+} \to  \AAC(\overline{R}), \; R^{[u]} \to  \A_{\overline{R}}^{[u]}, \; R^{[u,v]} \to  \A_{\overline{R}}^{[u,v]} \text{ et } R^{(0,v]+} \to \A_{\overline{R}}^{(0,v]+}. \]

 On note $\A_R$ (respectivement $\A_R^+$, $\AAC(R)$, $\A_R^{\star}$) l'image de $R_{\inf}$ (respectivement $R_{\inf}^+$, $R^+_{\cris}$, $R^{\star}$ pour $\star \in \{[u], [u,v], (0, v]+ \}$) par ce plongement. On peut alors d\'efinir une action du groupe $G_R$ sur ces anneaux. 

On consid\`ere :
\[R_m^{\square}:= \OC_C \{ X^{\frac{1}{p^m}}, \frac{1}{(X_1  \cdots X_a)^{\frac{1}{p^m}}}, \frac{ \varpi^{\frac{1}{p^m}}}{(X_{a+1} \cdots X_{a+b})^{\frac{1}{p^m}}} \} \]
et on note $R_{\infty}^{\square}$ la compl\'etion $p$-adique de $\varinjlim R_m^{\square}$. Soient $R_m$ et $R_{\infty}$ les compl\'et\'es $p$-adiques $R_m^{\square} \widehat{\otimes}_{\OC_C} R$ et $R_{\infty}^{\square} \widehat{\otimes}_{\OC_C} R$. 

On rappelle que $G_R:= \Gal(\overline{R}[\frac{1}{p}]/ R [ \frac{1}{p}])$. On note \[\Gamma_R := \mathrm{Gal}( R_{\infty}^{\square}[ \frac{1}{p} ] / R^{\square}[\frac{1}{p}])= \mathrm{Gal}(R_{\infty}[ \frac{1}{p} ] / R[\frac{1}{p}])\] le groupe des automorphismes de $R_{\infty}[ \frac{1}{p} ]$ qui fixent $R[\frac{1}{p}]$. On a $\Gamma_R \cong \Z_p^{d}$. 

De plus, comme $G_R$ agit sur $\A_{\overline{R}}$ (resp. $\AAC(\overline{R})$, $\A_{\overline{R}}^{\star}$), il agit sur $\A_R$ (resp. $\AAC(R), \A_R^{\star}$) via $\Gamma_R$. Si on choisit des g\'en\'erateurs topologiques $\gamma_1, \dots , \gamma_d$ de $\Gamma_R$, cette action est donn\'ee par 
\[ \gamma_k([X_k^{\flat}]) = [\varepsilon][X_k^{\flat}] \text{ et } \gamma_j([X_k^{\flat}])=[X_k^{\flat}] \text{ si } j \neq k. \]     

\begin{Rem}
\label{tauAuv}
On note $\tau_j:= \gamma_j -1$. Pr\'ecisons l'action des $\tau_j$ sur $\A_R^{+}$. Comme les $\gamma_j$ agissent trivialement sur $\AAinf$, on obtient que $\tau_j(R^+_{\inf, \square}) \subseteq \mu R^+_{\inf, \square}$. En utilisant la proposition \ref{fonctions implicites} pour $\lambda= \gamma_j$, $I= ( \mu)$ et $Z_{\lambda}=Z$, on obtient que $\gamma_j(Z)$ est dans $Z+ (\mu)$ et en utilisant l'isomorphisme $\A_R^{+} \cong R_{\inf}^+$, on en d\'eduit que $\tau_j(\A_R^{+}) \subseteq \mu \A_R^{+}$. 

Comme on a $R^{[u,v]}= \A^{[u,v]} \widehat{\otimes}_{\AAinf} R_{\inf}^+$, on a le m\^eme r\'esultat pour $\A_R^{[u,v]}$.   
\end{Rem}

\subsection{Passage des $(\varphi, \partial)$-modules aux $(\varphi, \Gamma)$-modules}

On commence par montrer le lemme suivant. Comme pour les preuves pr\'ec\'edentes, on ne dispose pas ici de l'interpr\'etation des anneaux $R^{[u]}$ et $R^{[u,v]}$ en anneaux de s\'eries de Laurent. La d\'emonstration se fait en travaillant directement avec les anneaux de p\'eriodes et en utilisant la description de la filtration $F^r\A^{[u,v]}$ donn\'ee en \ref{FrAuv}. 

\begin{Lem}
\label{filtration [u,v]}
Soit $\frac{v}{p} < 1 < v$ et $u\ge \frac{1}{p-1}$. Alors, l'application\footnote{bien d\'efinie car on a $t \in \AAC(\overline{R}) \subset \A_{\overline{R}}^{[u]}$.} $f \mapsto t^r f$ induit des $p^{3r}$-isomorphismes $\A_{\overline{R}}^{[u,v]} \to F^r \A_{\overline{R}}^{[u,v]}$ et $\A_{\overline{R}}^{[u,\frac{v}{p}]} \to \A_{\overline{R}}^{[u,\frac{v}{p}]}$.
\end{Lem}

\begin{proof}
Montrons d'abord la $p^{3r}$-surjectivit\'e de $t^r : \A_{\overline{R}}^{[u,v]} \to F^r \A_{\overline{R}}^{[u,v]}$. D'apr\`es la proposition \ref{FrAuv}, si $y$ est un \'el\'ement de $F^r \A_{\overline{R}}^{[u,v]}$, $p^ry$ s'\'ecrit $\xi^r x$ avec $x$ dans $
\A_{\overline{R}}^{[u,v]}$ : il suffit donc de voir que $t$ divise $p^2 \xi$ dans $\A_{\overline{R}}^{[u,v]}$. 
On \'ecrit 
$$ \mu = \widetilde{u_0} t \qquad \text{ avec } \qquad \widetilde{u_0}=  \sum_{n \ge 1} \frac{a(n)! p^{\tilde{a}(n)}}{n!} t^{\{n-1\}} \text{ o\`u } \begin{cases} \tilde{a}(n)=a(n) \text{ si } b(n) \neq 0 \\ \tilde{a}(n)=a(n)-1 \text{ si } b(n)=0\end{cases} $$
et on en d\'eduit que $\mu$ et $t$ engendrent les m\^emes id\'eaux dans $\AAC$ (voir \cite[5.2.4]{Fontaine94}). 
Il suffit alors de v\'erifier que $\mu$ divise $p^2 \xi$. Mais, par d\'efinition, $\xi = \frac{\mu}{\mu_1}$ avec $\mu_1= \varphi^{-1}(\mu)$ : on va montrer que $\frac{p^2}{\mu_1}$ est dans $\A^{[u,v]}$. 

Pour cela, on montre d'abord que $\mu = u_0 [ \bar{\mu} ]$ avec $u_0$ unit\'e de $\A_{\overline{R}}^{(0,\frac{v}{p}]+}$. On suit la preuve de Andreatta-Brinon dans \cite[4.3(d)]{AB08}.  
En utilisant que $\varphi(\mu)=(\mu+1)^p-1$, on v\'erifie qu'on a  : 
$$ \mu = [ \varepsilon ] -1 = [ \bar{\mu} ] + p [ \alpha_1 ] + p^2 [ \alpha_2 ] + \dots \text{ avec } v_{\E}( \alpha_n) \ge v_{\E}( \varepsilon^{\frac{1}{p^n}}-1) = \frac{1}{p^{n-1}(p-1)}. $$
Si on \'ecrit $\alpha_n = \bar{\mu} a_n$ avec $a_n$ dans $\E$, on a 
$$v_{\E}(a_n) \ge \frac{1-p^n}{p^{n-1}(p-1)}= - \sum_{k=0}^{n-1} p^{k-n+1}. $$
Donc pour tout $n \ge 1$, $\frac{v}{p} v_{\E}(a_n)+n \ge 0$. En notant $u_0= 1+p[a_1]+ p^2[a_2] + \dots$, on obtient $w_{\frac{v}{p}}(u_0-1) > 0$ donc $w_{\frac{v}{p}}(u_0)=0$ et $u_0$ est une unit\'e de $\A_{\overline{R}}^{(0,\frac{v}{p}]+}$. 

On montre de m\^eme que $\mu_1=u_1 [ \bar{\mu_1} ]$ avec $u_1$ unit\'e de $\A_{\overline{R}}^{(0,v]+}$ (et donc aussi de $\A_{\overline{R}}^{(0,\frac{v}{p}]+}$). On en d\'eduit
$$ \frac{p^2}{\mu_1} = \frac{p^2}{[\bar{\mu_1}]} u_1^{-1} \text{ avec } w_v(\frac{p^2}{[\bar{\mu_1}]})= \frac{- v}{(p-1)}+2 > \frac{-p}{p-1} +2 \ge 0 $$
et on obtient le r\'esultat voulu. 

La $p^{3r}$-surjectivit\'e de $\A_{\overline{R}}^{[u,\frac{v}{p}]} \to \A_{\overline{R}}^{[u,\frac{v}{p}]}$ s'obtient de la m\^eme fa\c con.

On v\'erifie ensuite $p^{3r}$-injectivit\'e. Soit $f$ dans $\A_{\overline{R}}^{[u,v]}$ tel que $t^r f =0$, on veut montrer $p^{3r} f=0$. Par ce qu'on a vu plus haut, $\mu^r f =0 = [\bar{\mu}]^r u_0^r f$ avec $u_0$ unit\'e de $\A_{\overline{R}}^{(0,\frac{v}{p}]+}$ donc $[\bar{\mu}]^rf=0$.
Enfin, on peut \'ecrire $p^2= [ \bar{\mu} ] \cdot x$ avec $x$ dans $\A_{\overline{R}}^{(0, \frac{v}{p}]+}$ : on a donc 
$$ p^{3r} f= (px [\bar{\mu}])^r f=0. $$

\end{proof} 

\begin{Rem} On suppose $\frac{v}{p} < 1 < v$ et $u\ge \frac{1}{p-1}$.
\label{injectivite r+1}
\begin{enumerate}
\item On en d\'eduit qu'on a un $p^{3r}$-isomorphisme $\A_R^{[u,v]} \xrightarrow{\sim} F^r \A_R^{[u,v]}$ (et $\A_{R}^{[u,\frac{v}{p}]} \xrightarrow{\sim}  \A_{R}^{[u,\frac{v}{p}]}$). 
\item La m\^eme preuve montre que le morphisme $f \mapsto t^{r+1} f$ de $\A_{\overline{R}}^{[u,v]} \to F^r \A_{\overline{R}}^{[u,v]}$ et $\A_{\overline{R}}^{[u,\frac{v}{p}]} \to \A_{\overline{R}}^{[u,\frac{v}{p}]}$ est $p^{2(r+1)}$-injectif.
\item On va montrer que les applications pr\'ec\'edentes induisent des $p^{6r}$-isomorphismes :
\[ \A_{\overline{R}}^{[u,v]}/p^n \to F^r \A_{\overline{R}}^{[u,v]}/p^n \text{ et } \A_{\overline{R}}^{[u, \frac{v}{p}]}/ p^n \to \A_{\overline{R}}^{[u, \frac{v}{p}]}/ p^n. \]
Pour cela, notons $A= \A_{\overline{R}}^{[u,v]}$ (respectivement $\A_{\overline{R}}^{[u, \frac{v}{p}]}$) et $B= F^r \A_{\overline{R}}^{[u,v]}$ (respectivement $\A_{\overline{R}}^{[u, \frac{v}{p}]}$) et montrons que $x \mapsto t^r x$ induit un $p^{6r}$-isomorphisme de $A/ p^n$ dans $B/p^n$. La surjectivit\'e d\'ecoule du lemme pr\'ec\'edent de mani\`ere \'evidente, montrons la $p^{6r}$-injectivit\'e : soit $x$ dans $A$ tel que $t^r x= p^n y$ pour $y$ dans $B$. On a ensuite $p^{3r} t^r x = p^n (p^{3r} y)= p^n (t^r z)$ avec $z$ dans $A$ (on utilise ici que la $p^{3r}$-surjectivit\'e). Ainsi $t^r(p^{3r}x - p^n z)$ est nul et on d\'eduit que $p^{3r}(p^{3r}x - p^n z)=0$. On obtient que $p^{6r}x$ est nul modulo $p^n A$ et donc que $t^r$ est $p^{6r}$-injective modulo $p^n$.     
\end{enumerate} 
\end{Rem}

Dans la suite, pour simplifier, on note $S=\A_R^{[u,v]}$ et $S'= \A_R^{[u, \frac{v}{p}]}$. On rappelle qu'on a choisi $(\gamma_j)_{1 \le j \le d}$ des g\'en\'erateurs topologiques de $\Gamma_R \cong \Z_p^d$ et, pour $1 \le i \le d$, on d\'efinit 
$\partial_i:= X_i \frac{\partial}{\partial X_i} \text{ et } J_i := \{ (j_1, \dots, j_i) \; | \; 1 \le j_1 \le \dots \le j_i \le d \}.$

On note $\chi : G_K \to \Z_p^{\times}$ le caract\`ere cyclotomique. Pour tout $g$ de $G_K$, on a 
$ g \cdot t = \chi(g) t$. Si on note $S(r)$ l'anneau $S$ muni de l'action de $G_K$ tordue par $\chi^r$, on obtient que la multiplication par $t^r$ induit une application Galois-invariante $S(r) \to S$. 
 
On d\'efinit les complexes : 
\[ \Kos( \Gamma_R, S(r) ):= S(r) \xrightarrow{(\gamma_j-1)} S(r)^{J_1} \to \dots \to S(r)^{J_d}, \] 
\[ \Kos( \varphi, \Gamma_R, S(r)):= [ \mathrm{Kos}( \Gamma_R, S(r) ) \xrightarrow{1 - \varphi} \mathrm{Kos}( \Gamma_R , S'(r) ) ], \]   
\[ \Kos( \partial, F^rS ):= F^rS \xrightarrow{(\partial_j)} (F^{r-1}S)^{J_1} \to \dots \to (F^{r-d}S)^{J_d}, \] 
\[ \Kos( \varphi, \partial, F^rS):= [ \mathrm{Kos}( \partial, F^rS ) \xrightarrow{p^r - p^{\bullet}\varphi} \mathrm{Kos}( \partial , S') ]. \] 

On a $C(S,  r) \xrightarrow{\sim} \Kos(\varphi, \partial, S)$ et en utilisant l'isomorphisme $R^{[u,v]} \xrightarrow{\sim} A_R^{[u,v]}$, on obtient 
$ C(S,r) \xrightarrow{\sim} \Kos(\varphi, \partial, F^rS).$

Le but est maintenant de prouver la proposition suivante (voir \cite[Prop. 4.2]{CN2017} pour l'analogue arithm\'etique) : 

\begin{Pro}
\label{(phi,gamma) et (phi,delta)}
\begin{enumerate}
\item Il existe un $p^{30r}$-quasi-isomorphisme \[\tau_{\le r} \Kos(\varphi, \Gamma_R, S(r)) \xrightarrow{\sim} \tau_{\le r} \Kos(\varphi, \partial, F^rS).\]
\item Il existe un $p^{58r}$-quasi-isomorphisme \[\tau_{\le r} \Kos(\varphi, \Gamma_R, S(r))_n \xrightarrow{\sim} \tau_{\le r} \Kos(\varphi, \partial, F^rS)_n\] o\`u $(.)_n$ d\'esigne la r\'eduction modulo $p^n$. 
\end{enumerate}
\end{Pro}
 
Pour montrer cela, on d\'efinit $\nabla_j := t \partial_j$ et on consid\`ere l'alg\`ebre de Lie associ\'ee au groupe $\Gamma_R$, $\Lie \Gamma_R$. Alors $\Lie \Gamma_R$ est un $\Z_p$-module libre de rang $d$, engendr\'ee par les $\nabla_j$ pour $1 \le j \le d$. On note : 
   
$$\Kos( \Lie \Gamma_R, S(r) ):= S(r) \xrightarrow{(\nabla_j)} S(r)^{J_1} \to \dots \to S(r)^{J_d},$$
$$ \Kos( \varphi, \Lie\Gamma_R, S(r)):= [ \Kos( \Lie\Gamma_R, S(r) ) \xrightarrow{1 - \varphi} \Kos( \Lie \Gamma_R , S'(r) ) ].$$

\begin{Lem}
\label{Lie to Delta}
\begin{enumerate}
\item Il existe un $p^{30r}$-quasi-isomorphisme \[ \tau_{\le r} \Kos( \varphi, \Lie \Gamma_R, S(r) ) \xrightarrow{\sim}  \tau_{\le r} \Kos(\varphi, \partial, F^rS). \] 
\item De m\^eme, il existe un $p^{58r}$-quasi-isomorphisme 
\[ \tau_{\le r} \Kos( \varphi, \Lie \Gamma_R, S(r) )_n \xrightarrow{\sim}  \tau_{\le r} \Kos(\varphi, \partial, F^rS)_n. \]
\end{enumerate}
\end{Lem}

\begin{proof}
On rappelle qu'on note $S= \A_R^{[u,v]}$ et $S'= \A_R^{[u, \frac{v}{p}]}$. Comme dans \cite[Lem. 4.4]{CN2017}, on d\'eduit du lemme \ref{filtration [u,v]} et des diagrammes :
\[
\xymatrix{
S(r) \ar[d]_{t^r}  \ar[r]^{(\nabla_j)} & S(r)^{J_1} \ar[d]_{t^r} \ar[r] & \dots  \ar[r] &  S(r)^{J_r} \ar[d]_{t^r} \ar[r] & S(r)^{J_{r+1}} \ar[d]_{t^r} \ar[r] & \dots \\
F^rS \ar[r]^{(\nabla_j)} & (F^{r}S)^{J_1} \ar[r] & \dots \ar[r] & (F^{r}S)^{J_r} \ar[r] & (F^r S)^{J_{r+1}} \ar[r] & \dots \\
F^rS \ar[u]_{t^0} \ar[r]^{(\partial_j)} & (F^{r-1}S)^{J_1} \ar[u]_{t^1} \ar[r] & \dots \ar[r] & S^{J_r} \ar[u]_{t^r} \ar[r] & S^{J_{r+1}} \ar[u]_{t^{r+1}} \ar[r] & \dots  
}
\]
et
\[
\xymatrix{
S'(r) \ar[d]_{t^r}  \ar[r]^{(\nabla_j)} & S'(r)^{J_1} \ar[d]_{t^r} \ar[r] & \dots  \ar[r] &  S'(r)^{J_r} \ar[d]_{t^r} \ar[r] & S'(r)^{J_{r+1}} \ar[d]_{t^r} \ar[r] & \dots \\
S' \ar[r]^{(\nabla_j)} & (S')^{J_1} \ar[r] & \dots \ar[r] & (S')^{J_r} \ar[r] & (S')^{J_{r+1}} \ar[r] & \dots \\
S' \ar[u]_{t^0} \ar[r]^{(\partial_j)} & (S')^{J_1} \ar[u]_{t^1} \ar[r] & \dots \ar[r] & (S')^{J_r} \ar[u]_{t^r} \ar[r] & (S')^{J_{r+1}} \ar[u]_{t^{r+1}} \ar[r] & \dots  
}
\] 
qu'on a des $p^{14r}$-quasi-isomorphismes  
\[ \tau_{\le r} \mathrm{Kos}( \Lie \Gamma_R, S(r) ) \xrightarrow{\sim} \tau_{\le r} \mathrm{Kos}(\partial, F^rS) \text{ et } \tau_{\le r} \mathrm{Kos}( \Lie \Gamma_R, S'(r) ) \xrightarrow{\sim}  \tau_{\le r} \mathrm{Kos}(\partial, S'). \] 
 
 En effet, par le lemme, les fl\`eches verticales du haut sont des $p^{3r}$-isomorphismes et on obtient donc un $p^{6r}$-quasi-isomorphisme entre les deux premiers complexes. De la m\^eme fa\c con, les fl\`eches du bas sont des $p^{3r}$-isomorphismes en degr\'e inf\'erieur \`a $r$. En degr\'e $(r+1)$, d'apr\`es la remarque \ref{injectivite r+1}, le morphisme $S \xrightarrow{t^{r+1}} F^rS$ est $p^{4r}$-injectif et on en d\'eduit le $p^{8r}$-quasi-isomorphisme entre les complexes tronqu\'es. 
Comme on a : 
\[ \Kos( \varphi, \Lie\Gamma_R, S(r)):= [ \Kos( \Lie\Gamma_R, S(r) ) \xrightarrow{1 - \varphi} \Kos( \Lie \Gamma_R , S'(r) ) ] \]
\[ \Kos( \varphi, \partial, F^rS):= [ \Kos( \partial, F^rS ) \xrightarrow{p^r - p^{\bullet} \varphi} \Kos( \partial , S') ] \]  
on obtient un $p^{30r}$-quasi-isomorphisme 
\[ \tau_{\le r} \Kos( \varphi, \Lie \Gamma_R, S(r) ) \xrightarrow{\sim}  \tau_{\le r} \Kos(\varphi, \partial, F^rS) \]
induit par le diagramme commutatif : 
\begin{equation}
\label{DNote}
 \xymatrix{
 \Kos( \Lie \Gamma_R, S(r)) \ar[r]^{1-\varphi} \ar[d] & \Kos(\Lie \Gamma_R, S'(r)) \ar[d]^{p^r} \\
 \Kos( \Lie \Gamma_R, S(r)) \ar[r]^{p^r(1-\varphi)} \ar[d]^{t^r}  & \Kos(\Lie \Gamma_R, S'(r)) \ar[d]^{t^r}  \\
 \Kos( \Lie \Gamma_R, F^rS)) \ar[r]^{p^r-\varphi}  & \Kos(\Lie \Gamma_R, S') \\
 \Kos(\partial, F^rS) \ar[u]^{t^{\bullet}} \ar[r]^{p^r-p^{\bullet}\varphi} & \ar[u]^{t^{\bullet}} \Kos(\partial, S').
 }
 \end{equation}    

Le r\'esultat modulo $p^n$ s'obtient de la m\^eme fa\c{c}on, en utilisant le dernier point de la remarque \ref{injectivite r+1}. 
\end{proof}

\begin{Rem}
\label{ExCps}
En degr\'e sup\'erieur \`a $r+1$, les fl\`eches du bas dans le diagramme ci-dessus ne sont plus surjectives et, sans la troncation, on perd le quasi-isomorphisme entre la deuxi\`eme et la troisi\`eme ligne.
\end{Rem}

\begin{Lem}
\label{Gamma to Lie}
Il existe un quasi-isomorphisme : \[ \Kos(\varphi, \Gamma_R, S(r)) \xrightarrow{\sim} \Kos( \varphi, \Lie \Gamma_R, S(r) ). \]
\end{Lem}

\begin{proof}
La preuve est semblable \`a celle de \cite[Prop. 4.5]{CN2017}. On va construire une application $\beta :  \mathrm{Kos}(\Gamma_R, S(r)) \to \mathrm{Kos}( \Lie \Gamma_R, S(r) )$ telle que le diagramme suivant soit commutatif : 
 $$ 
 \xymatrix{ 
 S(r) \ar[r]^{(\gamma_j-1)} \ar[d]_{\mathrm{Id}} & S(r)^{J_1} \ar[r] \ar[d]_{ \beta_1} & S(r)^{J_2} \ar[r] \ar[d]_{ \beta_2} & \dots \\
S(r) \ar[r]^{(\nabla_j)}  & S(r)^{J_1} \ar[r]  & S(r)^{J_2} \ar[r]& \dots
}
$$
et qui induit le quasi-isomorphisme voulu. 

On rappelle qu'on a not\'e $\tau_j:= \gamma_j-1$. 

Soient $(a_n)_{n \ge 1}$ et $(b_n)_{n \ge 1}$ les coefficients des s\'eries formelles : 
\[ \frac{\log(1+X)}{X}= 1+ a_1 X+ a_2 X^2+ \cdots  \text{ et }  \frac{X}{\log(1+X)}= 1 + b_1 X+ b_2 X^2 + \cdots \]
Pour tout $j$ de $\{ 1, \dots,  d \}$, on pose :  
\[ s_j := 1+ a_1 \tau_j + a_2 \tau_j^2 + \cdots \text{ et } s_j^{-1} := 1+ b_1 \tau_j + b_2 \tau_j^2 + \cdots \]
En utilisant la remarque \ref{tauAuv}, on voit que pour $x$ dans $\A_R^{[u,v]}$, les s\'eries $s_j(x)$ et $s_j^{-1}(x)$ convergent dans $\A_R^{[u,v]}$.    

Consid\'erons maintenant les applications $\beta_i : S(r)^{J_i} \to S(r)^{J_i}$ avec $\beta_i( (a_{\textbf{j}})_{ \textbf{j} \in J_i} ) = ( s_{j_1} \cdots s_{j_i} ( a_{ \textbf{j}}))_{\textbf j \in J_i} $. Par ce qu'on vient de voir, les $\beta_i$ sont bien d\'efinies et sont des isomorphismes. 

Il reste \`a voir que le diagramme commute. Mais on a 
$$ (s_{j_1} \cdots s_{j_i} \tau_{j_i} \cdots \tau_{j_1})= (s_{j_1} \cdots s_{j_{i-1}} \nabla_{j_i} \tau_{j_{i-1}} \cdots \tau_{j_1})= \dots = (\nabla_{j_1} \cdots \nabla_{j_i})$$
car $\nabla_j$ et $\tau_j$ commutent (comme on a tordu l'action par $\chi$). 

Enfin, en remarquant que $\beta \circ \varphi = \varphi \circ \beta$, on obtient que l'application $\beta$ nous donne bien le quasi-isomorphisme cherch\'e.
\end{proof}

\begin{Rem}
La convergence des s\'eries $s_j$ et $s_j^{-1}$ se d\'eduit ici directement de la d\'efinition de l'action des $\tau_j$ sur $\A_R^{[u,v]}$. Dans \cite{CN2017}, du fait de la variable suppl\'ementaire $T$, la preuve de cette convergence est plus technique et s'obtient, l\`a encore, par des consid\'erations sur des s\'eries de Laurent (voir \cite[ \textsection 2.5.3]{CN2017}). 
\end{Rem}

La proposition \ref{(phi,gamma) et (phi,delta)} se d\'eduit ensuite des deux pr\'ec\'edents lemmes.

\subsection{Changement d'anneau}

Pour terminer, on passe de l'anneau $\A_R^{[u,v]}$ \`a $\A_R$. 

\begin{Lem}
\label{Av to Auv1}
L'inclusion $\A_R^{(0,v]+} \hookrightarrow \A_R^{[u,v]}$ induit un quasi-isomorphisme : 
\[  \Kos(\varphi, \Gamma_R, \A_{R}^{(0,v]+}(r)) \xrightarrow{\sim} \Kos( \varphi,  \Gamma_R, \A_R^{[u,v]}(r) ).\]
\end{Lem}

\begin{proof}
L'id\'ee est la m\^eme que dans \cite[Lem. 4.8]{CN2017} : on v\'erifie que $1- \varphi : \A_R^{[u,v]}/ \A_R^{(0,v]+} \to \A_R^{[u,\frac{v}{p}]}/ \A_R^{(0,\frac{v}{p}]+}$ est un isomorphisme. Comme on a un isomorphisme entre $\A_R^{[u,v]}/ \A_R^{(0,v]+}$ et  $\A_R^{[u,\frac{v}{p}]}/ \A_R^{(0,\frac{v}{p}]+}$, on peut voir $1- \varphi$ comme un endomorphisme de $\A_R^{[u,v]}/ \A_R^{(0,v]+}$. 

On a de plus :  
\[ \varphi \left( \frac{[\beta]^n}{p^n} \right) = \frac{[\beta]^{pn}}{p^n} =p^{(p-1)n} \frac{[\beta]^{pn}}{p^{pn}}. \] 
On en d\'eduit que $\varphi( \A_R^{[u,v]}/ \A_R^{(0,v]+}) \subseteq p \cdot (\A_R^{[u,v]}/ \A_R^{(0,v]+})$ et par it\'eration,  
\[ \varphi^k( \A_R^{[u,v]}/ \A_R^{(0,v]+}) \subseteq p^k \cdot (\A_R^{[u,v]}/ \A_R^{(0,v]+}). \]
 On obtient que $\varphi$ est topologiquement nilpotent et $(1- \varphi)$ inversible. 
\end{proof}

Pour $S= \A_R, \A_R^{(0,v]+}$, on note
 $ \Kos(\psi, \Gamma_R, S):= [ \Kos(\Gamma_R, S) \xrightarrow{\psi-1} \Kos(\Gamma_R, S)].$
 
 \begin{Rem}
 Dans \cite{CN2017}, l'anneau $\A_R^{(0,v]+}$ n'est pas stable par $\psi$ et il est n\'ecessaire de multiplier $\A_R^{(0,v]+}$ par une constante $\pi_i^{-l}$ dans la d\'efinition de $ \Kos(\psi, \Gamma_R, \A_R^{(0,v]+})$. 
 \end{Rem}

\begin{Lem}
\label{phi-psi Av}
L'application 
\[ \xymatrix{
\Kos(\Gamma_R, \A_R^{(0,v]+}) \ar[r]^{1- \varphi} \ar[d]^{\mathrm{Id}} & \Kos(\Gamma_R, \A_R^{(0, \frac{v}{p}]+}) \ar[d]^{\psi} \\
\Kos(\Gamma_R, \A_R^{(0,v]+}) \ar[r]^{\psi-1} & \Kos(\Gamma_R, \A_R^{(0, v]+}) 
} \] 
induit un quasi-isomorphisme $\Kos( \varphi, \Gamma_R, S) \xrightarrow{\sim} \Kos( \psi, \Gamma_R, S)$.
\end{Lem}

\begin{proof}
Le raisonnement est semblable \`a celui de la preuve du lemme \ref{phi-psi Ru} : comme $\psi$ est surjective, il suffit de v\'erifier que $\Kos(\Gamma_R, \A_R^{(0, \frac{v}{p}]+})^{\psi=0}$ est acyclique. On utilise la d\'ecomposition de la remarque \ref{psi-decompo} et on est ramen\'e \`a prouver que $\Kos(\Gamma_R, \varphi(\A_R^{(0, \frac{v}{p}]+})u_{\alpha})$ est acyclique pour tout $\alpha=(\alpha_1, \dots, \alpha_d) \neq 0$. 

Soit alors $k$ tel que $\alpha_k \neq 0$. On peut supposer que $k=d$. Pour simplifier, on pose $M_{\alpha}= \varphi( \A_R^{(0,v]+}) u_{\alpha}$. On note $J'_i:= \{ (j_1, \dots, j_i) \; | \; 1 \le j_1 < \dots < j_i \le d-1  \}$ et on \'ecrit $\Kos(\Gamma_R, M_{\alpha})$ comme le complexe : 
\[ \xymatrix{ 
M_{\alpha} \ar[r]^-{(\tau_j)_{j \neq d}} \ar[d]^{\tau_d} & (M_{\alpha})^{J'_1} \ar[r] \ar[d]^{\tau_d} & (M_{\alpha})^{J'_2} \ar[r] \ar[d]^{\tau_d} & \dots \\
M_{\alpha} \ar[r]^-{(\tau_j)_{j \neq d}}  & (M_{\alpha})^{J'_1} \ar[r]  & (M_{\alpha})^{J'_2} \ar[r]  & \dots 
} \]
On va montrer que $\tau_d$ est bijective sur $M_{\alpha}$. Comme $M_{\alpha}$ est $p$-adiquement complet, il suffit de montrer la surjectivit\'e modulo $p$. 

On a $ \gamma_d \cdot u_{\alpha}= (\bar{\mu}+1)^{\alpha_d} u_{\alpha}$ et donc, pour $y$ dans $(\A_R^{(0,v]+}/p)$ : 
\[ (\gamma_d -1) \cdot ( \varphi(y)u_{\alpha}) = \varphi( \bar{\mu}_1 G(y) )u_{\alpha} \]
avec $\bar{\mu}_1= \varphi^{-1}(\bar{\mu})$ et $G(y) = (1+ \bar{\mu}_1)^{\alpha_d} \bar{\mu}_1^{-1} ( \gamma_d-1)y + \bar{\mu}_1^{-1}( ( 1+ \bar{\mu}_1)^{\alpha_d} -1)y$. Mais l'action de $(\gamma_d-1)$ est triviale modulo $\bar{\mu}$ (car, modulo $\bar{\mu}$, $\varepsilon =1$) et comme $\bar{\mu}= \bar{\mu}_1^p$, on obtient que, modulo $\bar{\mu}_1$, $G(y)= \alpha_dy$. On a alors que l'application $\varphi \circ G: (\A_R^{(0,v]+}/p) \to (\A_R^{(0,\frac{v}{p}]+}/p)$ est surjective modulo $\bar{\mu}$ et comme $(\A_R^{(0,v]+}/p)$ est $\bar{\mu}$-complet, on en d\'eduit que $\varphi \circ G$ est surjective et que $(\gamma_d-1)$ est surjective sur $M_{\alpha}$.
 
\end{proof}

\begin{Lem}
L'inclusion $\A_R^{(0,v]+} \hookrightarrow \A_R$ induit un quasi-isomorphisme 
\[  \Kos(\psi, \Gamma_R, \A_{R}^{(0,v]+}) \xrightarrow{\sim} \Kos( \psi,  \Gamma_R, \A_R ). \]
\end{Lem}

\begin{proof}
Comme dans \cite[Lem. 4.12]{CN2017}, il suffit de voir que $1- \psi : \A_R / \A_R^{(0,v]+} \to  \A_R / \A_R^{(0,v]+}$ est un isomorphisme. De la m\^eme fa\c{c}on que dans la preuve du lemme \ref{Ru to Ruv}, en remarquant que $1+ \psi + \psi^2+ \cdots$ converge sur les mon\^omes, on se ram\`ene \`a montrer le r\'esultat pour l'anneau $\A / \A^{(0,v]+}$. Mais on a montr\'e en \ref{Ru to Ruv} que $1-\psi$ \'etait surjective de $\A^{(0,v]+}$ dans lui-m\^eme et on obtient donc que $1-\psi$ est un isomorphisme sur le quotient.  
\end{proof}

On en d\'eduit la proposition suivante : 

\begin{Pro}
\label{Ar to Av}
On a un quasi-isomorphisme 
\[  \Kos(\varphi, \Gamma_R, \A_{R}^{(0,v]+}) \xrightarrow{\sim} \Kos( \varphi,  \Gamma_R, \A_R ). \]
\end{Pro}

Enfin, en combinant tous les r\'esultats de cette section, on obtient :

\begin{Pro}
\label{Auv to Ar}
Il existe un quasi-isomorphisme : 
\[  \Kos(\varphi, \Gamma_R, \A_{R}^{[u,v]}(r)) \xrightarrow{\sim} \Kos( \varphi,  \Gamma_R, \A_R(r) ). \]
\end{Pro}

\begin{Rem}
\label{Ar+ to Auv}
On peut montrer \'egalement que l'inclusion $\A_R^+ \hookrightarrow \A_R^{[u,v]}$ induit un quasi-isomorphisme $\Kos(\varphi,\Gamma_R, \A_{R}^{+}(r)) \xrightarrow{\sim} \Kos( \varphi,  \Gamma_R, \A_R^{[u,v]}(r) )$. En effet, par ce qu'on vient de voir, il reste \`a v\'erifier que l'inclusion $\A_R^+ \hookrightarrow \A_R^{(0,v]+}$ donne un quasi-isomorphisme 
$ \Kos(\psi, \Gamma_R, \A_{R}^{+}) \xrightarrow{\sim} \Kos( \psi,  \Gamma_R, \A_R^{(0,v]+} )$. 
Pour cela, il suffit de prouver que $1- \psi : \A_R^{(0,v]+} / \A^+_R \to  \A_R^{(0,v]+} / \A^+_R$ est un isomorphisme, ce qui est montr\'e dans la preuve du lemme \ref{Ru to Ruv}.
\end{Rem}

\section{Cohomologie de Galois}

Soit $R$ comme dans les sections pr\'ec\'edentes. Pour terminer la preuve du th\'eor\`eme de comparaison locale \ref{syn/galois}, il reste \`a voir qu'on a un quasi-isomorphisme $R\Gamma(\Gamma_R, \A_R) \cong R\Gamma(G_R, \A_{\overline{R}})$. Ce r\'esultat s'obtient via des arguments classiques de descente presque \'etale et de d\'ecompl\'etion.  

\subsection{Calcul de la cohomologie de Galois}

Le but de cette partie est de montrer qu'on a un quasi-isomorphisme : 
\begin{equation}
\label{qis galois}
[R \Gamma(\Gamma_R, \A_R(r)) \xrightarrow{1- \varphi} R \Gamma(\Gamma_R, \A_R(r))] \xleftarrow{\sim} R \Gamma(G_R, \Z_p(r)).
\end{equation} 

Comme la suite exacte 
\[ 0 \to \Z_p \to \A_{\overline{R}} \xrightarrow{1- \varphi} \A_{\overline{R}} \to 0 \]
donne un quasi-isomorphisme 
\[ R \Gamma(G_R, \Z_p(r)) \xrightarrow{\sim} [R \Gamma( G_R, \A_{\overline{R}}(r) ) \xrightarrow{1 - \varphi} R \Gamma( G_R, \A_{\overline{R}}(r) )], \]
il suffit de v\'erifier qu'on a un quasi-isomorphisme : 
\[ R \Gamma( \Gamma_R, \A_R(r)) \xrightarrow{\sim} R \Gamma(G_R, \A_{\overline{R}}(r)). \] 

Soit $\widetilde{R}_{\infty}$ la compl\'etion de la cl\^oture int\'egrale de $R_{\infty}$ dans la sous-$R_{\infty}[\frac{1}{p}]$-alg\`ebre de $\widehat{\overline{R}}[\frac{1}{p}]$ engendr\'ee par les $X^{\frac{1}{m}}_{a+b+1}, \dots, X_d^{\frac{1}{m}}$ pour $m \ge 1$. La m\^eme preuve que dans \cite[Lem. 5.8]{CN2017} donne :  

\begin{Pro}
\label{max etale}
$\overline{R}$ est l'extension maximale de $\widetilde{R}_{\infty}$ telle que $\overline{R}[\frac{1}{p}]/\widetilde{R}_{\infty}[\frac{1}{p}]$ est \'etale.
\end{Pro}

On note $\E_{\widetilde{R}_{\infty}}$ le tilt de $\widetilde{R}_{\infty}$, $\A_{\widetilde{R}_{\infty}}= W(\E_{\widetilde{R}_{\infty}})$ et $\widetilde{\Gamma}_R:= \Gal( \widetilde{R}_{\infty}[ \frac{1}{p}] / R [ \frac{1}{p}])$. Enfin, soit $H_R:= \ker( G_R \to \widetilde{\Gamma}_R)$. On d\'eduit de la proposition pr\'ec\'edente le r\'esultat suivant (le raisonnement est identique \`a celui de \cite[\textsection 2]{AB08}, \cite[\textsection 7]{AI08}, \cite{Colmez2003}) : 

\begin{Pro} \cite[4.13]{CN2017}
\label{muH}
\begin{enumerate}
\item Pour tout $i \ge 1$, on a $H^i( H_R, \A_{\overline{R}})=0$. En particulier, on a un quasi-isomorphisme : 
\[ R \Gamma( \widetilde{\Gamma}_R , \A_{\widetilde{R}_{\infty}}(r)) \xrightarrow{\sim} R \Gamma(G_R, \A_{\overline{R}}(r)). \]
\item Il existe un quasi-isomorphisme : 
\[ R \Gamma ( \Gamma_R , \A_{R_{\infty}}(r)) \xrightarrow{\sim} R \Gamma( \widetilde{\Gamma}_R , \A_{\widetilde{R}_{\infty}}(r)) .\] 
\end{enumerate}
\end{Pro}

Dans la suite on note $\mu_H$ le morphisme $R \Gamma( \Gamma_R , \A_{R_{\infty}}(r)) \to R \Gamma(G_R, \A_{\overline{R}}(r)). $
On passe de l'anneau $\A_{R_{\infty}}$ \`a $\A_R$ via un argument de d\'ecompl\'etion (voir \cite[7.16]{AI08}, \cite{KL16}) : 

\begin{Pro}\cite[4.13]{CN2017}
\label{muinfini}
On a un quasi-isomorphisme : 
$$ \mu_{\infty} : R\Gamma( \Gamma_R, \A_R(r) ) \xrightarrow{\sim} R\Gamma( \Gamma_R, \A_{R_{\infty}}(r) ). $$
\end{Pro}

\subsection{Comparaison avec le complexe syntomique}

On peut maintenant montrer le th\'eor\`eme :

\begin{The}
\label{Qis syn Galois}
Il existe une constante $N$ qui ne d\'epend pas de $R$ telle qu'on a des $p^{Nr}$-quasi-isomorphismes :
$$ \alpha_r : \tau_{\le r} \mathrm{Syn}(R,r) \to \tau_{\le r} R \Gamma(G_R, \Z_p(r)) $$
$$ \alpha_{r,n} : \tau_{\le r} \mathrm{Syn}(R,r)_n \to \tau_{\le r} R \Gamma(G_R, \Z/p^n(r)). $$
\end{The}

\begin{proof}
On vient de voir qu'on avait un quasi-isomorphisme 
$$[R \Gamma(\Gamma_R, \A_R(r)) \xrightarrow{1- \varphi} R \Gamma(\Gamma_R, \A_R(r))] \xleftarrow{\sim} R \Gamma(G_R, \Z_p(r)).$$
Mais $R \Gamma(\Gamma_R, \A_R(r))$ est calcul\'e par le complexe $\Kos(\Gamma_R, \A_R(r))$ et on a un quasi-isomorphisme naturel : 
$$R \Gamma(G_R, \Z_p(r)) \xrightarrow{\sim} \Kos(\varphi, \Gamma_R, \A_R(r)). $$
D'apr\`es la proposition \ref{(phi,gamma) et (phi,delta)}, on a un $p^{30r}$-quasi-isomorphisme :  
$$\tau_{\le r} \Kos(\varphi, \Gamma_R, \A_R^{[u,v]}(r)) \xrightarrow{\sim} \tau_{\le r} \Kos(\varphi, \partial, F^r\A_R^{[u,v]})$$
et en utilisant le quasi-isomorphisme du lemme \ref{Auv to Ar} :
$$  \Kos(\varphi, \Gamma_R, \A_R^{[u,v]}(r)) \xrightarrow{\sim} \Kos( \varphi,  \Gamma_R, \A_R(r) ), $$
on en d\'eduit un $p^{30r}$-quasi-isomorphisme : 
$$\tau_{\le r} R \Gamma(G_R, \Z_p(r)) \xrightarrow{\sim} \tau_{\le r} \Kos(\varphi, \partial, F^r\A_R^{[u,v]}). $$
Enfin, en utilisant l'isomorphisme naturel $\A_R^{[u,v]} \cong R^{[u,v]}$ et le $p^{12r}$-quasi-isomorphisme $\Syn(R_{\cris}^+,r) \xrightarrow{\sim} C(R^{[u,v]},r)$ \'etabli dans la section \textsection4, on obtient le $p^{Nr}$-quasi-isomorphisme $\alpha_r$ (avec $N=42$). La m\^eme d\'emonstration donne un $p^{Nr}$-quasi-isomorphisme $\alpha_{r,n}$ pour $N= 70$.   

\end{proof}

\section{Comparaison locale avec l'application de Fontaine-Messing} 

Dans cette partie, on montre que le morphisme de p\'eriode local construit dans le chapitre pr\'ec\'edent est \'egal \`a l'application de Fontaine-Messing modulo une certaine puissance de $p$. Une fois la construction du morphisme \'etablie, la preuve est tr\`es similaire \`a celle du cas arithm\'etique (\cite[\textsection 4.7]{CN2017}). 
On suppose toujours $\XF= \Spf(R)$ connexe.  

\subsection{Les anneaux $\E_{\overline{R}}^{[u,v]}$ et $\E_{\overline{R}}^{PD}$ et lemmes de Poincar\'e}

On rappelle ici la d\'efinition de $\E_{\overline{R}}^{[u,v]}$ et $\E_{\overline{R}}^{PD}$ donn\'ee dans \cite[\textsection 2.6]{CN2017} et les lemmes de Poincar\'e v\'erifi\'es par ces anneaux. 

\subsubsection{D\'efinition des anneaux $SA$}

Dans cette section, $S$ d\'esignera l'anneau $R_{\text{cris}}^+$ (respectivement $R^{[u,v]}$) et $A$ d\'esignera $\AAC(\overline{R})$ ou $\AAC(R)$ (respectivement $\A_{\overline{R}}^{[u,v]}$ ou $\A_R^{[u,v]}$). On a donc un morphisme injectif continu $\iota : S \to A$ et, si $B$ d\'esigne l'anneau $\AAC$ (respectivement $\A^{[u,v]}$), on note $f: S \otimes_B A \to A$ le morphisme tel que $f(x \otimes y)= \iota(x) y$.

On note $S A$ la log-PD-enveloppe compl\'et\'ee de $S \otimes_B A \to A$ par rapport \`a $\ker(f)$.
Plus explicitement, si $V_j = \frac{ X_j \otimes 1}{1 \otimes \iota(X_j)}$ pour $1 \le j \le d$, alors $S A$ est la compl\'etion $p$-adique de $S \otimes_B A$ auquel on ajoute tous les $(x \otimes 1 - 1 \otimes \iota(x))^{[k]}$ pour $x$ dans $S$ et les $(V_j-1)^{[k]}$.
On munit $S A$ d'une filtration en d\'efinissant $F^r S A$ comme l'adh\'erence de l'id\'eal engendr\'e par les \'el\'ements de la forme : 
$$ x_1 x_2 \prod_{1 \le j \le d} (V_j-1)^{[k_j]}$$
o\`u $x_1$ est dans $F^{r_1}S$, $x_2$ est dans $F^{r_2} A$ et $r_1+r_2+ \sum_{j=1}^d k_j \le r$. 

\begin{Pro}\cite[Lem. 2.36]{CN2017}
\label{structure SLambda}
\begin{enumerate}
\item Tout $x$ de $SA$ s'\'ecrit de mani\`ere unique sous la forme : 
\begin{equation}
\label{structure SLambda1}
 x= \sum_{ \mathbf{k}=(k_1, \dots, k_d) \in \N^{d}} x_{\mathbf{k}} \prod_{j=1}^d (1-V_j)^{[k_j]}
 \end{equation}
avec $x_{\mathbf{k}}$ dans $A$ et $x_{\mathbf{k}}$ tend vers $0$ quand $\mathbf{k}$ tend vers l'infini. 
\item De plus, 
$$ x= \sum_{ \mathbf{k} \in \N^{d}} x_{\mathbf{k}} \prod_{j=1}^d (1-V_j)^{[k_j]} \; \in \; F^r S A \quad \text{si et seulement si} \quad x_{\mathbf{k}} \; \in \; F^{r - | \mathbf{k} |} A \text{ pour tout } \mathbf{k}. $$
\end{enumerate}
\end{Pro}

On d\'efinit les anneaux suivants : 
\begin{itemize}
\item $\E_R^{PD}= S A$ (resp. $\E_R^{[u,v]}$) pour $S=R_{\text{cris}}^+$ (resp. $R^{[u,v]}$) et $A=\AAC(R)$ (resp. $\A_R^{[u,v]}$),
\item $\E_{R_{\infty}}^{PD}= S A$ (resp. $\E_{R_{\infty}}^{[u,v]}$) pour $S=R_{\text{cris}}^+$ (resp. $R^{[u,v]}$) et $A=\AAC(R_{\infty})$ (resp. $\A_{R_{\infty}}^{[u,v]}$),
\item $\E_{\overline{R}}^{PD}= S A$ (resp. $\E_{\overline{R}}^{[u,v]}$) pour $S=R_{\text{cris}}^+$  (resp. $R^{[u,v]}$) et $A= \AAC( \overline{R})$ (resp. $\A_{\overline{R}}^{[u,v]}$). 
\end{itemize}

On d\'eduit de la proposition \ref{structure SLambda} les inclusions : 
$  \E_R^{PD} \subseteq \E_{R_{\infty}}^{PD} \subseteq \E_{\overline{R}}^{PD}$ (et de m\^eme avec l'exposant $[u,v]$).

On peut maintenant \'etendre les actions de $G_R$ et de $\varphi$ \`a ces anneaux :

\begin{Pro}\cite[Lem. 2.40]{CN2017}
Si $SA= \E_M^{PD}$ (respectivement $\E_M^{[u,v]}$) pour $M \in \{R, R_{\infty}, \overline{R} \}$, on note $S A'= \E_M^{PD}$ (respectivement $\E_M^{[u,\frac{v}{p}]}$). Alors :

\begin{enumerate}
\item $\varphi$ s'\'etend de mani\`ere unique en un morphisme continu $S A \to S A'$.
\item L'action de $G_R$ s'\'etend de mani\`ere unique \`a un action continue sur $S A$, qui commute avec $\varphi$ et $\E_{R_{\infty}}^{\star}$ (pour $\star \in \{ PD, [u,v] \}$) est l'ensemble des points fixes de $\E_{\overline{R}}^{\star}$ par le sous-groupe $\Gamma_R \subseteq G_R$. 
\end{enumerate}
\end{Pro}

\subsubsection{Lemmes de Poincar\'e}

On consid\`ere maintenant le complexe 
\[F^r \Omega_{S A/ A}^{\bullet} := F^r S A \to F^{r-1} SA \otimes_{SA} \Omega^1_{S A/ A} \to F^{r-2} SA \otimes_{SA} \Omega^2_{S A/ A} \to \dots. \]

\begin{Pro}\cite[Lem. 2.37]{CN2017}
\label{Poincare1}
Le morphisme 
$ F^r A \to F^r \Omega_{S A / A}^{\bullet}$
est un quasi-isomorphisme. 
\end{Pro}   

\begin{proof}
La preuve est identique \`a celle de Colmez-Nizio{\l}. 
\end{proof}

Pour simplifier les notations, on note $R_1:=R^{[u,v]}$ et $R_2:= \A_R^{[u,v]}$. On a des plongements $\iota_i : R^{[u,v]} \to R_i$ pour $i \in \{1, 2 \}$. On note ensuite $R_3$ la log-PD-enveloppe compl\'et\'ee de $R_1 \otimes_{\A^{[u,v]}} R_2$ pour $\iota_2 \circ \iota_1^{-1} : R_1 \to R_2$ (et donc $R_3= \E_R^{[u,v]}$). On pose $X_{ij}:= \iota_i(X_j)$ pour $i \in \{1,2 \}$ et $j \in \{ 1, \dots , d \}$. Soient 
$$ \Omega_i^1= \bigoplus_{j=1}^d \Z \frac{ d X_{ij}}{X_{ij}}, \quad \Omega_3^1= \Omega_1^1 \oplus \Omega_2^1, \quad \Omega_i^n = \bigwedge^n \Omega_i^1 $$
$$ \text{ et } F^r \Omega_{R_i}^{\bullet} := F^r R_i \to F^{r-1} R_i \otimes_{\Z} \Omega_i^1 \to F^{r-2} R_i \otimes_{\Z} \Omega_i^2 \to \dots $$

\begin{Lem}\cite[Lem. 3.11]{CN2017}
\label{Poincare2}
Les applications $F^r \Omega_{R_1}^{\bullet} \to F^r \Omega_{R_3}^{\bullet}$ et $F^r \Omega_{R_2}^{\bullet} \to F^r \Omega_{R_3}^{\bullet}$ sont des quasi-isomorphismes.
\end{Lem}

\subsection{Application de Fontaine-Messing}

On note $\E_{\overline{R},n}^{PD}$ la r\'eduction modulo $p^n$ de $\E_{\overline{R}}^{PD}$ (alors $\E_{\overline{R},n}^{PD}$ est la log-PD-enveloppe de $\Spec(\overline{R}_n) \to \Spec(\AAC(\overline{R})_n \otimes_{\AAC_{,n}} R_{\text{cris}}^+)$). 

On a le diagramme commutatif suivant : 
\[
   \xymatrix{
   & \mathrm{Spf}(\E_{\overline{R},n}^{PD}) \ar[rd] & \\
    \mathrm{Spf}(\overline{R}_n) \ar@{^{(}->}[rr] \ar@{^{(}->}[ru]  \ar[d] & & \mathrm{Spf}(\AAC(\overline{R})_n \otimes_{\AAC} R_{\text{cris}}^+) \ar[d] \\
    \mathrm{Spf}(R_n) \ar@{^{(}->}[rr] \ar[d] & & \mathrm{Spf}(R_{\text{cris},n}^+) \ar[d] \\
    \mathrm{Spf}(\OC_{C,n}) \ar@{^{(}->}[rr] & &\mathrm{Spf}(A_{\rm{cris},n}) 
  }
\]
  On d\'efinit le complexe syntomique : 
  $ \mathrm{Syn}(\overline{R},r)_n := [ F^r \Omega_{\E_{\overline{R},n}^{PD}}^{\bullet} \xrightarrow{p^r- \varphi} \Omega_{\E_{\overline{R},n}^{PD}}^{\bullet} ]. $
   
 La suite $p^r$-exacte 
  \[0 \to \Z/p^n(r)' \to F^r\AAC( \overline{R})_n \xrightarrow{p^r- \varphi} \AAC( \overline{R})_n \to 0 \]
  donne un $p^r$-quasi-isomorphisme
\[  R \Gamma(G_R, \Z/p^n(r)') \xrightarrow{\sim} R \Gamma(G_R, [ F^r\AAC( \overline{R})_n \xrightarrow{p^r- \varphi} \AAC( \overline{R})_n ] ) \]
 et, par la proposition \ref{Poincare1}, on a : 
 \[ R \Gamma(G_R, [ F^r\AAC( \overline{R})_n \xrightarrow{p^r- \varphi} \AAC( \overline{R})_n ] ) \xrightarrow{\sim} R \Gamma(G_R, [ F^r \Omega_{E_{\overline{R},n}^{PD}}^{\bullet} \xrightarrow{p^r- \varphi} \Omega_{E_{\overline{R},n}^{PD}}^{\bullet} ] ). \]   
 On peut alors d\'efinir l'application de Fontaine-Messing :    
\[\alpha^{\rm FM}_{r,n} :  \mathrm{Syn}(R,r)_n \to R \Gamma(G_R, \Z/p^n(r)')\]

On dit que deux morphismes $f$ et $g$ sont $p^N$-\'egaux (avec $N$ une constante) si le morphisme induit par $f-g$ sur les groupes de cohomologie est tu\'e par $p^N$. On obtient la version g\'eom\'etrique du th\'eor\`eme 4.16 de \cite{CN2017}.
  
\begin{The}
\label{comparaison FM}
Il existe une constante $N$ qui ne d\'epend que de $p$ et de $r$ telle que l'application de Fontaine-Messing $\alpha_r^{\rm FM}$ (respectivement $\alpha_{r,n}^{\rm FM}$) est $p^{N}$-\'egale \`a l'application $\alpha^0_r$ (respectivement $\alpha^0_{r,n}$) donn\'ee par le th\'eor\`eme \ref{Qis syn Galois}.
\end{The}

\begin{proof}
On reprend les m\^emes notations que dans \cite[Th. 4.16]{CN2017}, c'est-\`a-dire : 
\begin{itemize}
\item[$\bullet$] $C_G$ (resp. $C_{\Gamma}$) d\'esignent les complexes de cocha\^ines continues de $G$ (respectivement $\Gamma$) 
\item[$\bullet$] $K_{\varphi}(F^rS):=[ F^rS \xrightarrow{p^r - \varphi} S' ]$ (et donc $K_{\varphi}(S):=[S \xrightarrow{1- \varphi} S']$) ;
\item[$\bullet$] $K_{\varphi, \partial}(F^rS):= \mathrm{Kos}( \varphi, \partial, F^rS)= [ \Kos( \partial, F^rS ) \xrightarrow{p^r - p^{\bullet} \varphi} \Kos( \partial , S') ]$, 
\item[$\bullet$] $K_{\varphi, \Gamma}(F^rS):=\Kos(\varphi, \Gamma, F^rS)= [ \Kos( \Gamma, F^rS ) \xrightarrow{p^r - \varphi} \Kos( \Gamma , S') ]$,
\item[$\bullet$] $K_{\varphi, \Gamma, \partial}(F^rS):=[\Kos(\Gamma, F^r\Omega^{\bullet}_S) \xrightarrow{p^r - \varphi} \mathrm{Kos}( \Gamma , \Omega^{\bullet}_{S'}) ]$,
\item[$\bullet$] $K_{\varphi, \Lie\Gamma}(F^rS):=[\Kos(\Lie\Gamma, F^rS) \xrightarrow{p^r- \varphi} \Kos(\Lie\Gamma, S')]$ et,
\item[$\bullet$] $K_{\varphi, \Lie\Gamma, \partial}(F^rS):=[\Kos(\Lie\Gamma, F^r\Omega^{\bullet}_S) \xrightarrow{p^r - \varphi} \Kos(\Lie\Gamma , \Omega^{\bullet}_{S'}) ].$
\end{itemize}

Le th\'eor\`eme se d\'eduit alors du diagramme commutatif :

\begin{footnotesize}
\hspace{-1cm}\[
\xymatrix @C=0.45cm @R=1cm{
K_{\varphi, \partial}(F^rR_{\text{cris}}^+) \ar[d]^{\rotatebox{90}{$\sim$}}_{\tau_{\le r}} \ar[r] & C_G(K_{\varphi,\partial }(F^r\E_{\overline{R}}^{PD})) \ar[d] & \ar[l]_{\sim}^{\ref{Poincare1}} C_G(K_{\varphi}(F^r \AAC(\overline{R}))) \ar[d] & \ar[l]_-{\sim}^-{\eqref{suite exacte Acris}} C_G( \Z_p(r))  \ar[d]^{\rotatebox{90}{$\sim$}}_{\eqref{suite exacte Av}} \ar[dr]_{\sim}^{\eqref{suite exacte Ar}} \ar[ld]^{\sim}_{\eqref{suite exacte Auv 3}} &  \\
K_{\varphi, \partial}(F^rR^{[u,v]}) \ar@/_2cm/[rddddd]^{\sim}_{\text{prop. \ref{Poincare2}}} \ar[r] & C_G(K_{\varphi, \partial}(F^r\E_{\overline{R}}^{[u,v]})) & \ar[l]_{\sim}^{\ref{Poincare1}} C_G(K_{\varphi}(F^r \A^{[u,v]}_{\overline{R}})) & \ar[l]^-{t^r} C_G( K_{\varphi}(\A^{(0,v]+}_{\overline{R}}(r))) \ar[r] & C_G(K_{\varphi}(\A_{\overline{R}}(r))) \\
& C_{\Gamma}(K_{\varphi, \partial}(F^r\E_{R_{\infty}}^{[u,v]}) ) \ar[u]^{\mu_H} & \ar[l]_{\sim}^{\ref{Poincare1}} C_{\Gamma}(K_{\varphi}(F^r \A_{R_{\infty}}^{[u,v]})) \ar[u]^{\mu_H} & \ar[l]^-{t^r} C_{\Gamma}(K_{\varphi}(\A_{R_{\infty}}^{(0,v]+}(r)))\ar[r] \ar[u]^{\mu_H} & C_{\Gamma}(K_{\varphi}(\A_{R_{\infty}}(r))) \ar[u]^{\text{prop. \ref{muH}}}_{\rotatebox{90}{$\sim$}} \\ 
& C_{\Gamma}(K_{\varphi, \partial}(F^r\E_{R}^{[u,v]}) ) \ar[u]^{\mu_{\infty}} & \ar[l]_{\sim}^{\ref{Poincare1}} C_{\Gamma}(K_{\varphi}(F^r \A_R^{[u,v]})) \ar[u]^{\mu_{\infty}} & \ar[l]^-{t^r} C_{\Gamma}(K_{\varphi}(\A_{R}^{(0,v]+}(r)))\ar[r] \ar[u]^{\mu_{\infty}}& C_{\Gamma}(K_{\varphi}(\A_{R}(r))) \ar[u]^{\text{prop. \ref{muinfini}}}_{\rotatebox{90}{$\sim$}} \\ 
& K_{\varphi, \Gamma, \partial}(F^r\E_R^{[u,v]}) \ar[u]_{\rotatebox{90}{$\sim$}} \ar[d]^{\rotatebox{90}{$\sim$}} & \ar[l]_{\sim}^{\ref{Poincare1}} K_{\varphi, \Gamma}(F^r \A_R^{[u,v]}) \ar[u]_{\rotatebox{90}{$\sim$}} \ar[d]^{\rotatebox{90}{$\sim$}} & \ar[l]^-{t^r} K_{\varphi, \Gamma}(\A_R^{(0,v]+}(r)) \ar[u]_{\rotatebox{90}{$\sim$}} \ar[r]^{\sim}_{\ref{Ar to Av}} \ar[d]^{\rotatebox{90}{$\sim$}}_{\ref{Gamma to Lie}} & K_{\varphi, \Gamma}(\A_R(r)) \ar[u]_{\rotatebox{90}{$\sim$}} \\ 
& K_{\varphi, \Lie\Gamma, \partial}(F^r\E_R^{[u,v]})  & \ar[l]_{\sim}^{\ref{Poincare1}} K_{\varphi,\Lie\Gamma}(F^r \A_R^{[u,v]})  & \ar[l]^-{\ref{filtration [u,v]}}_{\sim} K_{\varphi, \Lie\Gamma}(\A_R^{[u,v]}(r)) & \\
& K_{\varphi, \partial}(F^r\E_R^{[u,v]}) \ar[u]^{t^{\bullet}} & \ar[l]_{\sim}^{\ref{Poincare2}} K_{\varphi, \partial}(F^r \A_R^{[u,v]}) \ar[u]_{\rotatebox{90}{$\sim$}}^{t^{\bullet}, \tau_{\le r}, (\ref{Lie to Delta})} & & 
}
\]
\end{footnotesize}

Les fl\`eches horizontales de la quatri\`eme colonne vers la troisi\`eme sont donn\'ees par les compositions d'applications :
\[
 \xymatrix{
 \Kos(\Gamma_R, S(r)) \ar[r]^{1-\varphi} \ar[d] & \Kos(\Gamma_R, S'(r)) \ar[d]^{p^r} \\
 \Kos(\Gamma_R, S(r)) \ar[r]^{p^r(1-\varphi)} \ar[d]^{t^r}  & \Kos(\Gamma_R, S'(r)) \ar[d]^{t^r}  \\
 \Kos(\Gamma_R, F^rS)) \ar[r]^{p^r-\varphi}  & \Kos(\Gamma_R, S') \\
 } \text{ et }
 \xymatrix{
 \Kos( \Lie \Gamma_R, S(r)) \ar[r]^{1-\varphi} \ar[d] & \Kos(\Lie \Gamma_R, S'(r)) \ar[d]^{p^r} \\
 \Kos( \Lie \Gamma_R, S(r)) \ar[r]^{p^r(1-\varphi)} \ar[d]^{t^r}  & \Kos(\Lie \Gamma_R, S'(r)) \ar[d]^{t^r}  \\
 \Kos( \Lie \Gamma_R, F^rS)) \ar[r]^{p^r-\varphi}  & \Kos(\Lie \Gamma_R, S'). \\
 } 
\] 
Le passage de la troisi\`eme \`a la quatri\`eme ligne se fait en utilisant le quasi-isomorphisme entre les complexes de Koszul et la cohomologie de groupe.

\end{proof}

\section{R\'esultat global}

\`A partir de maintenant, on se place dans le cas o\`u il n'y a plus de diviseur \`a l'infini. Plus pr\'ecis\'ement, on suppose que localement $\XF$ s'\'ecrit $\Spf(R_0)$ avec $R_0^{\square} \to R_0$ la compl\'etion d'un morphisme \'etale et $R_0^{\square}:= \OC_K \{ X_1, \dots, X_d, \frac{1}{X_1 \cdots X_a}, \frac{\varpi}{X_{a+1} \cdots X_d} \}$ (pour $a$, $d$ dans $\N$). On suppose $\Spf(R_0)$ connexe. On rappelle que $\Spf(\OC_K)$ est muni de la log-structure donn\'ee par son point ferm\'e et $\XF$ de celle donn\'ee par $\OC_{\XF} \cap (\OC_{\XF} [ \frac{1}{p} ])^{\times} \hookrightarrow \OC_{\XF}$.

Le but de cette section est de montrer le th\'eor\`eme : 

\begin{The}
\label{comparaison globale}
Les applications locales du th\'eor\`eme \ref{Qis syn Galois} se globalisent en des morphismes : 
\[ \alpha^0_{r} : \tau_{\le r} R \Gamma_{ {\rm syn}}( \XF_{\OC_{C}},r) \to \tau_{\le r} R \Gamma_{\emph{\text{\'et}}}(\XF_C, \Z_p(r)'). \]
\[ \alpha^0_{r,n} : \tau_{\le r} R \Gamma_{ {\rm syn}}( \XF_{\OC_{C}},r)_n \to \tau_{\le r} R \Gamma_{\emph{\text{\'et}}}(\XF_C, \Z/p^n\Z(r)'). \]
De plus, ces morphismes sont des $p^{Nr}$-isomorphismes pour une constante $N$ universelle. 
\end{The}

En particulier, on a le r\'esultat (\cite[Cor. 5.12]{CN2017}) : 

\begin{Cor}
\label{alpha0}
Si $\XF$ est un log-sch\'ema formel propre \`a r\'eduction semi-stable sur $\OC_K$ alors pour $k \le r$, le $p^{Nr}$-isomorphisme ci-dessus induit un isomorphisme : 
\[ \alpha^0_{r,k} : H^k_{\mathrm{syn}}(\XF_{\OC_{C}},r)_{\Q} \xrightarrow{\sim} H_{\emph{\text{\'et}}}^k( \XF_{C}, \Q_p(r)). \] 
\end{Cor} 

\begin{proof}[Preuve du corollaire \ref{alpha0}]
Par d\'efinition de la limite homotopique, si $A^{\bullet}$ d\'esigne le complexe $R\Gamma_{\rm syn}(\XF_{\OC_C},r)$ (respectivement $R\Gamma_{\text{\'et}}(\XF_{C}, \Q_p(r))$) et $A^{\bullet}_n$ sa r\'eduction modulo $p^n$ (respectivement $R\Gamma_{\text{\'et}}(\XF_{C}, \Z/p^n\Z(r)')$) alors on a un triangle distingu\'e : 
\[ A^{\bullet} \to \prod_n A^{\bullet}_n \to \prod_n A^{\bullet}_n \to A^{\bullet}[1] \]
o\`u l'application $\prod_n A^{\bullet}_n \to \prod_n A^{\bullet}_n$ est donn\'ee par $(k_n) \mapsto (k_n-\psi_{n+1}(k_{n+1}))$ (avec $\psi_{n+1} : A^{\bullet}_{n+1} \to A^{\bullet}_n$ sont les morphismes associ\'es au syst\`eme projectif $(A^{\bullet}_n)$).  

Par le th\'eor\`eme \ref{comparaison globale}, pour tout $n$ et $k \le r$, on a un $p^{Nr}$-isomorphisme $H^kR\Gamma_{\syn}(\XF_{\OC_{C}},r)_n \xrightarrow{\sim} H^kR\Gamma_{\text{\'et}}(\XF_{C}, \Z/p^n\Z(r)')$. En \'ecrivant la suite exacte longue associ\'ee au triangle distingu\'e ci-dessus, on obtient un diagramme commutatif :
\[ \xymatrix{
\dots \ar[r] & \prod_n H^{k-1}_{\rm{syn}}(r)_n \ar[d]^{\rotatebox{90}{$\sim$}} \ar[r] & H^k_{\mathrm{syn}}(r) \ar[r] \ar[d] & \prod_n H^k_{\rm{syn}}(r)_n \ar[r] \ar[d]^{\rotatebox{90}{$\sim$}} & \prod_n H^k_{\rm{syn}}(r)_n \ar[r] \ar[d]^{\rotatebox{90}{$\sim$}} & \dots \\ 
\dots \ar[r] & \prod_n H^{k-1}_{\text{\'et}}(r)_n \ar[r] & H^k_{\text{\'et}}(r) \ar[r] & \prod_n H^k_{\text{\'et}}(r)_n \ar[r] & \prod_n H^k_{\text{\'et}}(r)_n \ar[r] & \dots } \]
o\`u pour simplifier on a not\'e $H^{k}_{\rm{syn}}(r)_n:= H^kR\Gamma_{\rm syn}(\XF_{\OC_C},r)_n$, $H^{k}_{\rm{syn}}(r):= H^kR\Gamma_{\rm syn}(\XF_{\OC_C},r)$ et de m\^eme pour la cohomologie \'etale. On tensorise ensuite par $\Q_p$ pour obtenir, pour tout $k \le r$, un isomorphisme :  
\[ \alpha^0_{r,k} : H^k_{\mathrm{syn}}(\XF_{\OC_{C}},r)_{\Q} \xrightarrow{\sim} H_{\text{\'et}}^k( \XF_{C}, \Q_p(r)). \] 
\end{proof}

\begin{Rem}
Dans \cite{Tsu99}, pour obtenir un morphisme compatible avec les applications symboles, une normalisation par un facteur $p^{-r}$ est n\'ecessaire (voir \cite[3.1.12]{Tsu99} et \cite[Proposition 3.2.4 (3)]{Tsu99} pour la compatibilit\'e avec les applications symboles). Cette normalisation n'apparait pas ici car le morphisme au niveau entier qu'on utilise n'est pas exactement \'egal \`a celui de \cite{Tsu99}, mais est obtenu en tordant celui-ci par un facteur $p^r$. 
\end{Rem}

Revenons \`a la preuve du th\'eor\'eme \ref{comparaison globale}. On a construit le morphisme $\alpha_{r}^0$ localement, il suffit de montrer qu'on peut le globaliser. Pour cela on utilise la m\'ethode de \v{C}esnavi\v{c}ius-Koshikawa dans \cite[\textsection 5]{CK17} (qui g\'en\'eralise celle de \cite{BMS16}).

On se place dans le cas o\`u $\XF_{\OC_C}= \mathrm{Spf}(R)$. On suppose que les intersections de deux composants irr\'eductibles de la fibre sp\'eciale $\mathrm{Spec}( R \otimes_{\OC_C} k)$ sont non vides (en particulier, $\Spf(R)$ est connexe)  et qu'il existe une immersion ferm\'ee
\begin{equation}
\label{im fermee}
 \XF_{\OC_C} \to \mathrm{Spf}(R_{\Sigma}^{\square}) \times_{\OC_C} \prod_{\lambda \in \Lambda} \mathrm{Spf} R_{\lambda}^{\square}
 \end{equation}
telle que  
\begin{enumerate}[label=(\roman*)]
\item $R_{\Sigma}^{\square}:= \OC_C \{ X_{\sigma}^{\pm 1} \; | \;  \sigma \in \Sigma \}$ avec $\Sigma$ un ensemble fini ; 
\item $R_{\lambda}^{\square}:= \OC_C \{ X_{\lambda, 1}, \dots , X_{\lambda, d}, \frac{1}{X_{\lambda, 1} \cdots X_{\lambda, a_{\lambda}}}, \frac{\varpi}{X_{\lambda,a_{\lambda}+1} \cdots X_{\lambda,d} }\}$ o\`u $\lambda \in \Lambda$ avec $\Lambda$ fini ;  
\item $\mathrm{Spf}(R) \to \mathrm{Spf}(R_{\Sigma}^{\square})$ est une immersion ferm\'ee ;
\item $\mathrm{Spf}(R) \to \mathrm{Spf}(R_{\lambda}^{\square})$ est \'etale pour tout $\lambda$ dans $\Lambda$.
\end{enumerate}
En particulier, pour tout $\lambda$ dans $\Lambda$, chaque composante irr\'eductible de la fibre sp\'eciale $\mathrm{Spec}( R \otimes k)$ est donn\'ee par un unique $(t_{\lambda, i})$ pour $a_{\lambda}+1 \le i \le d.$

On choisit des \'el\'ements $\{ u_{\sigma} \}_{ \sigma \in \Sigma}$ (respectivement $\{ u_{\lambda,i} \}_{1 \le i \le d}$ avec $\lambda \in \Lambda$) tels qu'on ait une application $X_{\sigma} \mapsto u_{\sigma}$ de $R^{\square}_{\Sigma}$ dans $R$ (respectivement $X_{\lambda, i} \mapsto u_{\lambda,i}$ de $R_{\lambda}^{\square}$ dans $R$).  

Dans la suite on note $\Spf(R_{\Sigma, \Lambda}^{\square})$ le produit $\Spf(R_{\Sigma}^{\square}) \times_{\OC_C} \prod_{\lambda \in \Lambda} \Spf( R_{\lambda}^{\square} )$.

Comme dans la section 2.2.2, on munit chaque $\Spf(R_{\lambda}^{\square})$ de la log-structure donn\'ee par la fibre sp\'eciale et on munit $\Spf(R_{\Sigma, \Lambda}^{\square})$ et $\Spf(R)$ des log-structures induites.   

\begin{Rem}
On peut toujours trouver une base de $\XF_{\text{\'et}}$ telle qu'on ait de telles immersions : soit $x$ un point de $\XF$. Si $x$ est dans le lieu lisse de $\XF$ alors localement $\XF_{\OC_C}$ s'\'ecrit $\Spf(\widehat{A})$ avec $A$ de type fini sur $\OC_C$ et donc il existe une immersion ferm\'ee $\Spec(A) \hookrightarrow \Spec( \OC_C [ X_1, \dots, X_n ])$. On recouvre ensuite $\Spec( \OC_C [ X_1, \dots, X_n ])$ par une union de tores et on obtient que localement $\XF_{\OC_C}$ s'\'ecrit $\Spf(R)$ tel qu'il existe une immersion ferm\'ee $\Spf(R) \hookrightarrow \Spf(R_{\Sigma}^{\square})$ avec $\Sigma$ un ensemble fini.   

Si maintenant $x$ n'est plus dans le lieu lisse de $\XF$, $\XF_{\OC_C}$ s'\'ecrit localement $\Spf(R_{\lambda})$ pour une $\OC_C$-alg\`ebre $R_{\lambda}$ telle qu'on ait une application \'etale $R_{\lambda}^{\square} \to R_{\lambda}$ avec
\[ R_{\lambda}^{\square}:= \OC_C \{ X_{\lambda, 1}, \dots , X_{\lambda, d}, \frac{1}{X_{\lambda, 1} \cdots X_{\lambda, a_{\lambda}}}, \frac{\varpi}{X_{\lambda,a_{\lambda}+1} \cdots X_{\lambda,d} }\} \]
et telle que $x$ est donn\'e par $(X_{\lambda,a_{\lambda}+1} \cdots X_{\lambda,d})$. Si on localise en $x$, on obtient une immersion ferm\'ee dans un tore formel $\Spf(R_{\Sigma}^{\square})$.  
\end{Rem} 

\begin{Rem}
\label{fonctoriel}
\begin{enumerate}
\item Pour un tel $R$, si $R_{\Sigma_1, \Lambda_1}^{\square}$ et $R_{\Sigma_2, \Lambda_2}^{\square}$ v\'erifient les conditions ci-dessus avec $\Sigma_1 \subseteq \Sigma_2$ et $\Lambda_1 \subseteq \Lambda_2$, on note $f_{1}^{2} : R_{\Sigma_1, \Lambda_1}^{\square} \to R_{\Sigma_2, \Lambda_2}^{\square}$ l'application induite par $X_{\sigma} \mapsto X_{\sigma}$ et $X_{\lambda,i} \mapsto X_{\lambda, i}$ pour $\sigma$ dans $\Sigma_1$, $\lambda$ dans $\Lambda_1$ et $i$ dans $\{1, \dots, d\}$. L'ensemble des $R_{\Sigma, \Lambda}^{\square}$ muni des applications ainsi d\'efinies forme alors un syst\`eme inductif filtrant. Il est inductif car si $\Sigma_1 \subseteq \Sigma_2 \subseteq \Sigma_3$ et $\Lambda_1 \subseteq \Lambda_2 \subseteq \Lambda_3$, on a 
$ f_{1}^{3} = f_{2}^{3} \circ f_{1}^{2}. $
V\'erifions qu'il est filtrant. Par d\'efinition du syst\`eme, pour $\Sigma_1 \subseteq \Sigma_2$ et $\Lambda_1 \subseteq \Lambda_2$, $f_{1}^{2}$ est la seule application de $R_{\Sigma_1, \Lambda_1}^{\square}$ dans $R_{\Sigma_2, \Lambda_2}^{\square}$. Consid\'erons \`a pr\'esent deux anneaux $R^{\square}_{\Sigma_1, \Lambda_1}$ et 
$R^{\square}_{\Sigma_2, \Lambda_2}$ v\'erifiant \eqref{im fermee} pour $R$. On note $\Sigma_3:= \Sigma_1 \sqcup \Sigma_2$ et $\Lambda_3= \Lambda_1 \sqcup \Lambda_2$. Alors $(\Sigma_3, \Lambda_3)$ v\'erifie les conditions \eqref{im fermee} pour $R$ et les deux morphismes $f_{1}^{3}$ et $f_2^3$ sont bien d\'efinis. 

\item Maintenant si $f: R \to R'$ est un morphisme \'etale tel qu'on ait une immersion ferm\'ee $\Spf(R') \to \Spf(R_{\Sigma', \Lambda'}^{\square})$, on note $\widetilde{\Sigma}:= \Sigma \sqcup \Sigma'$ et $\widetilde{\Lambda}= \Lambda \sqcup \Lambda'$. Alors $(\widetilde{\Sigma}, \widetilde{\Lambda})$ v\'erifie les conditions \eqref{im fermee} pour $R'$ : comme $R \to R'$ est \'etale, $R_{\lambda}^{\square} \to R'$ est \'etale pour tout $\lambda$ de $\widetilde{\Lambda}$. On a de plus une surjection $R_{\widetilde{\Sigma}}^{\square} \to R'$ donn\'ee par :
 \[ \begin{array}{lcll}
R_{\widetilde{\Sigma}}^{\square} & \to & R' & \\
 X_{\sigma} & \mapsto & f(u_{\sigma})& \text{ pour } \sigma \in \Sigma  \\ 
 X_{\sigma'} & \mapsto & u_{\sigma'} &\text{ pour } \sigma' \in \Sigma'   
\end{array} \]
et telle qu'on ait un diagramme commutatif : 
\[ \xymatrix{ R \ar[r]^{f} & R' \\ 
  R_{\Sigma'}^{\square} \ar[r] \ar[u] & R_{\widetilde{\Sigma}}^{\square} \ar[u] }.\]
  On obtient alors un morphisme de $R_{\Sigma, \Lambda}^{\square} \to R_{\widetilde{\Sigma}, \widetilde{\Lambda}}^{\square}$ et de $ \varinjlim_{(\Sigma, \Lambda)} R_{\Sigma, \Lambda}^{\square} \to \varinjlim_{(\Sigma', \Lambda')} R_{\Sigma', \Lambda'}^{\square}$ (o\`u la premi\`ere limite est prise sur l'ensemble des $(\Sigma, \Lambda)$ qui v\'erifient \eqref{im fermee} pour $R$ et la seconde limite sur l'ensemble des $(\Sigma', \Lambda')$ qui v\'erifient \eqref{im fermee} pour $R'$). 
\end{enumerate}
\end{Rem} 

\subsection{Construction du quasi-isomorphisme $\alpha_{r, \Sigma, \Lambda}$}

On travaille localement. Supposons que $\XF_{\OC_C}= \Spf(R)$ avec $R$ tel qu'il existe des immersions comme en \eqref{im fermee}.

On va d\'efinir un anneau $R_{\Sigma, \Lambda}^{\rm PD}$ en suivant la construction de \cite[\textsection 5.22]{CK17}. Notons $\Spf(\A_{\inf}(R_{\Sigma, \Lambda}^{\square}))$ le produit des 
\[ \Spf(\A_{\inf}(R_{\lambda}^{\square})):=\Spf(\A_{\inf} \{ X_{\lambda, 1}, \cdots , X_{\lambda, d}, \frac{1}{X_{\lambda, 1} \cdots X_{\lambda, a_{\lambda}}}, \frac{[\varpi^{\flat}]}{X_{\lambda,a_{\lambda}+1} \cdots X_{\lambda,d}} \})  \] et de $\Spf(\A_{\inf}(R_{\Sigma}^{\square})) :=\Spf(\A_{\inf}\{ X_{\sigma}^{\pm 1} \; | \;  \sigma \in \Sigma \}).$ 
Comme pr\'ec\'edemment, $\A_{\inf}$ est muni de la log-structure induite par $x \mapsto [x]$ de $\OC^{\flat} \setminus \{ 0 \} \to \A_{\inf}$. Les $\Spf(\A_{\inf}(R_{\lambda}^{\square}))$ sont munis de la log-structure associ\'ee \`a :
\[ \N^{d-a_{\lambda}} \sqcup_{\N} (\OC^{\flat} \setminus \{0 \}) \to \A_{\inf}(R_{\Sigma, \Lambda}^{\square}) \]
qui envoie $(n_i)$ sur $\prod_{a_{\lambda}+1 \le i \le d} X_i^{n_i}$ et $x \in \OC^{\flat}$ sur $[x] \in \A_{\inf}$, le morphisme $\N \to \N^{d-a_{\lambda}}$ est l'application diagonale et $\N \to \OC^{\flat} \setminus \{0 \}$ est donn\'ee par $m \mapsto [\varpi^{\flat}]^m$. Comme dans la remarque \ref{HomeGo}, on peut se ramener \`a des log-structures fines par changement de base (voir \cite[\textsection 5.9]{CK17}) et de cette fa\c{c}on $\A_{\inf}(R_{\Sigma, \Lambda}^{\square})$ est log-lisse sur $\A_{\inf}$. 

On a une immersion ferm\'ee : 
\[ \Spec(R/p) \hookrightarrow \Spf(\A_{\inf}(R_{\Sigma, \Lambda}^{\square})) \] 
et pour $m,n \ge 1$, on peut construire une log-PD-enveloppe (voir \cite[\textsection 5.22]{CK17} et \cite[1.3]{Bei2013}) : 
\[ \xymatrix{
& \Spec(R_{\Sigma, \Lambda,n,m}^{\rm PD}) \ar[rd] &  \\
\Spec(R/p) \ar @{^{(}->}[ru] \ar[rr] & & \Spec(\A_{\inf}(R_{\Sigma, \Lambda}^{\square})/(p^n, \xi^m)) 
}.
\]   
En fait, pour $m$ suffisamment grand (tel que $\xi^m \in p^n \AAC$), $R_{\Sigma, \Lambda,n,m}^{\rm PD}$ s'identifie \`a la log-PD-enveloppe de
\[ \Spec(R/p) \hookrightarrow \Spec(\A_{\cris}(R_{\Sigma, \Lambda}^{\square})/p^n) \text{ sur } \Spec(\OC_C/p) \hookrightarrow \Spec(\AAC/p^n). \]
Notons $R_{\Sigma, \Lambda, n}^{\rm PD}$ cette log-PD-enveloppe (ind\'ependante de $m$). On a $R_{\Sigma, \Lambda, n}^{\rm PD}/p^{n-1}=R_{\Sigma, \Lambda, n-1}^{\rm PD}$ et on obtient un sch\'ema formel $p$-adique $\Spf(R_{\Sigma, \Lambda}^{\rm PD})$ tel qu'on ait une factorisation :  
\[ \xymatrix{
& \Spf(R_{\Sigma, \Lambda}^{\rm PD}) \ar[rd] &  \\
\Spec(R/p) \ar @{^{(}->}[ru] \ar[rr] & & \Spf(\A_{\cris}(R_{\Sigma, \Lambda}^{\square})) 
}
\]
o\`u $\AAC(R_{\Sigma, \Lambda}^{\square}):= \AAC \widehat{\otimes}_{\AAinf} \AAinf(R_{\Sigma, \Lambda}^{\square})$ (le produit tensoriel est compl\'et\'e pour la topologie $p$-adique).

On a un morphisme de Frobenius sur $R_{\Sigma, \Lambda}^{\rm PD}$ venant de celui de $R/p$. On \'etend de m\^eme les applications diff\'erentielles $\partial_{\sigma}= \frac{d}{ d\log(X_{\sigma})}$ et $\partial_{\lambda,i}=  \frac{d}{ d \log(X_{\lambda,i})}$ \`a $R_{\Sigma, \Lambda}^{\rm PD}$. On a de plus une factorisation : 
\[ \Spec(R/p) \hookrightarrow \Spf(R) \hookrightarrow \Spf(R_{\Sigma, \Lambda}^{\rm PD}). \] 

\vspace{0.5cm}

Le but de cette partie est de construire un $p^{Nr}$-quasi-isomorphisme 
\[ \alpha_{r, \Sigma, \Lambda} : \tau_{\le r} \Kos( \varphi, \partial, F^r R_{\Sigma, \Lambda}^{PD}) \to \tau_{\le r} R\Gamma(G_R, \Z_p(r)) \]
 avec $N$ qui ne d\'epend ni de $\Sigma$, $\Lambda$, ni de $R$.

\vspace{0.5cm}

On consid\`ere les anneaux : 
\[ R^{[u]}_{\Sigma, \Lambda} = \A^{[u]} \widehat{\otimes}_{\AAC} R^{PD}_{\Sigma, \Lambda}, \qquad R^{[u,v]}_{\Sigma, \Lambda} = \A^{[u,v]} \widehat{\otimes}_{\AAC} R^{PD}_{\Sigma, \Lambda} \text{ et } R^{(0,v]+}_{\Sigma, \Lambda} = \A^{(0,v]+} \widehat{\otimes}_{\AAC} R^{PD}_{\Sigma, \Lambda} \]
o\`u les produits tensoriels sont compl\'et\'es pour la topologie $p$-adique.

L'application $R_{\Sigma, \Lambda}^{PD} \to R_{\Sigma, \Lambda}^{[u,v]}$ induit alors un morphisme sur les complexes de Koszul : 
\begin{equation}
\label{RPD to Ruv SigmaLambda}
 \Kos( \varphi, \partial, F^r R_{\Sigma, \Lambda}^{PD} ) \to \Kos( \varphi, \partial, F^r R_{\Sigma, \Lambda}^{[u,v]} ).
 \end{equation}
 
\vspace{0.5cm}

Soient $ R_{\Sigma, \infty}^{\square}$ et $R_{\lambda, \infty}^{\square}$ les compl\'etions $p$-adiques des anneaux
\[ \varinjlim_{n} \OC_C \{ X_{\sigma}^{\frac{\pm 1}{p^n}} \;  | \; \sigma \in \Sigma \} \text{ et }
\varinjlim_{n} \OC_C \{ X_{\lambda, 1}^{\frac{1}{p^n}}, \dots , X_{\lambda, d}^{\frac{1}{p^n}}, \frac{1}{(X_{\lambda, 1} \cdots X_{\lambda, a_{\lambda}})^{\frac{1}{p^n}}}, \frac{\varpi^{\frac{1}{p^n}}}{(X_{\lambda,a_{\lambda}+1} \cdots X_{\lambda,d})^{\frac{1}{p^n}} }\} \]
On note $R_{\Sigma, \Lambda, \infty}^{\square}$ le produit tensoriel des anneaux pr\'ec\'edents, $R_{\Sigma, \Lambda, \infty} := R_{\Sigma, \Lambda, \infty}^{\square} \widehat{\otimes}_{R_{\Sigma, \Lambda}^{\square}} R$ et $\AAC(R_{\Sigma, \Lambda, \infty})$, $\A^{[u,v]}_{R_{\Sigma, \Lambda, \infty}}$, $\A^{[u]}_{R_{\Sigma, \Lambda, \infty}}$ les anneaux associ\'es. 
Enfin, on consid\`ere les groupes $\Gamma_{\Sigma} := \Gal(R^{\square}_{\Sigma, \infty}[ \frac{1}{p}] / R^{\square}_{\Sigma}[ \frac{1}{p}]) \cong \Z_p^{\Sigma}$, $\Gamma_{\lambda}:= \Gal(R^{\square}_{\lambda, \infty}[ \frac{1}{p}] / R^{\square}_{\lambda}[ \frac{1}{p}]) \cong \Z_p^{d}$ et $\Gamma_{\Sigma, \Lambda}:= \Gamma_{\Sigma} \times \prod_{\lambda \in \Lambda} \Gamma_{\lambda}$.

Si $(\gamma_{\sigma})_{\sigma}$ et $(\gamma_{\lambda,i})_{1 \le i \le d}$ sont des g\'en\'erateurs topologiques de $\Gamma_{\Sigma}$ et de $\Gamma_{\lambda}$, les actions de $\Gamma_{\Sigma}$ et $\Gamma_{\lambda}$ sur $R_{\Sigma}^{\square}$ et $R_{\lambda}^{\square}$ sont donn\'ees par :
\[ \gamma_{\sigma}(X_{\sigma})= [ \varepsilon ]X_{\sigma} \text{ et }  \gamma_{\sigma}(X_{\sigma'})= X_{\sigma'} \text{ si } \sigma' \neq \sigma \] 
\[ \gamma_{\lambda,i}(X_{\lambda,i})= [ \varepsilon ]X_{\lambda,i} \text{ et }  \gamma_{\lambda,i}(X_{\lambda,j})= X_{\lambda,j} \text{ si } i \neq j. \] 
On en d\'eduit une action de $\Gamma_{\Sigma, \Lambda}$ sur $R^{PD}_{\Sigma, \Lambda}$, $\A^{[u]}_{\Sigma, \Lambda}$, $\A^{[u,v]}_{\Sigma, \Lambda}$. 

Enfin, pour $1 \le i \le | \Sigma | + d | \Lambda |$, on \'ecrit $J_i:= \{ (j_1, \dots, j_i) \; | \; 1 \le j_1 < \dots < j_i \le | \Sigma | + d | \Lambda | \}$. Les m\^emes preuves que celles des propositions \ref{Lie to Delta} et \ref{Gamma to Lie} donnent : 

\begin{Pro}
Il existe un $p^{30r}$-quasi-isomorphisme : 
\begin{equation}
\label{diff to lie Sigma}
\begin{split}
\tau_{\le r} \Kos( \varphi, \partial, F^r R_{\Sigma, \Lambda}^{[u,v]} )  \xrightarrow{t^{\bullet}}  \tau_{\le r} \Kos( \varphi, \Lie \Gamma_{\Sigma, \Lambda}, F^r R_{\Sigma, \Lambda}^{[u,v]} ) &  \xleftarrow{t^r} \tau_{\le r} \Kos(\varphi, \Lie \Gamma_{\Sigma, \Lambda}, R_{\Sigma, \Lambda}^{[u,v]}(r) )\\
 & \xleftarrow{\beta} \tau_{\le r} \Kos( \varphi, \Gamma_{\Sigma, \Lambda},R_{\Sigma, \Lambda}^{[u,v]}(r)).  
 \end{split}
 \end{equation}
\end{Pro}

\begin{Rem}
Modulo $p^n$, on obtient de la m\^eme fa\c{c}on un $p^{58r}$-quasi-isomorphisme.
\end{Rem}

Le morphisme $R_{\Sigma, \Lambda}^{PD} \to \AAC(R_{\Sigma, \Lambda, \infty})$ (\cite[(5.38.1)]{CK17}) induit un morphisme : 
\begin{equation}
\label{galois 1}
 R \Gamma( \Gamma_{\Sigma, \Lambda}, R_{\Sigma, \Lambda}^{[u,v]}(r)) \to R \Gamma( \Gamma_{\Sigma, \Lambda}, A^{[u,v]}_{R_{\Sigma, \Lambda, \infty}}(r)). 
 \end{equation}
On consid\`ere ensuite la fl\`eche 
\begin{equation}
\label{galois 11}
R \Gamma( \Gamma_{\Sigma, \Lambda}, A^{[u,v]}_{R_{\Sigma, \Lambda, \infty}}(r)) \to R \Gamma_{\text{pro\'et}}(R, \A^{[u,v]}_{\overline{R}}(r))
 \end{equation}
induite par le morphisme de bord associ\'e au torseur (voir \cite[01GY]{StacksProject}, \cite[(3.15.1)]{CK17}) :
\[ \Spf(R_{\Sigma, \Lambda, \infty}) \to \Spf(R) \]
et par le lemme $K(\pi,1)$ de Scholze, le terme de droite dans \eqref{galois 11} est calcul\'e par la cohomologie de Galois $R \Gamma(G_R, \A^{[u,v]}_{\overline{R}}(r))$.

On rappelle qu'on a \'egalement un quasi-isomorphisme :
\begin{equation}
\label{galois 2}
 R \Gamma ( G_{R}, \Z_p(r)) \cong [ R \Gamma_{\text{cont}}( G_{R}, F^r \A^{[u,v]}_{\overline{R}}) \xrightarrow{p^r - \varphi} R \Gamma( G_{R}, \A^{[u,v]}_{\overline{R}})].
\end{equation}

En composant les morphismes pr\'ec\'edents, on obtient une application : 
\[ \alpha_{r, \Sigma, \Lambda} : \tau_{\le r} \Kos( \varphi, \partial, F^r R_{\Sigma, \Lambda}^{PD} ) \to \tau_{\le r} R \Gamma( G_{R}, \Z_p(r)).\]

\begin{Pro}
\label{isom sigma}
Il existe une constante $N$ ind\'ependante de $R$, $\Sigma$ et $\Lambda$ telle que l'application $\alpha_{r, \Sigma, \Lambda}$ ci-dessus est un $p^{Nr}$-quasi-isomorphisme. 
\end{Pro}

\begin{proof}
On fixe un $\lambda$ dans $\Lambda$. On note \[R_{\cris, \lambda, \square}^+ := \AAC \{ X_{\lambda,1}, \dots, X_{\lambda,d}, \frac{1}{X_{\lambda,1} \cdots X_{\lambda, a_{\lambda}}}, \frac{[\varpi^{\flat}]}{X_{\lambda, a_{\lambda}+1} \cdots X_{\lambda , d}} \} \] et $R_{\cris, \lambda, \square}^+ \to R_{\cris, \lambda}^{+}$ un relev\'e \'etale de $R_{\lambda, \square} \to R$. On construit les anneaux $R_{\lambda}^{[u,v]}$ et $\A_{R_{\infty}, \lambda}^{[u,v]}$ associ\'es et on a alors un $p^{Nr}$-quasi-isomorphisme (par ce qui a \'et\'e fait plus haut) :  
\[ \alpha_{r, \lambda} : \tau_{\le r}  \Kos( \varphi, \partial, F^r R_{\cris, \lambda}^+) \xrightarrow{\sim} \tau_{\le r} R \Gamma(G_R, \Z_p(r)). \]
On a un diagramme commutatif :
\[ \xymatrix{
& \Spf(R_{\Sigma, \Lambda}^{\rm PD}) \ar[rd] &  \\
\Spec(R/p) \ar @{^{(}->}[ru] \ar[rr] \ar[d] & & \Spf(\A_{\cris}(R_{\Sigma, \Lambda}^{\square})) \ar[d] \\
\Spf(R_{\cris, \lambda}^+) \ar[rr] & & \Spf(R_{\cris, \lambda, \square}^+)  
}.
\]
Pour tout $n$, l'id\'eal associ\'e au morphisme $R_{\Sigma, \Lambda, n}^{\rm PD} \to R/p$ est localement nilpotent. Comme $R_{\cris,\lambda, \square}^+ \to R_{\cris, \lambda}^+ $ est \'etale, on peut relever la fl\`eche $R_{\cris, \lambda, \square}^+ \to R_{\Sigma, \Lambda,n}^{\rm PD}$ en une unique application $R_{\cris, \lambda}^+ \to R_{\Sigma, \Lambda,n}^{\rm PD}$ qui fait commuter le diagramme :
\[ \xymatrix{
R/p  & R_{\cris, \lambda}^+ \ar[l] \ar@{-->}[ld] \\
R_{\Sigma, \Lambda,n}^{\rm PD} \ar[u]  & R_{\cris, \lambda, \square}^+ \ar[l] \ar[u]. 
} \]
Comme ces applications sont compatibles avec les projections $R_{\Sigma, \Lambda, n}^{\rm PD} \to R_{\Sigma, \Lambda, n-1}^{\rm PD}$, on obtient un morphisme $R_{\cris, \lambda}^+ \to R_{\Sigma, \Lambda}^{\rm PD}$ tel que le diagramme 
\[ \xymatrix{
R/p  & R_{\cris, \lambda}^+ \ar[l] \ar@{-->}[ld] \\
R_{\Sigma, \Lambda}^{\rm PD} \ar[u]  & R_{\cris, \lambda, \square}^+ \ar[l] \ar[u]. 
} \]
est commutatif.  
On a alors un diagramme commutatif : 
\[ \xymatrix{ 
R_{\cris, \lambda}^+ \ar[r] \ar[d] &  R_{\lambda}^{[u,v]} \ar[r] \ar[d] & \A_{R_{\infty}, \lambda}^{[u,v]}  \ar[d]  \\ 
R_{\Sigma, \Lambda}^{\rm PD} \ar[r] &  R_{\Sigma, \Lambda}^{[u,v]} \ar[r] & \A_{R_{\Sigma, \Lambda, \infty}}^{[u,v]}  & 
} \]
et le r\'esultat se d\'eduit du diagramme commutatif suivant : 
\[
\xymatrix{
\tau_{\le r} R\Gamma_{\syn}(R,r) \ar@{=}[d] & \ar[l]_-{\sim} \tau_{\le r } \Kos(\varphi, \partial, F^r R_{\cris, \lambda}^+) \ar[r]^-{\alpha_{r, \lambda}}_-{\sim} \ar[d] &  \tau_{\le r}R \Gamma( G_R , \Z_p(r))  \ar@{=}[d] \\
\tau_{\le r} R\Gamma_{\syn}(R,r) & \ar[l]_-{\sim} \tau_{\le r} \Kos(\varphi, \partial, F^r R_{\Sigma, \Lambda}^{PD}) \ar[r]^-{\alpha_{r,\Sigma, \Lambda}} & \tau_{\le r} R \Gamma( G_{R} , \Z_p(r)). 
 }
\]

\end{proof}

\subsection{Preuve du th\'eor\`eme \ref{comparaison globale}}

On rappelle qu'on cherche \`a montrer le th\'eor\`eme : 

\begin{The}
Il existe un morphisme de p\'eriode : 
\[ \alpha^0_{r,n} : \tau_{\le r} R \Gamma_{ {\rm syn}}( \XF_{\OC_{C}},r)_n \to \tau_{\le r} R \Gamma_{\emph{\text{\'et}}}(\XF_C, \Z/p^n\Z(r)'). \]
De plus, ce morphisme est un $p^{Nr}$-isomorphisme pour une constante $N$ qui ne d\'epend que de $K$, $p$ et $r$. 
\end{The}

\begin{proof}
Soit $\mathfrak{U}^{\bullet} \to \XF$ un hyper-recouvrement affine. On note $R^k$ l'anneau tel que $\mathfrak{U}_{\OC_C}^k := \Spf(R^k)$.
 
Pour chaque $k$, on consid\`ere $(\Sigma_k, \Lambda_k)$ tel qu'on ait une immersion ferm\'ee 
\[ \mathfrak{U}_{\OC_C}^k \hookrightarrow \Spf(R_{\Sigma_k}^{\square}) \times_{\OC_C} \prod_{ \lambda_k \in \Lambda_k} \Spf(R_{\lambda_k}^{\square}) \]
 comme en \eqref{im fermee}. On note ensuite $\Spf(R_{\Sigma_k, \Lambda_k}^{PD})$ la log-PD-enveloppe compl\'et\'ee de $\Spf(R^k) \to \Spf(\AAC(R_{\Sigma_k, \Lambda_k}^{\square}))$. On note :
\[ \Syn(\mathfrak{U}_{\OC_C}^k, \AAC(R_{\Sigma_k, \Lambda_k}^{\square}),r)_n:= [ F^r \Omega^{\bullet}_{R_{\Sigma_k, \Lambda_k}^{PD}} \xrightarrow{p^r - \varphi}  \Omega^{\bullet}_{R_{\Sigma_k, \Lambda_k}^{PD}}]_n. \] 
L'application canonique $\Syn(\mathfrak{U}_{\OC_C}^k, \AAC(R_{\Sigma_k, \Lambda_k}^{\square}),r)_n \to R \Gamma_{\mathrm{syn}}(\mathfrak{U}_{\OC_C}^k, r)_n$ est alors un quasi-isomorphisme. 
De plus, on a vu que l'ensemble des $R^{\square}_{\Sigma_k, \Lambda_k}$ pour $(\Sigma_k, \Lambda_k)$ comme en \eqref{im fermee} forme un syst\`eme inductif filtrant. Le syst\`eme des $\Syn(\mathfrak{U}_{\OC_C}^k, \AAC(R_{\Sigma_k, \Lambda_k}^{\square}),r)_n$ est donc lui-m\^eme inductif filtrant. 
On en d\'eduit un quasi-isomorphisme : 
\begin{equation}
\label{qis Syn}
\varinjlim_{(\Sigma_k, \Lambda_k)}  \Syn(\mathfrak{U}_{\OC_C}^k, \AAC(R_{\Sigma_k, \Lambda_k}^{\square}),r)_n \xrightarrow{\sim} R \Gamma_{\syn}(\mathfrak{U}_{\OC_C}^{k}, r)_n.
\end{equation}

Mais pour chaque $(\Sigma_k, \Lambda_k)$, on a un quasi-isomorphisme $\Syn(\mathfrak{U}_{\OC_C}^k, \AAC(R_{\Sigma_k, \Lambda_k}^{\square}),r)_n \cong \Kos(\varphi, \partial, R_{\Sigma_k, \Lambda_k}^{PD})_n$. On v\'erifie que ce morphisme est fonctoriel. En effet, 
pour $R \to R'$ une application \'etale, soit $(\Sigma, \Lambda)$ associ\'e \`a $R$ et $(\Sigma', \Lambda')$ associ\'e \`a $R'$. Alors si on note $\widetilde{\Sigma}= \Sigma \cup \Sigma'$ et $\widetilde{\Lambda}= \Lambda \cup \Lambda'$, le morphisme $R_{\Sigma, \Lambda}^{\square} \to R_{\widetilde{\Sigma}, \widetilde{\Lambda}}^{\square}$ d\'ecrit plus haut induit un diagramme commutatif : 
\[ \xymatrix{ 
\Omega^{\bullet}_{R_{\Sigma, \Lambda}^{\rm PD}} \ar[d] \ar[r]^-{\sim} & \Kos(\partial, R_{\Sigma, \Lambda}^{PD}) \ar[d] \\
\Omega^{\bullet}_{R_{\widetilde{\Sigma}, \widetilde{\Lambda}}^{PD}} \ar[r]^-{\sim} & \Kos(\partial, R_{\widetilde{\Sigma}, \widetilde{\Lambda}}^{PD}) }. \]  
Comme les syst\`emes sont filtrants, on a pour tout $i$, 
\[ H^i(\varinjlim_{(\Sigma_k, \Lambda_k)} \Kos(\varphi, \partial, R_{\Sigma_k, \Lambda_k}^{PD})_n) \xrightarrow{\sim} \varinjlim_{(\Sigma_k, \Lambda_k)} H^i(\Kos(\varphi, \partial, R_{\Sigma_k, \Lambda_k}^{PD})_n) \]
 et de m\^eme pour les complexes $\Syn(\mathfrak{U}_{\OC_C}^k, \AAC(R_{\Sigma_k, \Lambda_k}^{\square}),r)_n$. On en d\'eduit un quasi-isomorphisme fonctoriel :  
\[ \varinjlim_{(\Sigma_k, \Lambda_k)} \Syn(\mathfrak{U}^k, \AAC(R_{\Sigma_k, \Lambda_k}^{\square}),r)_n \xrightarrow{\sim}  \varinjlim_{(\Sigma_k, \Lambda_k)} \Kos(\varphi, \partial, R_{\Sigma_k, \Lambda_k}^{PD})_n. \]
Par ailleurs, on a construit un $p^{Nr}$-quasi-isomorphisme $\alpha^k_{r, \Sigma, \Lambda}$ fonctoriel de $\tau_{\le r} \Kos(\varphi, \partial, R_{\Sigma_k, \Lambda_k}^{PD})_n$ dans $\tau_{\le r} R \Gamma(G_{R^k}, \Z/p^n\Z(r)')$, qui donne : 
\[  \varinjlim_{(\Sigma_k, \Lambda_k)} \tau_{\le r}\Kos(\varphi, \partial, R_{\Sigma_k, \Lambda_k}^{PD})_n \xrightarrow{\sim} \tau_{\le r}R \Gamma(G_{R^k}, \Z/p^n\Z(r)'). \]
On obtient un quasi-isomorphisme fonctoriel : 
\begin{equation}
\label{qis Cont}
\varinjlim_{(\Sigma_k, \Lambda_k)} \tau_{\le r} \Syn(\mathfrak{U}_{\OC_C}^k, \AAC(R_{\Sigma_k, \Lambda_k}^{\square}),r)_n \xrightarrow{\sim} \tau_{\le r} R \Gamma(G_{R^k}, \Z/p^n\Z(r)').
\end{equation}
En combinant les deux quasi-isomorphismes pr\'ec\'edents, on a :
\[\hspace{-0.5cm}
\xymatrix{
\tau_{\le r} R \Gamma_{\syn}(\mathfrak{U}_{\OC_C}^{k}, r)_n & \ar[l]_-{\sim}^-{\eqref{qis Syn}} \varinjlim_{(\Sigma_k, \Lambda_k)} \tau_{\le r} \Syn(\mathfrak{U}_{\OC_C}^k, \AAC(R_{\Sigma_k, \Lambda_k}^{\square}),r)_n \ar[r]^-{\sim}_-{\eqref{qis Cont}} &  \tau_{\le r} R \Gamma(G_{R^k}, \Z/p^n\Z(r)') 
}
\]
et ce quasi-isomorphisme est fonctoriel. 
On en d\'eduit un quasi-isomorphisme  
\[
\tau_{\le r} R \Gamma_{\syn}(\mathfrak{U}_{\OC_C}^{\bullet}, r)_n \xrightarrow{\sim}  \tau_{\le r} R \Gamma(G_{R^{\bullet}}, \Z/p^n\Z(r)'). 
\]
On d\'efinit
$\alpha^0_{r,n} : \tau_{\le r} R \Gamma_{\syn}(\XF_{\OC_{C}},r)_n \to \tau_{\le r} R \Gamma_{\text{\'et}}(\XF_{C}, \Z/p^n\Z(r)')$ comme la compos\'ee :
\begin{equation}
\begin{split}
 \tau_{\le r} R \Gamma_{\syn}(\XF_{\OC_{C}},r)_n \to \hocolim_{\rm HRF} \tau_{\le r} R\Gamma_{\syn}(\mathfrak{U}_{\OC_C}^{\bullet},r)_n \to \hocolim_{\rm HRF} \tau_{\le r} R \Gamma(G_{R^{\bullet}}, \Z/p^n\Z(r)')  \\ 
 \to \tau_{\le r} \hocolim_{\rm HRF} R\Gamma_{\text{\'et}}(\mathfrak{U}_C^{\bullet}, \Z/p^n \Z(r)') \leftarrow \tau_{\le r} R \Gamma_{\text{\'et}}(\XF_{C}, \Z/p^n\Z(r)').
 \end{split}
 \end{equation}
On va montrer que c'est un $p^{Nr}$-quasi-isomorphisme. La premi\`ere et derni\`ere fl\`eches sont des quasi-isomorphismes par descente cohomologique (voir \eqref{cohodescente} pour la cohomologie syntomique ; le m\^eme argument donne le r\'esultat pour la cohomologie \'etale).  
 La deuxi\`eme fl\`eche est un $p^{Nr}$-quasi-isomorphisme par ce qu'on vient de faire. Il reste \`a voir que, pour tout $k$, on a un quasi-isomorphisme :
 $ R\Gamma(G_{R^k}, \Z/p^n \Z(r)') \xrightarrow{\sim} R\Gamma_{\text{\'et}}(\Sp(R^k[\frac{1}{p}]), \Z/p^n\Z(r)')$. Mais cela d\'ecoule du fait que $\Sp(R^k[\frac{1}{p}])$ est $K( \pi, 1)$ pour $\Z/p^n\Z$ (voir remarque \ref{K(pi,1)}).
\end{proof}

On peut maintenant consid\'erer les faisceaux $\mathcal{H}^k(\mathcal{S}_n(r)_{\overline{\XF}})$ et $\overline{i}^{*}R^k \overline{j}_{*} \Z/p^n(r)'_{\XF_{C, \tr}}$ associ\'es aux pr\'efaisceaux $U \mapsto H^k(U, \mathcal{S}_n(r)_{\overline{\XF}})$ et $U \mapsto H^k(U, \overline{i}^{*}R \overline{j}_{*} \Z/p^n(r)'_{\XF_{C, \tr}} )$. En utilisant l'\'equivalence de cat\'egorie entre $\mathbf{Fsc}((\overline{Y})_{\text{\'et}}, \mathcal{D}(\Z/p^n))$ et $\mathcal{D}^+((\overline{Y})_{\text{\'et}}, \Z/p^n)$, le morphisme ci-dessus peut \^etre vu comme un morphisme dans $\mathcal{D}^+((\overline{Y})_{\text{\'et}}, \Z/p^n)$ et on en d\'eduit le th\'eor\`eme : 

\begin{The}
Pour tout $0 \le k \le r$, il existe un $p^{N}$-isomorphisme 
\[ \alpha^{0}_{r,n} : \mathcal{H}^k(\mathcal{S}_n(r)_{\overline{\XF}}) \to \overline{i}^{*}R^k \overline{j}_{*} \Z/p^n(r)'_{\XF_{C, \tr}} \]
o\`u $N$ est un entier qui d\'epend de $p$ et de $r$, mais pas de $\XF$ ni de $n$. 
\end{The}

  \subsection{Conjecture semi-stable}

On suppose toujours que $\XF$ est propre \`a r\'eduction semi-stable sur $\OC_K$, sans diviseur \`a l'infini. On rappelle qu'on a la cohomologie de Hyodo-Kato $R \Gamma_{\HK}(\XF):= R \Gamma_{\rm{cris}}( \XF_k / W(k)^0 )_{\Q}$ et 
si $D$ est un $(\varphi, N)$-module filtr\'e, $D\{ r\}$ d\'esigne le module $D$ muni du Frobenius $p^r \varphi$ et $D_K$ est muni de la filtration $(F^{\bullet+r}D_K)$.
Le but de cette partie est de prouver la conjecture semi-stable de Fontaine-Jannsen : 

\begin{The}
\label{Fontaine-Jannsen1}
Soit $\XF$ un sch\'ema formel propre semi-stable sur $\OC_K$. On a un $\B_{\rm{st}}$-lin\'eaire, Galois \'equivariant isomorphisme : 
\begin{equation}
\label{Fontaine-Jannsen}
\tilde{\alpha}^0 : H^i_{\emph{\text{\'et}}}( \XF_{C}, \Q_p) \otimes_{\Q_p} \B_{\mathrm{st}} \xrightarrow{\sim} H^i_{\HK}( \XF) \otimes_{F} \B_{\mathrm{st}}.
\end{equation}
De plus, cet isomorphisme pr\'eserve l'action du Frobenius et celle de l'op\'erateur de monodromie et apr\`es tensorisation par $\B_{\mathrm{dR}}$, il induit un isomorphisme filtr\'e : 
\begin{equation}
H^i_{\emph{\text{\'et}}}( \XF_{C}, \Q_p) \otimes_{\Q_p} \B_{\mathrm{dR}} \xrightarrow{\sim} H^i_{\rm dR}( \XF_K) \otimes_{K} \B_{\mathrm{dR}}.
\end{equation}
\end{The}

Pour cela commen\c{c}ons par rappeler les lemmes suivants (voir \cite[Prop. 5.22]{CN2017} et \cite[Prop. 5.20]{CN2017}) :

\begin{Lem}
\label{exactecohosyn}
Il existe un quasi-isomorphisme naturel : 
\[ R \Gamma_{\mathrm{syn}}(\XF_{\OC_{C}}, r)_{\Q} \xrightarrow{\sim}   [[ R\Gamma_{\mathrm{HK}}(\XF) \otimes_F \B_{\mathrm{st}}^+]^{\varphi=p^r, N=0} \xrightarrow{\iota_{\rm{HK}}} (R\Gamma_{\mathrm{dR}}(\XF_K) \otimes_K \B_{\mathrm{dR}}^+)/F^r] \]
o\`u $\iota_{\rm{HK}}$ est induit par le quasi-isomorphisme de Hyodo-Kato : 
$\iota_{\mathrm{HK}} : R \Gamma_{\mathrm{HK}}(\XF) \otimes_F K \xrightarrow{\sim} R \Gamma_{\mathrm{dR}}(\XF_K).$
De plus, pour tout $i$, $H^i[ R\Gamma_{\mathrm{HK}}(\XF) \otimes_F \B_{\mathrm{st}}^+]^{\varphi=p^r, N=0} \cong (H^i_{\mathrm{HK}}(\XF) \otimes_F \B_{\mathrm{st}}^+)^{\varphi=p^r, N=0}$ et si $i < r$, la suite exacte longue associ\'ee induit une suite exacte courte :
\[ 0 \to H^i_{\mathrm{syn}}( \XF_{\OC_{C}},r)_{\Q} \to (H^i_{\mathrm{HK}}(\XF) \otimes_F \B_{\mathrm{st}}^+)^{\varphi=p^r, N=0} \to (H^i_{\mathrm{dR}}(\XF_K) \otimes_K \B_{\mathrm{dR}}^+)/F^r \to 0. \]
\end{Lem}

\begin{Lem}
\label{BCadm}
Le $(\varphi, N)$-module filtr\'e $(H^i_{\rm HK}(\XF),H^i_{\rm dR}(\XF_K), \iota_{\rm HK})$ est faiblement admissible et pour $i \le r$ et on a un isomorphisme (Frobenius-\'equivariant) :
\[H^i_{\rm syn}(\XF_{\OC_C},r)_{\Q} \otimes_{\Q_p} \Bst \xrightarrow{\sim} H^i_{\rm HK}(\XF) \otimes_F \Bst \{-r\}. \] 
\end{Lem}

\begin{Rem}
Le r\'esultat du lemme ci-dessus reste vrai pour tout ensemble de triplets $D^i=(D^i_{\rm st}, D^i_{\rm dR}, \iota^i)$ o\`u, pour tout $i$, 
\begin{enumerate}
\item $D^i_{\rm st}$ est un $F$-espace vectoriel de dimension finie muni d'un endomorphisme de Frobenius $\varphi$ bijectif semi-lin\'eaire et d'un op\'erateur de monodromie $N$ tels que $N \varphi = p \varphi N$, 
\item $D^i_{\rm dR}$ est un $K$-espace vectoriel de dimension finie muni d'une filtration d\'ecroissante, s\'epar\'ee et exhaustive,
\item $\iota^i : D^i_{\rm st} \to D^i_{\rm dR}$ est $F$-lin\'eaire,
\end{enumerate}
tels que $F^{i+1} D_{\rm dR}^i=0$ et tels qu'on ait une suite exacte longue : 
\[ \dots \to H^i(r) \to X^r_{\rm st}(D^i) \to X^r_{\rm dR}(D^i) \to H^{i+1}(r) \to \dots \]
avec $X^r_{\rm st}(D^i)=(t^{-r} \B_{\rm st}^{+} \otimes_F D_{\rm st})^{\varphi=1, N=0}$, $X^r_{\rm dR}(D^i)=(t^{-r} \B_{\rm dR}^{+} \otimes_K D_{\rm dR})/F^0$ et $\dim_{\Q_p} H^i(r)$ est finie pour $i \le r$. La preuve utilise la th\'eorie des espaces de Banach-Colmez (voir \cite{Col2002} et \cite[\textsection 5.2.2]{CN2017}).
\end{Rem} 

\begin{proof}[Preuve du th\'eor\`eme \ref{Fontaine-Jannsen1}]
Par le corollaire \ref{alpha0}, pour $i\le r$, on a un isomorphisme $\alpha^0_{r} : H^i_{ {\rm syn}}( \XF_{\OC_{C}},r)_{\Q} \xrightarrow{\sim} H^i_{\text{\'et}}(\XF_C, \Q_p(r))$ et en utilisant le lemme \ref{exactecohosyn}, on obtient un morphisme :
\[ H^i_{\text{\'et}}(\XF_C, \Q_p(r))  \xleftarrow{\sim} H^i_{ {\rm syn}}( \XF_{\OC_{C}},r)_{\Q} \to (H^i_{\rm HK}(\XF) \otimes_F \Bst)^{\varphi=p^r, N=0}. \]
Le morphisme de p\'eriode $\tilde{\alpha}^0$ est alors donn\'e par : 
\begin{equation}
\label{tildealpha0}
\begin{split}
 H^i_{\text{\'et}}(\XF_C, \Q_p) \otimes_{\Q_p} \Bst  \xleftarrow[\sim]{t^r} H^i_{\text{\'et}}(\XF_C, \Q_p(r)) \otimes_{\Q_p} \Bst \{r \} \xleftarrow[\sim]{\alpha^0_{r}} H^i_{ {\rm syn}}( \XF_{\OC_{C}},r)_{\Q} \otimes_{\Q_p} \B_{\rm st} \{r \} \\   \to H^i_{\rm HK}(\XF) \otimes_F \Bst
  \end{split}
  \end{equation}
et la fl\`eche $H^i_{ {\rm syn}}( \XF_{\OC_{C}},r)_{\Q} \otimes_{\Q_p} \B_{\rm st} \{r\} \to H^i_{\rm HK}(\XF) \otimes_F \Bst$ est un isomorphisme par le lemme \ref{BCadm}. On obtient l'isomorphisme \eqref{Fontaine-Jannsen}. Enfin, par construction de $\alpha^0_r$, l'application ci-dessus est Galois-\'equivariante, elle est compatible avec l'action de $\varphi$ et de $N$ et avec la filtration apr\`es tensorisation par $\B_{\dR}$. 
\end{proof}

\section{Comparaison des morphismes de p\'eriodes}

On suppose que $\XF$ est un sch\'ema formel propre \`a r\'eduction semi-stable sur $\OC_K$, sans diviseur \`a l'infini. On a construit, dans la section pr\'ec\'edente, un isomorphisme : 
\begin{equation*}
\begin{split}
\tilde{\alpha}^0 & : H^i_{\text{\'et}}( \XF_{C}, \Q_p) \otimes_{\Q_p} \B_{\mathrm{st}} \xrightarrow{\sim} H^i_{\rm HK}( \XF) \otimes_{F} \B_{\mathrm{st}} \\
\tilde{\alpha}^0_{\rm dR} & : H^i_{\text{\'et}}( \XF_{C}, \Q_p) \otimes_{\Q_p} \B_{\mathrm{dR}} \xrightarrow{\sim} H^i_{\rm dR}( \XF_K) \otimes_{K} \B_{\mathrm{dR}}.
\end{split}
\end{equation*}
D'autres d\'emonstrations de la conjecture de Fontaine-Jannsen ont \'egalement \'et\'e donn\'ees par \v{C}esnavi\v{c}ius-Koshikawa (\cite{CK17}) et par Tsuji (\cite{Tsu99}). Le morphisme de p\'eriode utilis\'e par Tsuji est l'application de Fontaine-Messing, $\tilde{\alpha}^{\rm FM}$ d\'efinie ci-dessous. \v{C}esnavi\v{c}ius-Koshikawa d\'efinissent, eux, un morphisme de p\'eriode $\tilde{\alpha}^{\rm CK}$ en g\'en\'eralisant la construction de Bhatt-Morrow-Scholze (\cite{BMS16}) et dont on rappelle la construction plus loin.   
Le but de cette section est d'utiliser $\tilde{\alpha}^0$ pour montrer que ces deux morphismes de p\'eriode $\tilde{\alpha}^{\rm CK}$ et $\tilde{\alpha}^{\rm FM}$ co\"incident.

\subsection{Comparaison de $\tilde{\alpha}^{\rm{FM}}$ et $\tilde{\alpha}^{0}$ \label{section-compFM}}

On commence par comparer l'application de Fontaine-Messing avec le morphisme construit dans le chapitre pr\'ec\'edent. 

\begin{The}
\label{comparaison FM global}
Soit $\XF$ un sch\'ema formel propre semi-stable sur $\OC_K$. Les isomorphismes $\alpha^{\rm{FM}}_{r}$ et $\alpha^0_{r}$ de 
$H^i_{\text{\emph{{\'et}}}}( \XF_{C}, \Q_p(r))$ dans $H^i_{\mathrm{syn}}(\XF_{\OC_{C}}, r)_{\Q}$ sont \'egaux. 
En particulier, les morphismes de p\'eriodes 
\[ \tilde{\alpha}^{\rm{FM}}, \tilde{\alpha}^{0} : H^i_{\text{\emph{\'et}}}( \XF_{C}, \Q_p) \otimes_{\Q_p} \B_{\mathrm{st}} \to H^i_{\rm HK}( \XF) \otimes_{F} \B_{\mathrm{st}} \]
\[ \tilde{\alpha}_{\dR}^{\rm{FM}}, \tilde{\alpha}_{\dR}^{0} :H^i_{\text{\emph{\'et}}}( \XF_{C}, \Q_p) \otimes_{\Q_p} \B_{\dR} \to H^i_{\rm HK}( \XF) \otimes_{F} \B_{\dR} \]
sont \'egaux. 
\end{The}

Il suffit de montrer que les applications enti\`eres sont $p^{N}$-\'egales pour une certaine constante $N$. La proposition \ref{comparaison FM} donne l'\'egalit\'e locale. Pour obtenir l'\'egalit\'e entre les morphismes globaux, on va commencer par construire des applications $\alpha_{r, \Sigma, \Lambda}^{\rm FM}$ et les comparer au $\alpha^0_{r, \Sigma, \Lambda}$ d\'efinis plus haut.

On d\'efinit les anneaux suivants : 
\begin{itemize}
\item $\E_{R_{\Sigma, \Lambda}}^{PD}= S A$ (resp. $\E_{R_{\Sigma, \Lambda}}^{[u,v]}$) pour $S=A=R_{\Sigma, \Lambda}^{PD}$ (resp. $R_{\Sigma, \Lambda}^{[u,v]}$),
\item $\E_{R_{\Sigma,\Lambda, \infty}}^{PD}= S A$ (resp. $\E_{R_{\Sigma, \infty}}^{[u,v]}$) pour $S=R_{\Sigma, \Lambda}^{PD}$ (resp. $R_{\Sigma, \Lambda}^{[u,v]}$) et $A=\AAC(R_{\Sigma, \Lambda, \infty})$ (resp. $\A_{R_{\Sigma, \Lambda, \infty}}^{[u,v]}$),
\item $\E_{\overline{R}_{\Sigma, \Lambda}}^{PD}= S A$ (resp. $\E_{\overline{R}_{\Sigma, \Lambda}}^{[u,v]}$) pour $S=R_{\Sigma, \Lambda}^{PD}$  (resp. $R_{\Sigma, \Lambda}^{[u,v]}$) et $A= \AAC( \overline{R})$ (resp. $\A_{\overline{R}}^{[u,v]}$). 
\end{itemize}


On ne dispose plus ici des lemmes de Poincar\'e mais on peut montrer qu'on a un quasi-isomorphisme
\[ \tau_{\le r} \Kos(\varphi, \partial, F^r \E_{\overline{R}_{\Sigma, \Lambda}}^{PD}) \xleftarrow{\sim}  \tau_{\le r} \Kos(\varphi, \partial, F^r \AAC(\overline{R})) \]
en se ramenant au cas o\`u on ne travaille qu'avec une seule coordonn\'ees $\lambda \in \Lambda$. En effet, si on fixe un tel $\lambda$, l'application $R_{\cris, \lambda}^+ \to R_{\Sigma, \Lambda}^{PD}$ induit un diagramme commutatif : 
\[ \xymatrix{
\tau_{\le r} \Kos(\varphi, \partial, F^r \E_{\overline{R}_{\lambda}}^{PD})  \ar[d]^{\rotatebox{90}{$\sim$}} &\ar[l]^{\sim} \tau_{\le r} \Kos(\varphi, \partial, F^r \AAC(\overline{R})) \ar@{=}[d] \\
\tau_{\le r} \Kos(\varphi, \partial, F^r \E_{\overline{R}_{\Sigma, \Lambda}}^{PD})  & \ar[l] \tau_{\le r} \Kos(\varphi, \partial, F^r \AAC(\overline{R})) 
}
\] 
o\`u $\E_{\overline{R}_{\lambda}}^{PD}$ est l'anneau $SA$ pour $A=\AAC(\overline{R})$ et $S=R_{\cris, \lambda}^+$. La premi\`ere fl\`eche verticale est un quasi-isomorphisme car les complexes 
$\Kos(\varphi, \partial, F^r \E_{\overline{R}_{\lambda}}^{PD}) $ et $ \Kos(\varphi, \partial, F^r \E_{\overline{R}_{\Sigma, \Lambda}}^{PD})$ calculent tous les deux la cohomologie syntomique $\Syn(\overline{R},r)$. On en d\'eduit que la fl\`eche horizontale de la ligne du bas est un quasi-isomorphisme. 

On montre de la m\^eme fa\c{c}on qu'on a un quasi-isomorphisme : 
\[ \tau_{\le r} \Kos(\varphi, \partial, F^r \E_{\overline{R}_{\Sigma, \Lambda}}^{[u,v]}) \xleftarrow{\sim}  \tau_{\le r} \Kos(\varphi, \partial, F^r \A_{\overline{R}}^{[u,v]}).  \]

L'application de Fontaine-Messing $\alpha_{\Sigma, \Lambda}^{\rm FM}$ est donn\'ee par la compos\'ee : 
\begin{equation*}
\begin{split}
\tau_{\le r} \Kos(\varphi, \partial, F^rR_{\Sigma, \Lambda}^{PD}) \to \tau_{\le r} C(G_R, \Kos(\varphi, \partial, F^r \E_{\overline{R}_{\Sigma, \Lambda}}^{PD})) \xleftarrow{ \sim} \tau_{\le r}C(G_R, \Kos(\varphi, F^r \AAC(\overline{R}))) \\ \xleftarrow{\sim} \tau_{\le r} C(G_R, \Z_p(r)').
\end{split}
\end{equation*} 

On a le lemme suivant :  

\begin{Lem}
\label{egalSigmaLambda}
Il existe $N$ qui ne d\'epend ni de $R$, ni de $\Sigma$ et $\Lambda$ telle que l'application de Fontaine-Messing $\alpha_{r, \Sigma, \Lambda}^{\rm FM}$ est $p^{Nr}$-\'egale \`a l'application $\alpha^0_{r, \Sigma, \Lambda}$.
\end{Lem}

\begin{proof}
Le r\'esultat se d\'eduit d'un diagramme similaire \`a celui de la preuve de la proposition \ref{comparaison FM}. La diff\'erence principale est que les fl\`eches qui vont de la troisi\`eme vers la deuxi\`eme ligne sont maintenant des morphismes de bord. Il est aussi \`a noter que les fl\`eches horizontales qui vont de la troisi\`eme colonne vers la deuxi\`eme ne sont (a priori) plus des quasi-isomorphismes dans ce cas, \`a l'exception de celles de la premi\`ere, deuxi\`eme et derni\`ere ligne. 
\end{proof}

\begin{proof}[Preuve du th\'eor\`eme \ref{comparaison FM global}]
L'application globale \[\alpha_{r}^{\rm FM} : \tau_{\le r} R\Gamma_{\rm syn}(\XF_{\OC_C},r) \to \tau_{\le r} R \Gamma_{\text{\'et}}(\XF_C, \Z_p(r)') \] s'obtient de la m\^eme fa\c{c}on que dans la preuve du th\'eor\`eme \ref{comparaison globale} : elle est donn\'ee par la compos\'ee : 
\begin{equation}
\label{alphaFMtot}
\begin{split}
 \tau_{\le r} R \Gamma_{\text{syn}}(\XF_{\OC_{C}},r)_n \to \hocolim_{\rm HRF} \tau_{\le r} R\Gamma_{\text{syn}}(\mathfrak{U}_{\OC_C}^{\bullet},r) \to \hocolim_{\rm HRF} \tau_{\le r} R \Gamma(G_{R^{\bullet}}, \Z_p(r)')  \\ 
 \to \tau_{\le r} \hocolim_{\rm HRF} R\Gamma_{\text{\'et}}(\mathfrak{U}_C^{\bullet}, \Z_p(r)') \leftarrow \tau_{\le r} R \Gamma_{\text{\'et}}(\XF_{C}, \Z_p(r)').
 \end{split}
 \end{equation} 
o\`u le quasi-isomorphisme (fonctoriel) $\tau_{\le r} R\Gamma_{\rm{syn}}(\mathfrak{U}_{\OC_C}^{\bullet},r) \to \tau_{\le r} R \Gamma(G_{R^{\bullet}}, \Z_p(r)')$ pour $\mathfrak{U}^{\bullet}= \Spf(R^{\bullet}) \to \XF$ un hyper-recouvrement affine est induit par : 
\begin{equation}
\label{alphaFMKos}
\tau_{\le r} R \Gamma_{\rm{syn}}(\mathfrak{U}_{\OC_C}^{k}, r)  \xleftarrow[\eqref{qis Syn}]{\sim} \varinjlim_{(\Sigma_k, \Lambda_k)} \tau_{\le r} \Syn(\mathfrak{U}_{\OC_C}^k, \AAC(R_{\Sigma_k, \Lambda_k}^{\square}),r) \xrightarrow[\varinjlim \alpha_{r, \Sigma_k, \Lambda_k}^{\rm FM}]{\sim} \tau_{\le r} R \Gamma(G_{R^k}, \Z_p(r)')
\end{equation}
Le lemme \ref{egalSigmaLambda} donne la $p^{Nr}$-\'egalit\'e entre les morphismes $\varinjlim \alpha_{r, \Sigma_k, \Lambda_k}^{\rm FM}$ et $\varinjlim \alpha_{r, \Sigma_k, \Lambda_k}^{0}$ et donc entre les morphismes induits \eqref{alphaFMKos}. On en d\'eduit que $\alpha_{r}^{\rm FM} : \tau_{\le r} R\Gamma_{\rm syn}(\XF_{\OC_C},r) \to \tau_{\le r} R \Gamma_{\text{\'et}}(\XF_C, \Z_p(r)')$ et $\alpha_{r}^{0} : \tau_{\le r} R\Gamma_{\rm syn}(\XF_{\OC_C},r) \to \tau_{\le r} R \Gamma_{\text{\'et}}(\XF_C, \Z_p(r)')$ sont $p^{Nr}$-\'egales. En inversant $p$, on obtient l'\'egalit\'e entre les morphismes rationnels et on en d\'eduit finalement $\tilde{\alpha}^{\rm{FM}} = \tilde{\alpha}^{0}$.

\end {proof}

\subsection{Comparaison de $\tilde{\alpha}^{\rm{CK}}$ et $\tilde{\alpha}^{0}$}

On consid\`ere maintenant le morphisme de p\'eriode $\tilde{\alpha}^{\rm{CK}} : H^i_{\emph{\text{\'et}}}( \XF_{C}, \Q_p) \otimes_{\Q_p} \B_{\mathrm{st}} \xrightarrow{\sim} H^i_{\rm HK}( \XF) \otimes_{F} \B_{\mathrm{st}}$ d\'efini par Bhatt-Morrow-Scholze (\cite{BMS16}) dans le cas o\`u $\XF$ est \`a bonne r\'eduction et par \v{C}esnavi\v{c}ius-Koshikawa (\cite{CK17}) pour $\XF$ \`a r\'eduction semi-stable. On rappelle rapidement sa d\'efinition locale dans la section qui suit. On montre ensuite qu'il est \'egal \`a $\widetilde{\alpha}^{0}$.

\subsubsection{D\'efinition de l'application locale $\gamma^{\rm{CK}}$}

\begin{Def}\cite[\textsection 6]{BMS16}
Soit $A$ un anneau et $K$ un complexe de $A$-modules, sans $f$-torsion. On note $\eta_f(K)$ le sous-complexe de $K [\frac{1}{f}]$ dont le $i$-\`eme terme est donn\'e par : 
$$ (\eta_fK)^i:=\{x\in f^iK^i:dx\in f^{i+1}K^{i+1}\}. $$
\end{Def}

L'application $K^i \to (\eta_f K)^i, x \mapsto f^i x$ induit un isomorphisme $H^i(K)/H^i(K)[f]\cong H^i(\eta_fK)$ (o\`u $H^i(K)[f]$ d\'esigne la $f$-torsion de $H^i(K)$) et on en d\'eduit que si $K \to K'$ est un quasi-isomorphisme alors $\eta_fK \to \eta_fK'$ l'est aussi. On d\'efinit le foncteur $L \eta_f$ de $D(A)$ dans $D(A)$ de la fa\c{c}on suivante : pour $K$ dans $D(A)$ on choisit un complexe $C$ quasi-isomorphique \`a $K$ et dont les termes sont sans $f$-torsion et on pose 
$ L \eta_f K:= \eta_f C. $


Soit $\XF$ un sch\'ema formel sur $\OC_K$, \`a r\'eduction semi-stable et soit $X:= \XF_C$ la fibre g\'en\'erique de $\XF_{\OC_C}$, vue comme une vari\'et\'e analytique rigide. On note $\OC_X$ le faisceau structural du site pro-\'etale $X_{\text{pro\'et}}$, $\OC_X^+$ le faisceau int\'egral associ\'e et $\hat{O}_X^+$ sa compl\'etion. On d\'efinit le complexe de faisceaux $\A_{\rm{inf}, X}$ comme la compl\'etion $p$-adique d\'eriv\'ee\footnote{La compl\'etion $p$-adique d\'eriv\'ee $\widehat{K}$ d'un complexe $K \in D(X_{\proet}, \Z_p)$ est d\'efinie localement par $\widehat{K}|_U= \holim_n (K|_U \otimes_{\Z}^L \Z/p^n)$} de $W( \hat{\OC}_X^{+, \flat}).$ 
On consid\`ere alors le complexe de faisceaux suivant : 
\[ A \Omega_{\XF_{\OC_C}} := L \eta_{\mu}(R \nu_* (\A_{\rm{inf}, X})) \in \mathcal{D}(\XF_{\OC_C,\text{\'et}}, \A_{\text{inf}}) \]
o\`u $\nu : X_{\text{pro\'et}} \to \XF_{\OC_C,\text{\'et}}$ est la projection. 
Dans \cite[5.6]{BMS16}, il est montr\'e que l'application naturelle $\A_{\text{inf}} \to \A_{\text{inf},X}$ induit un morphisme 
\[ R \Gamma( X_{\text{\'et}}, \Z_p) \otimes^L_{\Z_p} \A_{\text{inf}} \cong R \Gamma( X_{\text{pro\'et}}, \Z_p) \otimes^L_{\Z_p} \A_{\rm{inf}} \to R \Gamma_{\text{pro\'et}}(X, \A_{\text{inf},X}) \]
tel que la cohomologie de son c\^one est tu\'ee par $W(\mathfrak{m}_C^{\flat})$ (o\`u $\mathfrak{m}_C$ est l'id\'eal maximal de $\OC_C$). Comme $\mu=[\varepsilon]-1$ est dans $W(\mathfrak{m}_C^{\flat})$, on en d\'eduit le r\'esultat suivant : 

\begin{The}\cite[2.3]{CK17}
\label{Ainf-etale}
Si $\XF$ est propre sur $\OC_K$, il existe un quasi-isomorphisme :
\begin{equation}
\label{Et-Ainf}
R \Gamma( X_{\text{\emph{\'et}}}, \Z_p) \otimes^L_{\Z_p} \A_{\mathrm{inf}} [ \frac{1}{\mu} ] \xrightarrow{\sim} R\Gamma_{\text{\emph{\'et}}}( \XF_{\OC_C}, A \Omega_{\XF_{\OC_C}}) \otimes^L_{\A_{\mathrm{inf}}} \A_{\mathrm{inf}} [ \frac{1}{\mu} ].
\end{equation}
\end{The}  

On suppose maintenant que $\XF_{\OC_C}= \Spf(R)$ avec $R$ \'etale sur $R_{\square}$. On reprend les notations $R_{\rm inf}^+$, $R_{\rm cris}^+$ de \textsection 2.2.2 et $\A_R^+$, $\A_{\rm cris}(R)$ de \textsection 5.1 (on rappelle que $R_{\rm inf}^+ \cong \A_R^+$ et $R_{\rm cris}^+ \cong \A_{\rm cris}(R)$). On d\'efinit $R_{\infty}$ comme avant. On a alors que $R_{\infty}$ est perfecto\"ide et muni d'une action de $\Gamma_R$. Ainsi, la tour $X_{\infty}:= "\varprojlim" \Spa(R_m [ \frac{1}{p}], R_m)$ est un recouvrement affino\"ide perfecto\"ide de $\XF_{\OC_C}$.

\begin{The}\cite[\textsection 3.3]{CK17}   
\begin{enumerate}
\item On a des quasi-isomorphismes : 
\begin{equation}
\label{continue-proetale}
\begin{split}
 L \eta_{(\zeta_p-1)} (R \Gamma( \Gamma_R, R_{\infty})) \xrightarrow{\sim} L \eta_{(\zeta_p-1)} (R \Gamma_{\text{\emph{pro\'et}}}( X, \hat{O}_X^{+})) \\
 L \eta_{\mu} (R \Gamma( \Gamma_R, \A_{\mathrm{inf}}(R_{\infty}))) \xrightarrow{\sim} L \eta_{\mu} (R \Gamma_{\text{\emph{pro\'et}}}( X, \A_{\mathrm{inf},X})).
 \end{split}
 \end{equation}
 \item Plus g\'en\'eralement, si $R'_{\infty}$ d\'efinit un recouvrement affino\"ide perfecto\"ide de $R$ muni d'une action d'un groupe $\Gamma'$ et tel que $R_{\infty}'$ est un raffinement de $R_{\infty}$, alors on a des quasi-isomorphismes :    
 \begin{equation}
 \label{continue-proetale(1)}
\begin{split}
 L \eta_{(\zeta_p-1)} (R \Gamma( \Gamma', R'_{\infty})) \xrightarrow{\sim} L \eta_{(\zeta_p-1)} (R \Gamma_{\text{\emph{pro\'et}}}( X, \hat{O}_X^{+})) \\
 L \eta_{\mu} (R \Gamma( \Gamma', \A_{\mathrm{inf}}(R'_{\infty}))) \xrightarrow{\sim} L \eta_{\mu} (R \Gamma_{\text{\emph{pro\'et}}}( X, \A_{\mathrm{inf},X})).
 \end{split}
 \end{equation}
 \end{enumerate}
 \end{The} 
 
Le but est de construire un quasi-isomorphisme entre $R\Gamma_{\rm cris}( (R/p) / \AAC)$ et $R \Gamma( X_{\text{\'et}}, \Z_p) \otimes^L_{\Z_p} \AAC$. L'application va \^etre induite par la compos\'ee : 
\[ \Kos(\partial, \A_{\rm cris}(R)) \xrightarrow{\tilde{\beta}} \eta_{\mu} \Kos(\Gamma_R, \A_{\rm cris}(R)) \to \eta_{\mu} \Kos( \Gamma_R, \A_{\rm cris}(R_{\infty})) \]
o\`u $\tilde{\beta}$ est construit ci-dessous. Le probl\`eme est que l'anneau $\AAC$ vient de la compl\'etion d'un anneau qui n'est pas de type fini sur $\A_{\inf}$ et en particulier, $\AAC/\mu$ n'est pas $p$-adiquement s\'epar\'e. Notamment, $\eta_{\mu} \Kos( \Gamma_R, \A_{\rm cris}(R_{\infty}))$ ne calcule pas n\'ecessairement $R \Gamma( X_{\text{\'et}}, \Z_p) \otimes^L_{\Z_p} \AAC$. Pour contourner cette difficult\'e, Bhatt-Morrow-Scholze passent par l'anneau $\A_{\rm cris}^{(m)}$, construit \`a partir du m\^eme anneau que $\AAC$ mais tronqu\'e en degr\'e $m$.

Plus pr\'ecis\'ement, pour $m$ dans $\N$, on note $\A_{\rm cris}^{(m)}$ la compl\'etion $p$-adique de 
 $ \A_{\text{inf}}[ \frac{\xi^k}{k!}, k \le m]$.
 On a toujours une surjection $\theta : \A_{\rm cris}^{(m)} \to \OC_C$ et on \'etend le Frobenius de $\A_{\text{inf}}$ en $\varphi : \A_{\rm cris}^{(m)} \to \A_{\rm cris}^{(m)}$. Remarquons que pour $m<p$, on a $\A_{\cris}^{(m)}=\A_{\inf}$ et pour $m \ge p$, les topologies $p$-adiques et $(p,\mu)$-adique sur $\A_{\cris}^{(m)}$ co\"incident. On peut de plus identifier $\AAC$ \`a la compl\'etion $p$-adique de $\varinjlim \A_{\rm cris}^{(m)}$. On d\'efinit de la m\^eme fa\c con $\A_{\rm cris}^{(m)}(\overline{R})$ et $\A_{\rm cris}^{(m)}(R_{\infty})$. Enfin, soit $\A_{\rm cris}^{(m)}(R):= \A^+_R \widehat{\otimes}_{\A_{\text{inf}}} \A_{\rm cris}^{(m)}$. 

On a des isomorphismes (voir \cite[12.7]{BMS16}) :
\begin{equation} 
\label{continue et Acris}
\begin{split}
R \Gamma( \Gamma_R, \A_{\text{inf}}(R_{\infty})) \widehat{\otimes}^L_{\A_{\text{inf}}} \A_{\rm cris}^{(m)}  & \xrightarrow{\sim} R \Gamma( \Gamma_R, \A_{\rm cris}^{(m)}(R_{\infty})) \\
  L \eta_{\mu} (R \Gamma( \Gamma_R, \A_{\text{inf}}(R_{\infty}))) \widehat{\otimes}^L_{\A_{\text{inf}}} \A_{\rm cris}^{(m)} & \xrightarrow{\sim} L \eta_{\mu}(R \Gamma( \Gamma_R, \A_{\rm cris}^{(m)}(R_{\infty}))).
  \end{split}
 \end{equation}
 En combinant les r\'esultats \eqref{continue-proetale} et \eqref{continue et Acris}, on obtient : 
 
 \begin{Cor}
 \label{Ainf-Kos}
 Il existe un quasi-isomorphisme naturel : 
\[  R \Gamma_{\text{\emph{\'et}}}(\XF_{\OC_C}, A \Omega_{\XF_{\OC_C}}) \widehat{\otimes}^L_{\A_{\mathrm{inf}}} \AAC \xrightarrow{\sim} (\varinjlim_{m} ( \eta_{\mu} (\Kos(\Gamma_R,\A_{\rm{cris}}^{(m)}(R_{\infty}))))^{\widehat{~}}. \]
\end{Cor}
 
On a alors le th\'eor\`eme (voir \cite[Th. 12.1]{BMS16} et \cite[Th. 5.4]{CK17}) :

\begin{The}
\label{comparaison cris}
Il existe un quasi-isomorphisme Frobenius-\'equivariant : 
\begin{equation}
\gamma^{\rm{CK}} : R \Gamma_{\mathrm{cris}}( (\XF_{\OC_C}/p)/ \AAC) \xrightarrow{\sim} R \Gamma_{\text{\emph{\'et}}}(\XF_{\OC_C}, A \Omega_{\XF_{\OC_C}}) 
\widehat{\otimes}^L_{\A_{\mathrm{inf}}} \AAC.
\end{equation}
\end{The}

\begin{proof}
Les d\'etails de la preuve sont donn\'es dans \cite[\textsection 5.5 \`a \textsection 5.16]{CK17}. On rappelle ici rapidement la construction de $\gamma^{\rm CK}$. Pour $m \ge p^2$, on d\'efinit $\widetilde{\beta}^{(m)} : \Kos(\partial, \A_{\rm cris}^{(m)}(R)) \to \eta_{\mu} \Kos(\Gamma_R, \A_{\rm cris}^{(m)}(R))$ par : 
\[
\xymatrix{ 
\A_{\rm{cris}}^{(m)}(R) \ar[r]^-{(\partial_j)} \ar[d]^{\mathrm{Id}} & (\A_{\rm cris}^{(m)}(R))^{J_1} \ar[r] \ar[d]^{(\tilde{\beta}_j)} & \cdots \ar[r] & (\A_{\rm cris}^{(m)}(R))^{J_n} \ar[r] \ar[d]^{(\tilde{\beta}_{j_1} \cdots \tilde{\beta}_{j_n})} & \cdots \\
\A_{\rm{cris}}^{(m)}(R) \ar[r]^-{(\gamma_j-1) \quad} & (\mu \A_{\rm cris}^{(m)}(R))^{J_1} \ar[r] & \cdots \ar[r] & (\mu^n \A_{\rm cris}^{(m)}(R))^{J_n} \ar[r] & \cdots
}
\]
avec $\tilde{\beta}_j := \sum_{n \ge 1} \frac{t^n}{n!} \partial_j^{n-1}$. L'application $\widetilde{\beta}^{(m)}$ est bien d\'efinie puisqu'elle est induite par le morphisme de complexes :
\[ \xymatrix{
\A_{\rm cris}^{(m)}(R) \ar[r]^{\partial_j} \ar[d]^{\text{Id}} & \A_{\rm cris}^{(m)}(R) \ar[d]^{\widetilde{\beta}_j} \\
\A_{\rm cris}^{(m)}(R) \ar[r]^{\gamma_j-1} & \A_{\rm cris}^{(m)}(R) 
} \]
et $\tilde{\beta}_j \partial_j= \gamma_j-1$ (en utilisant $\Kos(\partial, \A_{\rm cris}^{(m)}(R))= \bigotimes_{j=1}^d (\A_{\rm cris}^{(m)}(R) \xrightarrow{\partial_j} \A_{\rm cris}^{(m)}(R))$). Ce sont des isomorphismes par le lemme 5.15 de \cite{CK17}. Pour obtenir le r\'esultat, il suffit ensuite de montrer que l'inclusion $\A_{\rm cris}^{(m)}(R) \hookrightarrow \A_{\rm cris}^{(m)}(R_{\infty})$ induit un isomorphisme 
\begin{equation}
\label{R-Rinfty}
 \eta_{\mu} \Kos(\Gamma_R, \A_{\rm{cris}}^{(m)}(R)) \xrightarrow{\sim} \eta_{\mu} \Kos(\Gamma_R, \A_{\rm{cris}}^{(m)}(R_{\infty})).
 \end{equation}
  Pour cela, on \'ecrit $\A^{(m)}_{\rm cris}(R_{\infty})$ comme une somme $\A^{(m)}_{\rm cris}(R) \oplus (N_{\infty}\widehat{\otimes}_{\AAinf} \A_{\rm cris}^{(m)})$ (voir \cite[3.14.5]{CK17}) avec $\mu \cdot H^i(\Gamma_R, N_{\infty} \widehat{\otimes}_{\AAinf} \A_{\rm cris}^{(m)})=0$ (voir \cite[3.32]{CK17}).  

Finalement, le quasi-isomorphisme $\gamma^{\rm CK}$ est donn\'e par la compos\'ee : 
\begin{equation*}
\begin{split}
 R \Gamma_{\mathrm{cris}}( (\XF_{\OC_C}/p)/ \AAC)&  \xrightarrow{\sim} (\varinjlim_{m} \Kos(\partial, \A_{\rm cris}^{(m)}(R)))^{\widehat{~}} \\
 & \xrightarrow[\sim]{\tilde{\beta}}  (\varinjlim_{m} \eta_{\mu} \Kos(\Gamma_R, \A_{\rm cris}^{(m)}(R)))^{\widehat{~}} \\
 & \xrightarrow{\sim} (\varinjlim_{m}  \eta_{\mu} \Kos(\Gamma_R, \A_{\rm cris}^{(m)}(R_{\infty})))^{\widehat{~}} \\
 & \xleftarrow[\sim]{\eqref{Ainf-Kos}} R \Gamma_{\text{\'et}}(\XF_{\OC_C}, A \Omega_{\XF_{\OC_C}}) 
\widehat{\otimes}^L_{\A_{\mathrm{inf}}} \AAC. 
\end{split}
\end{equation*}
\end{proof}
 
\subsubsection{Application globale}

Comme pr\'ec\'edemment, on va \'ecrire les complexes du th\'eor\`eme \ref{comparaison cris} de mani\`ere fonctorielle. 
On commence par le complexe de droite. On prend $(\Sigma, \Lambda)$ comme en \eqref{im fermee} et on d\'efinit $\Gamma_{\Sigma, \Lambda}$, $R^{\square}_{\Sigma, \Lambda, \infty}$ et $R_{\Sigma, \Lambda, \infty}$ comme avant. 
 On a alors un recouvrement pro-\'etale affino\"ide perfecto\"ide : 
 \begin{equation}
 \label{cover1}
  \Spa(R_{\Sigma, \Lambda, \infty}[ \frac{1}{p}], R_{\Sigma, \Lambda, \infty}) \to \Spa(R[\frac{1}{p}], R).
  \end{equation}
 Le raisonnement pr\'ec\'edent (et en particulier l'isomorphisme \eqref{continue-proetale(1)}) donne un quasi-isomorphisme : 
 \begin{equation}
 \label{etale et sigma} 
 R \Gamma_{\text{\'et}}(\XF_{\OC_C}, A \Omega_{\XF_{\OC_C}}) \widehat{\otimes}^L_{\A_{\rm{inf}}} \AAC \xrightarrow{\sim} (\varinjlim_{m}  \eta_{\mu} \Kos(\Gamma_R,\A_{\rm{cris}}^{(m)}(R_{\Sigma, \Lambda, \infty})))^{\widehat{~}}. 
 \end{equation}
On va maintenant \'ecrire le complexe $\Omega^{\bullet}_{R_{\Sigma, \Lambda}^{PD}/ \AAC}$ sous une forme semblable \`a celle de \eqref{etale et sigma}. On consid\`ere l'exactification de l'immersion $\Spec(R/p) \hookrightarrow \Spf(\A_{\rm cris}(R_{\Sigma, \Lambda}^{\square}))$ (voir \cite[\textsection 5.25 \`a 5.31]{CK17}) :
\[ \Spec(R/p) \xhookrightarrow{j} \mathcal{Y} \xrightarrow{q} \Spf(\A_{\rm cris}(R_{\Sigma, \Lambda}^{\square})) \]
avec $j$ immersion ferm\'ee exacte et $q$ log-\'etale. Si on note $D_j$ la log-PD-enveloppe associ\'ee \`a $j$, on a alors un isomorphisme $R_{\Sigma, \Lambda}^{PD} \xrightarrow{\sim} \widehat{D_j}$. On note $D_j^{(m)}$ l'anneau des puissances divis\'ees de degr\'e inf\'erieur \`a $m$ et $R_{\Sigma, \Lambda}^{PD, (m)}$ la compl\'etion de l'image de $D_j^{(m)}$ par cet isomorphisme.   

Les $R_{\Sigma, \Lambda}^{PD,(m)}$ sont alors des $\A_{\rm cris}^{(m)}$-alg\`ebres munies d'une action continue de $\Gamma_{\Sigma, \Lambda}$ et d'un Frobenius. On a de plus un $\AAC$-isomorphisme $R_{\Sigma, \Lambda}^{PD} \cong ( \varinjlim_{m} R_{\Sigma, \Lambda}^{PD,(m)})^{\widehat{~}}$, qui est \'equivariant pour l'action du Frobenius et de $\Gamma_{\Sigma, \Lambda}$. On obtient : 
 \begin{equation}
 \label{galois et sigma}
R \Gamma_{\rm cris}( (\XF_{\OC_C}/p)/ \AAC) \xrightarrow{\sim} (\varinjlim_m \Kos( \partial, R_{\Sigma, \Lambda}^{PD,(m)}))^{\widehat{~}}.
\end{equation} 
 
 Enfin, en se ramenant au cas o\`u on ne consid\`ere qu'une seule coordonn\'ees $\lambda$ (donn\'e par \ref{comparaison cris}), on construit un isomorphisme fonctoriel (voir \cite[\textsection 5.38-5.39]{CK17}) : 
 \begin{equation}
 \label{cris-galois sigma}
 \tilde{\beta}_{\Sigma, \Lambda} : (\varinjlim_m \Kos( \partial, R_{\Sigma, \Lambda}^{PD,(m)}))^{\widehat{~}} \xrightarrow{\sim}  (\varinjlim_{m}  \eta_{\mu} \Kos(\Gamma_{\Sigma, \Lambda}, \A_{\rm{cris}}^{(m)}(R_{\Sigma, \Lambda, \infty})))^{\widehat{~}}. 
 \end{equation}
 
 On peut maintenant prouver la version globale du th\'eor\`eme \ref{comparaison cris} :
  
 \begin{The}\cite[5.43]{CK17}
 \label{gammaCK}
 Soit $\XF$ un sch\'ema formel propre sur $\OC_K$ \`a r\'eduction semi-stable. Il existe un quasi-isomorphisme compatible avec le morphisme de Frobenius : 
\[ \gamma^{\rm{CK}} : R \Gamma_{\mathrm{cris}}( (\XF_{\OC_C}/p)/ \AAC) \xrightarrow{\sim} R \Gamma_{\text{\emph{\'et}}}(\XF_{\OC_C}, A \Omega_{\XF_{\OC_C}}) 
\widehat{\otimes}^L_{\A_{\mathrm{inf}}} \AAC. \]
 \end{The}
 
 \begin{proof}
 La preuve est similaire \`a celle du th\'eor\`eme \ref{comparaison globale}. Le morphisme global $\gamma^{\rm CK}$ est donn\'e par la compos\'ee : 
\begin{equation*}
\begin{split}
R \Gamma_{\rm{cris}}( (\XF_{\OC_C}/p)/ \AAC) & \xrightarrow{\sim} \hocolim_{\rm HRF} R \Gamma_{\rm{cris}}( (\mathfrak{U}^{\bullet}_{\OC_C}/p)/ \AAC)    \\
& \xrightarrow{\sim} \hocolim_{\rm HRF} R \Gamma_{\text{\'et}}(\mathfrak{U}^{\bullet}_{\OC_C}, A \Omega_{\mathfrak{U}^{\bullet}_{\OC_C}})  \\
&  \xleftarrow{\sim} R \Gamma_{\text{\'et}}(\XF_{\OC_C}, A \Omega_{\XF_{\OC_C}})
\end{split}
\end{equation*}
o\`u, si $\mathfrak{U}^{\bullet} \to \XF$ est un hyper-recouvrement affine de $\XF$ avec $\mathfrak{U}_{\OC_C}^{k}:= \Spf(R^k)$, le quasi-isomorphisme $R \Gamma_{\rm{cris}}( (\mathfrak{U}^{\bullet}_{\OC_C}/p)/ \AAC) \xrightarrow{\sim} R \Gamma_{\text{\'et}}(\mathfrak{U}^{\bullet}_{\OC_C}, A \Omega_{\mathfrak{U}^{\bullet}_{\OC_C}})$ est induit par :
\begin{equation*}
\begin{split}
 R \Gamma_{\rm{cris}}( (R^k/p)/ \AAC) & \xrightarrow[\eqref{galois et sigma}]{\sim} \varinjlim_{(\Sigma_k, \Lambda_k)} (\varinjlim_m \Kos( \partial, R_{\Sigma_k, \Lambda_k}^{PD,(m)}))^{\widehat{~}} \\
 & \xrightarrow[\eqref{cris-galois sigma}]{\sim} \varinjlim_{(\Sigma_k, \Lambda_k)} (\varinjlim_{m} \eta_{\mu} \Kos(\Gamma_{\Sigma, \Lambda}, \A_{\rm{cris}}^{(m)}(R_{\Sigma, \Lambda, \infty})))^{\widehat{~}} \\
 &  \xleftarrow[\eqref{etale et sigma}]{\sim} R \Gamma_{\text{\'et}}(\XF_{\OC_C}, A \Omega_{\XF_{\OC_C}}) \widehat{\otimes}^L_{\A_{\rm{inf}}} \AAC. 
 \end{split}
 \end{equation*}
 \end{proof}
 
On a un quasi-isomorphisme (il d\'epend du choix de $\varpi$ et est compatible avec les actions du Frobenius et de l'op\'erateur de monodromie \cite[9.2]{CK17}) :
 \begin{equation}
 \iota_{{\rm cris}, \varpi}^{\rm B} : R\Gamma_{\rm cris}(\XF_k/W(k)^0) \otimes^L_{W(k)} \B_{\rm st}^+ \xrightarrow{\sim} R\Gamma_{\rm cris}((\XF_{\OC_C}/p)/\AAC) \otimes^L_{\AAC} \B_{\rm st}^+ \end{equation}
 et on en d\'eduit l'isomorphisme de p\'eriode $\tilde{\alpha}^{\rm CK}$ : 
 
 \begin{The}\cite[9.5]{CK17}
 \label{SemistCK}
Soit $\XF$ un sch\'ema formel propre sur $\OC_K$ \`a r\'eduction semi-stable. Il existe un quasi-isomorphisme :
\[ \tilde{\alpha}^{\rm CK} : R\Gamma_{\emph{\text{\'et}}}(\XF_C, \Z_p) \otimes_{\Z_p}^L \Bst \xrightarrow{\sim} R\Gamma_{\rm cris}(\XF_k/W(k)^0) \otimes^L_{W(k)} \B_{\rm st}. \]
De plus, $\tilde{\alpha}^{\rm CK}$ est compatible avec les actions du morphisme de Frobenius, du groupe de Galois et de l'op\'erateur de monodromie et il induit un isomorphisme filtr\'e apr\`es tensorisation par $\Bdr$.  
 \end{The}
  
 \begin{proof}
 Le morphisme $\tilde{\alpha}^{\rm CK}$ est donn\'e par la compos\'ee : 
 \begin{equation*}
 \begin{split}
R\Gamma_{\text{\'et}}(\XF_C, \Z_p) \otimes_{\Z_p}^L \Bst & \xrightarrow[\sim]{\eqref{Ainf-etale}} R \Gamma_{\text{\'et}}(\XF_{\OC_C}, A \Omega_{\XF_{\OC_C}}) 
\otimes^L_{\A_{\mathrm{inf}}} \Bst \\
& \xleftarrow[\sim]{\gamma^{\rm CK}} R \Gamma_{\mathrm{cris}}( (\XF_{\OC_C}/p)/ \AAC)\otimes^L_{\AAC} \B_{\rm st} \\
& \xleftarrow[\sim]{\iota^B_{{\rm cris}, \varpi}} R\Gamma_{\rm cris}(\XF_k/W(k)^0) \otimes^L_{W(k)} \B_{\rm st}.
\end{split}
\end{equation*}
 \end{proof}
 En inversant $p$, on obtient un quasi-isomorphisme rationnel : 
 \[ \tilde{\alpha}^{\rm CK} : R\Gamma_{\text{\'et}}(\XF_C, \Q_p) \otimes_{\Q_p}^L \Bst \xrightarrow{\sim} R\Gamma_{\rm HK}(\XF) \otimes^L_{F} \B_{\rm st}. \] 
 
 \subsubsection{Comparaison avec $\tilde{\alpha}^{\text{0}}$}

Soit $\XF$ un sch\'ema formel propre sur $\OC_K$ \`a r\'eduction semi-stable, le but de cette section est de prouver le th\'eor\`eme suivant :  
 
 \begin{The}
 \label{correct1}
Les morphismes de p\'eriodes 
\[ \tilde{\alpha}^{\rm{CK}}, \tilde{\alpha}^{0} : H^i_{\text{\emph{\'et}}}( \XF_{C}, \Q_p) \otimes_{\Q_p} \B_{\mathrm{st}} \to H^i_{\rm HK}( \XF) \otimes_{F} \B_{\mathrm{st}} \]
\[ \tilde{\alpha}_{\dR}^{\rm{CK}}, \tilde{\alpha}_{\dR}^{0} :H^i_{\text{\emph{\'et}}}( \XF_{C}, \Q_p) \otimes_{\Q_p} \B_{\dR} \to H^i_{\rm HK}( \XF) \otimes_{F} \B_{\dR} \]
sont \'egaux. En particulier, on a $\tilde{\alpha}^{\rm{CK}}=\tilde{\alpha}^{\rm{FM}}$ et $\tilde{\alpha}^{\rm CK}_{\dR}= \tilde{\alpha}^{\FM}_{\dR}$.  
 \end{The}
 
  \begin{proof}Soit $r\geq 0$. 
 On rappelle que le morphisme $\tilde{\alpha}^0$ est donn\'e par la compos\'ee (voir \eqref{tildealpha0}) : 
\begin{equation*}
\begin{split}
\tau_{\le r} R \Gamma_{\text{\'et}} (\XF_C, \Q_p) \otimes^L_{\Q_p} \Bst \xleftarrow[\sim]{t^r} \tau_{\le r} R \Gamma_{\text{\'et}} (\XF_C, \Q_p(r)) \otimes^L_{\Q_p} \Bst \{ r \} \xleftarrow[\alpha^0_r]{\sim} \tau_{\le r} R\Gamma_{\rm syn}(\XF_{\OC_C},r)_{\Q} \otimes_{\Q_p}^L \Bst \{ r \}  \\ \xrightarrow{\rm can} R \Gamma_{\mathrm{cris}}( (\XF_{\OC_C}/p)/ \AAC)_{\Q} \otimes^L_{\AAC[\frac{1}{p}]} \B_{\rm st} \xleftarrow[\sim]{\iota^B_{{\rm cris}, \varpi}} \tau_{\le r} R\Gamma_{\rm HK}(\XF) \otimes^L_{F} \Bst
\end{split}
\end{equation*} 
 et $\tilde{\alpha}^{\rm CK}$ par : 
\begin{equation*}
\hspace{-1cm}\begin{split}
 R\Gamma_{\text{\'et}}(\XF_C, \Q_p) \otimes_{\Q_p}^L \Bst  \xrightarrow[\sim]{\eqref{Ainf-etale}} R \Gamma_{\text{\'et}}(\XF_{\OC_C}, A \Omega_{\XF_{\OC_C}})_{\Q} 
\otimes^L_{\A_{\mathrm{inf}}[\frac{1}{p}]} \Bst  \xleftarrow[\sim]{\gamma^{\rm CK}} R \Gamma_{\mathrm{cris}}( (\XF_{\OC_C}/p)/ \AAC)_{\Q} \otimes^L_{\AAC[\frac{1}{p}]} \B_{\rm st} 
\\  \xleftarrow[\sim]{\iota^B_{{\rm cris}, \varpi}} R\Gamma_{\rm HK}(\XF) \otimes^L_{F} \B_{\rm st}\end{split}
\end{equation*}

Pour prouver le th\'eor\`eme, il suffit de montrer que pour $r\geq 2d$, $d=\dim(X)$, le rectangle ext\'erieur du diagramme suivant commute (au moins au niveau des cohomologies) :  
 \[ \hspace{-1cm} {\footnotesize
\xymatrix{
\tau_{\le r} R\Gamma_{\text{\'et}}(\XF_C, \Q_p) \otimes_{\Q_p}^L \Bst  \ar[d]^-{\rotatebox{90}{$\sim$}}_{\eqref{Ainf-etale}} & \ar[l]_-{\sim}^-{t^r} \tau_{\le r}  R \Gamma_{\text{\'et}} (\XF_C, \Q_p(r)) \otimes^L_{\Q_p} \Bst\{r\} \ar[dl]_{f_r}&  
\ar[l]^-{\alpha^0_r}_-{\sim}\tau_{\le r}  R\Gamma_{\rm syn}(\XF_{\OC_C},r)_{\Q} \otimes_{\Q_p}^L \Bst\{r\}  \ar[d] \ar[dl]_{{\rm can}}\\
\tau_{\le r}  R \Gamma_{\text{\'et}}(\XF_{\OC_C}, A \Omega_{\XF_{\OC_C}})_{\Q} \otimes^L_{\A_{\mathrm{inf}}[\frac{1}{p}]} \Bst & \ar[l]^-{\sim}_{\gamma^{\rm CK}} \tau_{\le r} R \Gamma_{\mathrm{cris}}( (\XF_{\OC_C}/p)/ \AAC)_{\Q} \otimes^L_{\AAC[\frac{1}{p}]} \B_{\rm st} &  \ar[l]^-{\sim}_-{\iota^B_{{\rm cris}, \varpi}} \tau_{\le r} R\Gamma_{\rm HK}(\XF) \otimes^L_{F} \B_{\rm st} 
}}\]
L'application $f_r$ est d\'efinie de fa\c{c}on \`a ce que le triangle de gauche soit commutatif. Le triangle de droite est commutatif par \cite[\textsection 3.1]{NN16}, il suffit donc de montrer que le trap\`eze int\'erieur commute. Il est suffisant de le montrer au niveau des cohomologies et c'est le r\'esultat du lemme \ref{egalite1} ci-dessous.
 \end{proof} 
 
  \begin{Lem}
 \label{egalite1}
Soit $i\leq r$. 
Le diagramme suivant est commutatif : 
 \[ \xymatrix{
H^i_{\text{\emph{\'et}}}(\XF_{\OC_C}, A \Omega_{\XF_{\OC_C}})_{\Q} \otimes^L_{\A_{\mathrm{inf}}[\frac{1}{p}]} \Bst & \ar[l]^-{\gamma^{\rm CK}}_-{\sim} H^i_{\mathrm{cris}}( (\XF_{\OC_C}/p)/ \AAC)_{\Q} \otimes^L_{\AAC[\frac{1}{p}]} \B_{\rm st} \\
H^i_{\text{\emph{\'et}}}(\XF_C, \Q_p(r))  \otimes^L_{\Q_p} \Bst \{ r \}  \ar[u]^{f_r} & H^i_{\rm syn}(\XF_{\OC_C},r)_{\Q} \otimes^L_{\Q_p} \Bst \{ r \} \ar[u]^{{\rm can}} \ar[l]^{\alpha_r^0}_{\sim} 
 }\]
 \end{Lem} 
 
  \begin{proof}
 (i) {\em R\'eduction.} 
Pour prouver le lemme, par $\B^+_{\rm cris}[\frac{1}{\mu}] $-lin\'earit\'e, il suffit de voir que le diagramme suivant est commutatif :
\[ \xymatrix{
H^i_{\text{\emph{\'et}}}(\XF_{\OC_C}, A \Omega_{\XF_{\OC_C}})_{\Q} \otimes^L_{\A_{\mathrm{inf}}[\frac{1}{p}]}  \B^+_{\rm cris}[\frac{1}{\mu}] & \ar[l]^-{\gamma^{\rm CK}}_-{\sim} H^i_{\mathrm{cris}}( (\XF_{\OC_C}/p)/ \AAC)_{\Q} \otimes^L_{\A_{\rm cris}[\frac{1}{p}]}\B^+_{\rm cris}[\frac{1}{\mu}]  \\
H^i_{\text{\'et}}(\XF_C, \Q_p(r)) \ar[u]^{f_r} & H^i_{\rm syn}(\XF_{\OC_C},r)_{\Q} \ar[u]^{{\rm can}} \ar[l]^{\alpha_r^0}_{\sim} 
 }\]
Consid\'erons le diagramme :  
  \begin{equation}
   \label{paris1}\xymatrix{H^i_{\text{\'et}}(\XF_{\OC_C}, A \Omega_{\XF_{\OC_C}})_{\Q} \otimes^L_{\A_{\mathrm{inf}}[\frac{1}{p}]} \A^{[u,v]}[\frac{1}{p},\frac{1}{\mu}] & \ar[l]^-{\gamma^{\rm CK}_{[u,v]}}_-{\sim} H^i_{\mathrm{cris}}( (\XF_{\OC_C}/p)/ \AAC)_{\Q} \otimes^L_{\A_{\rm cris}[\frac{1}{p}]}\A^{[u,v]}[\frac{1}{p},\frac{1}{\mu}]  \\
H^i_{\text{\'et}}(\XF_{\OC_C}, A \Omega_{\XF_{\OC_C}})_{\Q} \otimes^L_{\A_{\mathrm{inf}}[\frac{1}{p}] } \AAC[\frac{1}{p},\frac{1}{\mu}]  \ar[u]^{{\rm Id}\otimes{\rm can}} & \ar[l]^-{\gamma^{\rm CK}}_-{\sim} H^i_{\mathrm{cris}}( (\XF_{\OC_C}/p)/ \AAC)_{\Q} \otimes^L_{\A_{\rm cris}[\frac{1}{p}]}\AAC[\frac{1}{p},\frac{1}{\mu}] \ar[u]^{{\rm Id}\otimes{\rm can}} \\
H^i_{\text{\'et}}(\XF_C, \Q_p(r))  \ar[u]^{f_r} & H^i_{\rm syn}(\XF_{\OC_C},r)_{\Q} \ar[u]^{{\rm can}} \ar[l]^{\alpha_r^0}_{\sim} 
 }
 \end{equation}
 
La fl\`eche verticale de gauche $\beta_1:={\rm Id}\otimes{\rm can}$ est injective. En effet, par le th\'eor\`eme \ref{Ainf-etale}, on a
\begin{align*}
  H^i_{\text{\'et}}(\XF_{\OC_C}, A \Omega_{\XF_{\OC_C}})_{\Q} \otimes^L_{\A_{\mathrm{inf}}[\frac{1}{p}]} \A_{\mathrm{inf}}[\frac{1}{p},\frac{1}{\mu}] \stackrel{\sim}{\leftarrow}H^i_{\text{\'et}}(\XF_C, \Q_p(r))  \otimes^L_{\Q_p}\A_{\mathrm{inf}}[\frac{1}{p},\frac{1}{\mu}] ,
  \end{align*}
et on en d\'eduit les isomorphismes : 
  \begin{align*}
  & H^i_{\text{\'et}}(\XF_{\OC_C}, A \Omega_{\XF_{\OC_C}})_{\Q} \otimes^L_{\A_{\mathrm{inf}}[\frac{1}{p}]} \A^{[u,v]}[\frac{1}{p},\frac{1}{\mu}]\stackrel{\sim}{\leftarrow}H^i_{\text{\'et}}(\XF_C, \Q_p(r))  \otimes^L_{\Q_p}\A^{[u,v]}[\frac{1}{p},\frac{1}{\mu}],\\
  & H^i_{\text{\'et}}(\XF_{\OC_C}, A \Omega_{\XF_{\OC_C}})_{\Q} \otimes^L_{\A_{\mathrm{inf}}[\frac{1}{p}] } \AAC[\frac{1}{p},\frac{1}{\mu}] \stackrel{\sim}{\leftarrow}H^i_{\text{\'et}}(\XF_C, \Q_p(r))  \otimes^L_{\Q_p}\AAC[\frac{1}{p},\frac{1}{\mu}]. 
\end{align*}
En utilisant que l'application canonique $\AAC[\frac{1}{p},\frac{1}{\mu}] \to \A^{[u,v]}[\frac{1}{p},\frac{1}{\mu}]$ est injective et $H^i_{\text{\'et}}(\XF_C, \Q_p(r)) $ est de rang fini sur $\Q_p$ on obtient que $\beta_1$ est injective.
 
On en d\'eduit que pour prouver le lemme, il suffit de montrer que le carr\'e ext\'erieur du diagramme (\ref{paris1}) commute. Mais en utilisant une nouvelle fois le th\'eor\`eme \ref{Ainf-etale} puis le th\'eor\`eme \ref{gammaCK} (et que $H^i_{\text{\'et}}(\XF_C, \Q_p(r)) $ est de $\Q_p$-rang fini), on a 
 \begin{align*}
 H^i_{\text{\'et}}(\XF_{\OC_C}, A \Omega_{\XF_{\OC_C}})_{\Q} \otimes^L_{\A_{\mathrm{inf}}[\frac{1}{p}]} \A^{[u,v]} [\frac{1}{p},\frac{1}{\mu}]  & \stackrel{\sim}{\to}
 H^i_{\text{\'et}}(R\Gamma(\XF_{\OC_C}, A \Omega_{\XF_{\OC_C}}) \otimes^L_{\A_{\mathrm{inf}}[\frac{1}{p}]} \A^{[u,v]}[\frac{1}{p},\frac{1}{\mu}] ),\\
 H^i_{\mathrm{cris}}( (\XF_{\OC_C}/p)/ \AAC)_{\Q} \otimes^L_{\AAC[\frac{1}{p}]}\A^{[u,v]}  [\frac{1}{p},\frac{1}{\mu}] & \stackrel{\sim}{\to} H^i(R\Gamma_{\mathrm{cris}}( (\XF_{\OC_C}/p)/ \AAC)_{\Q} \otimes ^L_{\AAC[\frac{1}{p}]}\A^{[u,v]}  [\frac{1}{p},\frac{1}{\mu}] )
 \end{align*}
et il suffit donc de montrer que le diagramme
\begin{equation}
\label{paris2}
\xymatrix{
\tau_{\leq r}R\Gamma_{\text{\'et}}(\XF_{\OC_C}, A \Omega_{\XF_{\OC_C}})_{\Q} \otimes^L_{\A_{\mathrm{inf}}[\frac{1}{p}]} \A^{[u,v]}[\frac{1}{p}, \frac{1}{\mu}] & \ar[l]^-{\gamma^{\rm CK}_{[u,v]}}_-{\sim} \tau_{\leq r}R\Gamma_{\mathrm{cris}}( (\XF_{\OC_C}/p)/ \AAC)_{\Q} \otimes^L_{\A_{\rm cris}[\frac{1}{p}]}\A^{[u,v]}[\frac{1}{p}, \frac{1}{\mu}]  \\
\tau_{\leq r}R\Gamma_{\text{\'et}}(\XF_C, \Q_p(r))  \ar[u]^{f_r} & \tau_{\leq r}R\Gamma_{\rm syn}(\XF_{\OC_C},r)_{\Q} \ar[u]^{{\rm can}} \ar[l]^{\alpha_r^0}_{\sim} 
}
\end{equation}
commute.

Pour cela, on va montrer que le diagramme suivant est $p^{N}$-commutatif (pour un $N \in \N$) :
\begin{equation}
\label{entier1}
\xymatrix{
\tau_{\leq r}R\Gamma_{\text{\'et}}(\XF_{\OC_C}, A \Omega_{\XF_{\OC_C}}) \widehat{\otimes}^L_{\A_{\mathrm{inf}}} \A^{[u,v]}[\frac{1}{\mu}] & \ar[l]^-{\gamma^{\rm CK}_{[u,v]}}_-{\sim} \tau_{\leq r}R\Gamma_{\mathrm{cris}}( (\XF_{\OC_C}/p)/ \AAC) \widehat{\otimes}^L_{\A_{\rm cris}}\A^{[u,v]}[\frac{1}{\mu}]  \\
\tau_{\leq r}R\Gamma_{\text{\'et}}(\XF_C, \Z_p(r))  \ar[u]^{f_r} & \tau_{\leq r}R\Gamma_{\rm syn}(\XF_{\OC_C},r) \ar[u]^{{\rm can}} \ar[l]^{\alpha_r^0}_{\sim} 
}
\end{equation}

 (ii) {\em Calcul local.} On suppose dans un premier temps que $\XF= \Spf(R)$ avec $R$ la compl\'etion d'une alg\`ebre \'etale sur $R_{\square}$. On va voir que la principale diff\'erence entre la construction du morphisme de \cite{CK17} et celle de $\alpha_r^0$ est que l'application $\tilde{\beta}$ utilis\'e par Bhatt-Morrow-Scholze est \'egale \`a l'inverse de l'application $\beta$ construite plus haut, tordue par $t$. Plus pr\'ecis\'ement, on va montrer qu'on a un diagramme $p^{c(r)}$-commutatif (pour une certaine constante $c(r) \in \N$) :  
 \begin{equation}
\label{CK-CN}
\begin{footnotesize}
\hspace{-1.5cm}
\xymatrix{K_{\varphi, \partial}(F^r \AAC(R)) \ar@/_2cm/[dddd]_-{\wr} \ar[d] \\
 K_{\partial}(\AAC(R))^{[u,v]} \ar@/^1cm/[rrrr]^{\gamma^{\rm CK}_{[u,v]}}\ar[r]^-{\tilde{\beta}}_-{\sim} \ar[d] & ( \eta_{\mu} \widetilde{K}_{\Gamma}(\A_{\rm cris}(R)))^{[u,v]} \ar[r]_-{\sim}^-{\eqref{R-Rinfty}} \ar[d] &  ( \eta_{\mu}\widetilde{K}_{\Gamma}( \A_{\rm cris}(R_{\infty})))^{[u,v]} \ar[d] & \ar[l]^-{\sim}_-{\eqref{continue et Acris}} \eta_{\mu} (K_{\Gamma}(\A^+_{R_{\infty}})^{[u,v]})\ar[d]^{\wr} \ar[dl]^{\sim} \ar[r]^{\widetilde{\mu}_H}_{\sim}&  \eta_{\mu} 
( C_{G}(\A^+_{\overline{R}})^{[u,v]})\ar[dl]^{\sim}\\
 K_{\partial}(\A_R^{[u,v]}) \ar[r]^-{\tilde{\beta}} & \eta_{\mu} K_{\Gamma}(\A_R^{[u,v]}) \ar[r]  &  \eta_{\mu} K_{\Gamma}(\A_{R_{\infty}}^{[u,v]})  \ar[r]^{\widetilde{\mu}_H}_{\sim} & \eta_{\mu} C_{G}(\A_{\overline{R}}^{[u,v]})\\
 K_{\partial}(F^r\A_R^{[u,v]}) \ar[r]^-{\beta'}_{\sim} \ar[u] &  K_{\Gamma}(\A_R^{[u,v]}(r)) \ar[r] \ar[u]^{t^r} &  K_{\Gamma}(\A_{R_{\infty}}^{[u,v]}(r)) \ar[u]^{t^r} \ar[r]^{\widetilde{\mu}_H}& 
 C_{G}(\A_{\overline{R}}^{[u,v]}(r))  \ar[u]^{t^r}  \\
 K_{\varphi, \partial}(F^r\A_R^{[u,v]}) \ar[r]^-{\beta'}_-{\sim} \ar[u] &  K_{\varphi, \Gamma}(\A_R^{[u,v]}(r)) \ar[r] \ar[u] &  K_{\varphi, \Gamma}(\A_{R_{\infty}}^{[u,v]}(r))  \ar[u]\ar[r]^{\widetilde{\mu}_H} & 
C_{G, \varphi}(\A_{\overline{R}}^{[u,v]}(r))  \ar[u] \\\
 & K_{\varphi, \Gamma}(\A_R^{(0,v]+}(r)) \ar[u]_{\rotatebox{90}{$\sim$}}^{\eqref{Av to Auv1}} \ar[d]^{\rotatebox{90}{$\sim$}}_{\eqref{Ar to Av}} \ar[r] & K_{\varphi,\Gamma}(\A_{R_{\infty}}^{(0,v]+}(r)) \ar[u] \ar[d] \ar[r]^{\widetilde{\mu}_H}& 
 C_{G, \varphi}(\A_{\overline{R}}^{(0,v]+}(r)) \ar[d]^{\wr}\ar[u]^{\wr}\\
 & K_{\varphi, \Gamma}(\A_R)(r)) \ar[r]_{\sim}^{\eqref{muinfini}} & K_{\varphi, \Gamma}(\A_{R_{\infty}}(r)) \ar[r]_{\sim}^{\eqref{muH}} & C_{G, \varphi}(\A_{\overline{R}}(r)) & C_G(\Z_p(r)) \ar[l]^{\sim}\ar[lu]^{\sim}\ar[luu]_{\sim}\ar[uuuuu]^{f_r}
}
\end{footnotesize} 
 \end{equation}
 
 \vspace{0.5cm}
Tous les complexes sont tronqu\'es par $\tau_{\leq r}$. L'exposant $[u,v]$ sur la deuxi\`eme ligne d\'esigne $(-) \widehat{\otimes}^L_{\A_{\text{inf}}} \A^{[u,v]}$. On pose $K_{\Gamma}(-):={\rm Kos}(\Gamma,-)$, $K_{\partial}(-):=\Kos(\partial,-)$ et $C_G(-)$ d\'esigne le complexe de cocha\^ines continues de $G$. Les notations $K_{\varphi, \Gamma}$ et $C_{G, \varphi}(-)$ sont celles de \ref{comparaison FM}. On \'ecrit $\eta_{\mu}\widetilde{K}_{\Gamma}(\A_{\rm cris}(R)):=\varinjlim \eta_{\mu} K_{\Gamma}(\A_{\rm cris}^{(m)}(R))$ et $
\eta_{\mu}\widetilde{K}_{\Gamma}(\A_{\rm cris}(R_{\infty})):=\varinjlim \eta_{\mu} K_{\Gamma}(\A_{\rm cris}^{(m)}(R_{\infty}))$. Enfin, les morphismes $\tilde{\mu}_H$ sont ceux induits par les morphismes 
\[ \Kos(\Gamma_R, \A^{?}_{R_{\infty}}) \to C(\Gamma_R, \A^{?}_{R_{\infty}}) \to C(G_R, \A^{?}_{\overline{R}}) \]
pour $? \in \{ (0,v]+, [u,v] \}$, o\`u la deuxi\`eme fl\`eche est le morphisme de bord du recouvrement de Galois $R_{\infty}$ de $R$. 

L'application $\beta'$ est donn\'ee par la compos\'ee :
{\small \[\hspace{-1.5cm}\tau_{\le r} \Kos( \partial, \varphi, F^r\A_R^{[u,v]}) \xrightarrow{t^{\bullet}} \tau_{\le r} \Kos(\Lie \Gamma, \varphi , F^r\A_R^{[u,v]} )   \xleftarrow{t^r} \tau_{\le r} \Kos(\Lie \Gamma, \varphi, \A_R^{[u,v]}(r)) \xleftarrow{\beta} \tau_{\le r} \Kos(\Gamma,\varphi,\A_R^{[u,v]}(r)) \] } 
 o\`u \[ \beta :  \A_R^{[u,v]}(r)^{J_j} \to  \A_R^{[u,v]}(r)^{J_j}, (a_{i_1 \dots i_j}) \mapsto (\beta_{i_j} \dots \beta_{i_1} (a_{i_1 \dots i_j})) \text{ avec }
  \beta_{k}:= \sum_{n \ge 1} \frac{(-1)^{n-1}}{n} \tau_k^{n-1} \] 
 En particulier, sur le premier terme, $\beta'$ induit : 
 {\small \begin{equation} \label{cat-jap} \tau_{\le r} \Kos(\partial, F^r\A_R^{[u,v]})  \xrightarrow{t^{\bullet}} \tau_{\le r} \Kos(\Lie \Gamma, F^r\A_R^{[u,v]} )   \xleftarrow{t^r} \tau_{\le r} \Kos(\Lie \Gamma, \A_R^{[u,v]}(r))  \\ \xleftarrow{\beta} \tau_{\le r} \Kos(\Gamma,\A_R^{[u,v]}(r)). \end{equation}}
Les deux premi\`eres fl\`eches sont des $p^{c(r)}$-quasi-isomorphisme (pour certaines constantes $c(r)$) et l'application $\beta$ est un isomorphisme. 

L'application $\widetilde{\beta}^{(m)} : (\A_{\rm cris}^{(m)}(R))^{J_i} \to (\A_{\rm cris}^{(m)}(R))^{J_i}$ est donn\'ee par les 
$\widetilde{\beta_k}= \sum_{n \ge 1} \frac{t^n}{n!} \partial^{n-1}_{k}. $ 
Soit $(a_{i_1 \dots i_j})$ dans $(F^{r-j} \A_R^{[u,v]})^{J_j}$, on a 
\[ t^r \beta'((a_{i_1 \dots i_j}))= (t^r \beta^{-1}_{i_1} \dots \beta^{-1}_{i_j} (t^{j-r}a_{i_1 \dots i_j}))=((\beta_{i_1}^{-1}t)\dots (\beta_{i_j}^{-1}t)( a_{i_1 \dots i_j})) \]
et $(\beta_{i_k}^{-1}t)= (\sum_{n \ge 0} b_n \tau_{i_k}^n) t$ o\`u les $b_n$ sont les coefficients de la s\'erie $\frac{X}{\log(1+X)}$. 
Comme \[ \widetilde{\beta}_{i_k}= \sum_{n \ge 1} \frac{t^n}{n!} \partial_{i_k}^{n-1} = (\sum_{n \ge 1} \frac{1}{n!} (t \partial_{i_k})^{n-1})t \] avec $(t \partial_{i_k})= \log(\gamma_{i_k})$, on obtient 
$ \beta_{i_k}^{-1}t=\widetilde{\beta}_{i_k} $
et on en d\'eduit que le carr\'e
\begin{equation}
\label{HstBee}
\xymatrix{ 
K_{\partial}(\A_R^{[u,v]}) \ar[r]^-{\tilde{\beta}} & \eta_{\mu} K_{\Gamma}(\A_R^{[u,v]}) \\
K_{\partial}(F^r\A_R^{[u,v]}) \ar[r]^-{\beta'} \ar[u] &  K_{\Gamma}(\A_R^{[u,v]}(r)) \ar[u]^{t^r}} 
\end{equation}
commute, o\`u $\beta'$ est l'application donn\'ee par \eqref{cat-jap}. On obtient finalement que le diagramme \eqref{CK-CN} est $p^{c(r)}$-commutatif (pour un certain $c(r) \in \N$).

Montrons que le trap\`eze dans le coin sup\'erieur droit :
\begin{equation}
\label{paris3}
\xymatrix{
 \eta_{\mu} K_{\Gamma}(\A^+_{R_{\infty}})^{[u,v]}\ar[dr]^{\sim} \ar[d]^{\wr} \ar[r]^{\widetilde{\mu}_H}_{\sim} & \eta_{\mu} C_{G}(\A^+_{\overline{R}})^{[u,v]}\ar[d]^{\wr}\\
   \eta_{\mu}K_{\Gamma}(\A_{R_{\infty}}^{[u,v]})  \ar[r]^{\widetilde{\mu}_H}_{\sim}  &  \eta_{\mu} C_{G}(\A_{\overline{R}}^{[u,v]}) 
}
\end{equation}
est constitu\'e de quasi-isomorphismes. 

La fl\`eche verticale de gauche est un isomorphisme de complexes. La fl\`eche horizontale du haut est un quasi-isomorphisme par \eqref{continue-proetale}. Montrons que la fl\`eche verticale de droite est un quasi-isomorphisme. Par d\'efinition, le complexe $C_{G}(\A^+_{\overline{R}})^{[u,v]}$ est \'egal \`a : 
\[ \holim_n ((R\Gamma(G, \A_{\overline{R}}^+) \otimes^L_{\AAinf} \A^{[u,v]}) \otimes^L_{\AAinf} \AAinf/(\xi,p)^n). \]
Mais $R\Gamma(G, \A_{\overline{R}}^+)$ est calcul\'e par le complexe born\'e $\Kos(\Gamma_R, \A_{R_{\infty}}^+)$ dont les termes sont plats sur $\AAinf$. On a donc 
\[ R\Gamma(G, \A_{\overline{R}}^+) \otimes^L_{\AAinf} \A^{[u,v]} \cong R\Gamma(G, \A_{\overline{R}}^+) \otimes_{\AAinf} \A^{[u,v]}. \]
Pour $n$ suffisamment grand (tel que $\xi^n \in p \A^{[u,v]}$), on a 
\[ {\rm Tor}^1_{\AAinf}(\A^{[u,v]}, \AAinf/(\xi,p)^n) \hookrightarrow {\rm Tor}^1_{\AAinf}(\A^{[u,v]}, \AAinf/p)=0 \] 
et donc 
\[ (R\Gamma(G, \A_{\overline{R}}^+) \otimes^L_{\AAinf} \A^{[u,v]}) \otimes^L_{\AAinf} \AAinf/(\xi,p)^n \cong R\Gamma(G, \A_{\overline{R}}^+) \otimes_{\AAinf} \A^{[u,v]}/(\xi,p)^n. \]
Les termes du complexes ci-dessus sont donn\'es par 
\[ {\rm Map}_{\rm cont}(G^m,\A^+_{\overline{R}}) \otimes_{\A_{\rm inf}} \A^{[u,v]}/(p, \xi)^n \] 
qui est isomorphe \`a (car $G$ est profini) :
\[ {\rm Map}_{\rm cont}(G^m,\A^+_{\overline{R}} \otimes_{\A_{\rm inf}}\A^{[u,v]}/(p, \xi)^n). \]
En prenant la limite sur $n$, on obtient finalement le quasi-isomorphisme :
\[ C_{G}(\A^+_{\overline{R}})^{[u,v]} \xrightarrow{\sim} C_{G}(\A_{\overline{R}}^{[u,v]}). \]

On en d\'eduit que les autres fl\`eches de (\ref{paris3}) sont elles aussi des quasi-isomorphismes. Finalement, une chasse au diagramme (de (\ref{CK-CN})) montre que \eqref{entier1} et donc (\ref{paris2}) commute
 dans le cas o\`u $\XF:= \Spf(R)$. Il reste \`a voir que le r\'esultat reste vrai pour $\XF$ global. 
 
(iii) {\em Cas global.}  
 Dans un premier temps, on suppose toujours $\XF= \Spf(R)$. 
Si $(\Sigma, \Lambda)$ sont tels qu'on ait une immersion \eqref{im fermee}, en d\'efinissant $R_{\Sigma, \Lambda}^{PD}$, $R_{\infty, \Sigma, \Lambda}$, etc ... comme pr\'ec\'edemment, on peut r\'e\'ecrire le diagramme \eqref{CK-CN} en utilisant les versions $(\Sigma, \Lambda)$ des anneaux. On prend ensuite la limite sur l'ensemble des $(\Sigma, \Lambda)$ et on obtient un diagramme commutatif fonctoriel : 
{\footnotesize \[\hspace{-1.5cm}\xymatrix{
\varinjlim_{\Sigma, \Lambda} \tau_{\le r}(\Kos(\partial, R_{\Sigma, \Lambda}^{PD})_{\Q} \widehat{\otimes}^L_{\A_{\rm cris}[\tfrac{1}{p}]} \A^{[u,v]}[\frac{1}{p},\frac{1}{\mu}]) \ar[rr]^-{\sim} & &  \varinjlim_{\Sigma, \Lambda} \tau_{\le r} (\eta_{\mu} 
C_G(\A^+_{\overline{R}_{ \Sigma, \Lambda}})_{\Q} \widehat{\otimes}^L_{\A_{\rm inf}[\tfrac{1}{p}]}  \A^{[u,v]}[\frac{1}{p},\frac{1}{\mu}]) \\
\varinjlim_{\Sigma, \Lambda} \tau_{\le r}\Kos(\partial, \varphi, F^rR_{\Sigma, \Lambda}^{PD})_{\Q} \ar[r]^-{\sim} \ar[u] & \tau_{\le r} 
C_G(\Kos(\varphi, \A_{\overline{R}}(r)))_{\Q} &  \tau_{\le r} C_G(\Z_p(r))_{\Q} \ar[l]_-{\sim} \ar[u]_{t^r}
}
\]}
Soit $\XF$ global. En consid\'erant des hyper-recouvrements affines $\mathfrak{U}^{\bullet} \to \XF$ comme pr\'ec\'edemment et en remarquant que les deux complexes de la ligne du haut du diagramme ci-dessus calculent respectivement
\[ \tau_{\leq r}R\Gamma_{\mathrm{cris}}( (\UF^k_{\OC_C}/p)/ \AAC)_{\Q} \widehat{\otimes}^L _{\A_{\rm cris}[\frac{1}{p}]}\A^{[u,v]}[\frac{1}{p}, \frac{1}{\mu}] \text{ et } \tau_{\leq r}R\Gamma_{\text{\'et}}(\UF^k_{\OC_C}, A \Omega_{\XF_{\OC_C}})_{\Q} \widehat{\otimes}^L_{\A_{\mathrm{inf}}[\frac{1}{p}]} \A^{[u,v]}[\frac{1}{p}, \frac{1}{\mu}]  \]
on en d\'eduit que le diagramme suivant commute : 
{\footnotesize \[ \xymatrix{
\tau_{\le r} R \Gamma_{\rm{cris}}( \XF )_{\Q}^{[u,v]} \ar[r]^-{\sim} & \tau_{\le r} \hocolim R \Gamma_{\rm{cris}}( (\mathfrak{U}^{\bullet})_{\Q}^{[u,v]}  \ar[r]^-{\sim} & \tau_{\le r} \hocolim R \Gamma_{\text{\'et}}(A \Omega_{\mathfrak{U}^{\bullet}})_{\Q}^{[u,v]}    &\ar[l]_-{\sim} \tau_{\le r} R \Gamma_{\text{\'et}}(A \Omega_{\XF})_{\Q}^{[u,v]}  \\
\tau_{\le r} R \Gamma_{\rm syn}(\XF_{\OC_C},r)_{\Q} \ar[r]^-{\sim} \ar[u] & \tau_{\le r} \hocolim R \Gamma_{\rm syn}(\mathfrak{U}^{\bullet}_{\OC_C},r)_{\Q} \ar[u]  \ar[r]^-{\sim} & \tau_{\le r} \hocolim R \Gamma_{\text{\'et}}(\mathfrak{U}^{\bullet}_C, \Q_p(r)) \ar[u]^{t^r} & \ar[l]_-{\sim} \tau_{\le r}R \Gamma_{\text{\'et}}(\XF_C, \Q_p(r)) \ar[u]^{t^r}
}\] }
o\`u on a not\'e $R\Gamma_{\cris}(-):=R\Gamma_{\mathrm{cris}}( ((-)_{\OC_C}/p)/ \AAC)$ et $R\Gamma_{\text{\'et}}(A \Omega_{(-)_{\OC_C}}):=R\Gamma_{\text{\'et}}((-)_{\OC_C}, A \Omega_{(-)_{\OC_C}})$.
On obtient le r\'esultat du lemme.

 \end{proof}

\bibliographystyle{plain}
\bibliography{biblio} 
\end{document}